\documentclass[fleqn,a4paper]{cas-sc}
\usepackage[numbers]{natbib}

\usepackage{amsmath,amssymb,epsfig,float,color,mathtools}
\usepackage{anyfontsize}

\usepackage{multirow}
\usepackage{graphicx}
\usepackage{array}

\usepackage{booktabs}
\usepackage{siunitx}
\usepackage{pdflscape}
\usepackage{verbatim} 

\usepackage{scalerel,subfigure,booktabs,acronym}

\usepackage{siunitx}

\usepackage{amsthm}

\definecolor{lightblue}{rgb}{0.22,0.45,0.70}
\definecolor{mygray}{rgb}{0.7,0.7,0.7}

\usepackage{enumitem}
\newcommand\cero{\boldsymbol{0}}
\newcommand{\norm}[1]{\left\|#1\right\|}

\newcommand\vdiv{\mathop{\mathrm{div}}\nolimits}

\newcommand\tr{\mathop{\mathrm{tr}}\nolimits}

\newcommand{\bG}{\mathbf{G}}

\newcommand{\bI}{\mathbf{I}}
\newcommand{\Om}{\Omega}

\newcommand{\bu}{\boldsymbol{u}}

\newcommand{\bg}{\boldsymbol{g}}

\newcommand{\bn}{\boldsymbol{n}}
\newcommand{\bt}{\boldsymbol{t}}
\newcommand{\bv}{\boldsymbol{v}}
\newcommand{\bw}{\boldsymbol{w}}
\newcommand{\bx}{\boldsymbol{x}}

\newcommand{\br}{\mathbf{r}}

\newcommand{\ff}{\boldsymbol{f}}

\newcommand{\bnabla}{\boldsymbol{\nabla}}

\newcommand{\bchi}{\boldsymbol{\chi}}

\newcommand{\bPi}{\boldsymbol{\Pi}}

\newcommand{\0}{\boldsymbol{0}}

\newcommand\cA{\mathcal{A}}
\newcommand\cK{\mathcal{K}}

\newcommand\cD{\mathcal{D}}

\newcommand\cT{\mathcal{T}}
\newcommand\cE{\mathcal{E}}
\newcommand\cF{\mathcal{F}}

\newcommand\cM{\mathcal{M}}

\newcommand\cS{\mathcal{S}}

\newcommand\cP{\mathcal{P}}

\newcommand\bP{\mathbf{P}}
\newcommand\bM{\mathbf{M}}
\newcommand\bA{\mathbf{A}}
\newcommand\bB{\mathbf{B}}
\newcommand\bC{\mathbf{C}}

\newcommand\bV{\mathbf{V}}
\newcommand\tV{\widetilde{\bV}}

\newcommand\bbP{\mathbb{P}}
\newcommand\bbM{\mathbb{M}}
\newcommand\bH{\mathbf{H}}
\newcommand\bL{\mathbf{L}}
\newcommand\rL{\mathrm{L}}
\newcommand\rW{\mathrm{W}}

\newcommand\bbX{\mathbb{X}}

\usepackage{bbold}
\newcommand\bbPi{\mathbb{\Pi}}

\newcommand\rH{\mathrm{H}}
\newcommand\rP{\mathrm{P}}
\newcommand\rQ{\mathrm{Q}}
\newcommand\rR{\mathrm{R}}

\newcommand\rM{\mathrm{M}}
\newcommand\rB{\mathrm{B}}
\newcommand\bsigma{\boldsymbol{\sigma}}

\newcommand\bDelta{\boldsymbol{\Delta}}

\def\XXint#1#2#3{{\setbox0=\hbox{$#1{#2#3}{\int}$ }
		\vcenter{\hbox{$#2#3$ }}\kern-.6\wd0}}

\allowdisplaybreaks

\acrodef{pde}[PDE]{partial differential equation} 
\acrodefplural{pde}[PDEs]{partial differential equations} 
\acrodef{dof}[DOF]{degree of freedom} 
\acrodefplural{dof}[DOFs]{degrees of freedom} 
\acrodef{fe}[FE]{finite element} 
\acrodefplural{fe}[FEs]{finite elements} 
\acrodef{vem}[VEM]{virtual element method}
\acrodefplural{vem}[VEMs]{virtual element methods}
\acrodef{snm}[SNM]{silicon nanopore membrane}

\newtheorem{lemma}{Lemma}[section]

\newtheorem{proposition}{Proposition}[section]
\newtheorem{theorem}{Theorem}[section]
\newtheorem{corollary}{Corollary}[section]
\numberwithin{equation}{section}
\numberwithin{figure}{section}
\numberwithin{table}{section}

\begin{document}
\let\WriteBookmarks\relax
\def\floatpagepagefraction{1}
\def\textpagefraction{.001}
\shorttitle{VEM for Stokes/Biot--Kirchhoff bulk--surface models}
\title[mode = title]{Analysis and virtual element discretisation of a Stokes/Biot--Kirchhoff bulk--surface model}

\shortauthors{Dassi, Khot, Rubiano \& Ruiz-Baier}

\author[1]{Franco Dassi}[orcid=0000-0001-5590-3651]
\ead{franco.dassi@unimib.it}

\author[2,3]{Rekha Khot}[orcid=0009-0005-3230-2649]
\ead{Rekha.Khot@inria.fr}

\author[4]{Andr\'es E. Rubiano}[orcid=0000-0002-5557-4963]
\ead{Andres.RubianoMartinez@monash.edu}\cormark[1]

\author[4,5]{Ricardo Ruiz-Baier}[orcid=0000-0003-3144-5822]
\ead{Ricardo.RuizBaier@monash.edu}

\affiliation[1]{organization={Dipartimento di Matematica e Applicazioni, Università degli studi di Milano Bicocca}, addressline={Via Roberto Cozzi 55}, postcode={20125}, city={Milano}, country={Italy}}
\affiliation[2]{organization={SERENA Project-Team, Centre Inria de Paris}, postcode={F-75647}, city={Paris}, country={France}}
\affiliation[3]{organization={CERMICS, ENPC, Institut Polytechnique de Paris}, postcode={F-77455}, city={Marne-la-Vallée cedex 2}, country={France}}
\affiliation[4]{organization={School of Mathematics, Monash University},     addressline={9 Rainforest Walk}, postcode={3800},  city={Melbourne}, state={Victoria}, country={Australia}}
\affiliation[5]{organization={Universidad Adventista de Chile}, addressline={Casilla 7-D}, city={Chill\'an}, country={Chile}}
\cortext[cor1]{Corresponding author.}

\begin{abstract} 
We analyse a coupled 3D-2D model with a free fluid governed by  Stokes flow in the bulk and a poroelastic plate described by the Biot-Kirchhoff equations on the surface. Assuming the form of a double perturbed saddle-point problem, the unique solvability of the continuous formulation is proved using Fredholm's theory for compact operators and the Babu\v{s}ka--Brezzi approach for saddle-point problems with penalty. We propose a stable  virtual element method, establishing a discrete inf-sup condition under a small mesh assumption through a Fortin interpolant that requires only $\bH^1$-regularity for the Stokes problem. We show the well-posedness of the monolithic discrete formulation and introduce an equivalent fixed-point approach employed at the implementation level. The optimal convergence of the method in the energy norm is proved theoretically and is also confirmed numerically via computational experiments. We demonstrate an application of the model and the proposed scheme in the simulation of immune isolation using encapsulation with silicon nanopore membranes.
\end{abstract}


\maketitle

\section{Introduction}
Bulk--surface interaction problems, wherein coupled physical phenomena occur both within a three-di\-men\-sional domain and on a lower-dimensional manifold (such as a boundary or an embedded interface), arise naturally in a wide range of geophysical, biomedical, and industrial applications. Typical examples include subsurface fluid transport through fractured porous media \cite{tran2021effect}, nutrient exchange across biological membranes \cite{zhang2025nanoconfined}, and the operation of selective barriers such as semi-permeable filtration membranes \cite{dutt2023ultrathin}. These systems are often characterised by complex multiscale dynamics, where interfacial processes strongly influence bulk behaviour and vice versa.

This work is motivated by the modelling and simulation of an incompressible fluid evolving in a bulk domain and interacting with a poromechanical thin structure, specifically a deformable, porous interface. Such settings are prototypical in the context of immune isolation devices and filtration technologies \cite{Song2016}, where mechanical deformation and fluid exchange through compliant membranes are tightly coupled (see also, e.g., \cite{ak2024integration}). 

The geometrical configuration and the relevant operating regimes in numerous applications suggest an effective model where the interface is treated as a Kirchhoff--Biot poroelastic plate \cite{iliev16}, thereby reducing the complexity of the structural component while retaining its key physical features. In the bulk, one considers Stokes flow under mixed boundary conditions to account for viscous incompressible fluid dynamics. The coupling at the fluid--structure interface governs both momentum and mass transfer, and it requires a careful analytical and numerical treatment.

A growing body of literature has addressed numerical methods and theoretical analysis for general mixed-dimensional and bulk--surface coupled systems. Examples include models for surface-bound receptor dynamics in cellular biology \cite{elliott17}, surface transport coupled with bulk Darcy flow in ecological models \cite{galiano11}, and hybrid methods for fractured media \cite{alboin02,chernyshenko18}. Advanced discretisation strategies such as trace finite elements \cite{gross15} and cut finite elements \cite{hansbo16} have been developed to handle the geometric complexity of lower-dimensional manifolds embedded in higher-dimensional domains. We also mention partitioned methods, such as   \cite{FERNANDEZ2013}, that treat the fluid and solid domains independently at each time step, coupling them through explicit interface conditions. 

More directly related to the current study is the work in \cite{bociu21}, where the authors analyse the coupling between a free fluid and both thin and thick poroelastic layers using an asymptotically consistent multilayer model. Their analysis employs Galerkin discretisation and compactness arguments to establish well-posedness. However, our setting differs significantly: we focus solely on the interaction between a free Stokes fluid and a thin poroelastic plate, without including a thick poroelastic subdomain. The poroelastic interface is modelled using a reduced Kirchhoff--Biot theory, as introduced in \cite{iliev16}, and discretised following recent developments in the \ac{vem} for such structures \cite{khot23}. 

Our primary interest lies in the numerical analysis of this bulk--surface coupled system using a \ac{vem}. For coupled diffusion problems on bulk--surface geometries, \acp{vem} have shown great promise in recent works for 2D--1D and 3D--2D configurations \cite{frittelli21,frittelli23}. In the present paper, we build upon and extend these developments to treat a fluid--structure interaction problem coupling Stokes flow and a Kirchhoff--Biot poroelastic plate. This approach leverages the advantages of divergence-free virtual elements and the ability of \ac{vem} to preserve discrete complex structures for Stokes problem in three dimensions (see \cite{beirao20}), as well as the relatively straightforward construction of conforming virtual elements for fourth-order problems  with fewer \acp{dof} even in two dimensions, especially when compared to, for instance, Argyris finite elements \cite{brenner2008mathematical}.  The coupling across the interface introduces non-trivial challenges in terms of both continuous and discrete stability, which we address through careful analytical constructions.

The principal contributions of this paper can be summarised as follows: a rigorous continuous analysis of the coupled bulk--surface system formulated as a double saddle-point problem, using perturbed saddle-point theory and Fredholm alternative arguments, the design of a compatible \ac{vem} for the coupled Stokes/Kirchhoff--Biot system on general polygonal and polyhedral meshes; the construction and related estimates of a novel Fortin interpolation operator with $\bH^1$-regularity and a commuting diagram property, tailored to enforce an extended discrete inf-sup condition involving the bulk--surface coupling; a detailed derivation of optimal a priori error estimates in the energy norm under a small mesh assumption; an open-source implementation of the method in the \texttt{VEM++} library \cite{dassi2023vem++}, using a splitting scheme that is shown to be equivalent (under suitable assumptions) to the monolithic formulation; and finally, numerical experiments that confirm the theoretical convergence rates, and that demonstrate the effectiveness of the proposed approach in simulating fluid--structure interaction relevant to immune isolation via \ac{snm} devices.

\paragraph{Plan of the paper} The contents of the remainder of this work are organised as follows. The statement of the coupled bulk--surface model, the domain configuration, and the weak formulation of the problem are presented in Section~\ref{sec:eq}. Fredholm theorems together with the abstract theory for perturbed saddle-point problems are the main tools used in Section~\ref{sec:wellp} to show that the continuous problem is well-posed. In Section~\ref{sec:vem}, we state the conforming VE spaces, provide appropriate \acp{dof} and introduce suitable projection maps that comprise the definition of the \ac{vem} formulation. This discrete problem is proven to be well-posed in Section~\ref{sec:wellp-h}, and the analysis of convergence is detailed in Section~\ref{sec:error}. The implementation of the method and the corresponding splitting scheme are explained in Section~\ref{sec:implementation}. Finally, Section~\ref{sec:results} presents representative numerical examples that confirm the rates of convergence specified by the theoretical analysis and the applicability of the model.

\paragraph{Preliminaries} 
Let $\Omega$ be a bounded domain in $\mathbb{R}^3$ with boundary $\partial\Omega$ split disjointly between a smooth sub-boundary $\Gamma$ and a flat surface $\Sigma$ with outward pointing unit normal $\bn_\Sigma$. 
We denote by $\nabla$ the gradient in $\mathbb{R}^3$ and by $\nabla_\Sigma = \cP_\Sigma \nabla$ the tangent gradient on $\Sigma$, 
where $\cP_\Sigma$ denotes the projection of $\mathbb{R}^3$ onto the tangent plane of $\Sigma$, that is $\cP_\Sigma = \bI - \bn_\Sigma \otimes\bn_\Sigma$ (where $\bI$ is the identity), see, e.g., \cite{olshanskii18}. Other differential operators associated with the surface $\Sigma$ will be denoted with the subscript $\Sigma$, such as the divergence $\vdiv_\Sigma = \tr(\nabla_\Sigma)$, Laplacian $\vdiv_\Sigma(\nabla_\Sigma) = \Delta_\Sigma$ (Laplace--Beltrami), Hessian $\nabla_\Sigma^2$, and bi-Laplacian $\Delta^2_\Sigma$.  

We use standard notation (see, e.g., \cite{mclean2000}) and denote, for $s\geq 0$, by $\rH^s(\Omega)$ the usual Hilbertian Sobolev space of scalar functions with domain $\Omega$, and denote by $\bH^s(\Omega)$ their vector counterpart. The norm of $\rH^s(\Omega)$ is denoted $\norm{\cdot}_{s,\Omega}$ and the corresponding semi-norm $|\cdot|_{s,\Omega}$. We also use the convention  $\rH^0(\Omega):=\rL^2(\Omega)$ and write $(\cdot, \cdot)_\Omega$ to denote the inner product in $\rL^2(\Omega)$ (similarly for the vector counterpart).

Throughout the paper, we will use the symbol $\lesssim$ to denote less or equal up to a constant that does not depend on the mesh size. 
\section{Governing equations and weak formulation}\label{sec:eq}
We assume that the domain $\Omega$ is filled with a viscous incompressible fluid whose dynamics is governed by Stokes' equations written in terms of bulk velocity $\bu:\Omega \to \mathbb{R}^3$ and bulk pressure $p:\Omega \to \mathbb{R}$
\[ \rho_f \partial_t \bu - \mu \bDelta \bu + \nabla p = \ff \quad \text{and} \quad \vdiv \bu = 0 \qquad \text{in $\Omega$},\]
where $\rho_f$ is the fluid density, $\mu$ is the fluid viscosity, and $\ff$ is the external body force. The boundary $\Gamma$ is the wall of the container separated into $\Gamma^{\bu}$ and $\Gamma^{\bsigma}$, on which we consider mixed no-slip velocities and zero normal stress, and $\Sigma$ represents a flat poroelastic plate in contact with the fluid. On this surface the dynamics are governed by the Biot (or Biot--Kirchhoff)  equations stated in terms of normal deflections of the plate $w:\Sigma \to \mathbb{R}$ and the first moment of the fluid  pressure head in the interstitial plate $\varphi:\Sigma \to \mathbb{R}$. By $\partial\Sigma$ we denote the border of $\Sigma$ and by $\bn_{\partial\Sigma}$ we denote the unit normal pointing outwards from $\partial\Sigma$ and lying on the tangent plane of $\Sigma$. 

In the form of Biot--Kirchhoff equations considered herein, the first assumption is that the deformations of the solid phase are consistent with the Kirchhoff--Love hypothesis, and so plate filaments that were originally perpendicular to the middle surface remain orthogonal to the  deflected centred surface. Secondly, as in \cite{iliev16}, we suppose that the apparent fluid pressure in that system has the physical meaning of the first moment of the pressure across the thickness of the plate, and that the filtration in the poroelastic plate (through Darcy's law) occurs in the tangent plane to $\Sigma$. This is different from the works \cite{bociu21,gurvich22,marciniak15} where the filtration occurs predominantly in the normal direction and then a 2.5D type of model is required for the poroelastic plate and the pressure moment is not used as an unknown. 
Bearing in mind these considerations, we are left with the following set of equations for the plate  (see, e.g., \cite{iliev16,khot23}) 
\[\rho_p\partial_{tt} w + D \Delta^2_\Sigma w + \alpha\Delta_\Sigma \varphi  = m -\bsigma\bn_\Sigma \cdot \bn_\Sigma\quad \text{and}\quad 
\partial_t c_0 \varphi - {\alpha} \Delta_\Sigma \partial_t w - \kappa\Delta_\Sigma \varphi  = g  \qquad \text{in $\Sigma$},\]
where $\rho_p>0$ is the inertial parameter (plate density), $D>0$ denotes the elastic stiffness (flexural rigidity) of the plate, $\alpha>0$ is the rescaled Biot--Willis coefficient, $c_0\geq 0$ is the storativity of the fluid-solid matrix (the net compressibility of constituents), the term $c_0\varphi - \alpha \Delta_\Sigma w$ is the total amount of fluid in the plate, $\kappa>0$ is the plate permeability rescaled with fluid viscosity, $m:\Sigma \to \mathbb{R}$ is the distributed load on the plate, and $g:\Sigma \to \mathbb{R}$ is a source/sink of fluid. Note that on the right-hand side of the deflection equation the force balance also has a contribution coming from the normal stress of the fluid.

As the present work focuses on the spatial discretisation using \acp{vem}, we discard the time dependence of the problem by applying a semi-discretisation in time with constant time step $\tau>0$. Making abuse of notation  
regarding the load and source terms (which will now contain contributions from the solutions at the previous time steps), 
in summary, we have the following bulk--surface coupled system 
\begin{subequations}\label{eq:coupled}
\begin{align}
\frac{\rho_f}{\tau}\bu - \mu \bDelta \bu + \nabla p & =  \ff & \quad \text{in $\Omega$}, \label{eq:stokes1}\\
\vdiv\bu & =0  & \quad \text{in $\Omega$}, \label{eq:stokes2}\\
\frac{c_0}{\tau} \varphi - \frac{\alpha}{\tau} \Delta_\Sigma w - \kappa\Delta_\Sigma \varphi & = g  & \quad \text{in $\Sigma$}, \label{eq:plate1}\\
\frac{\rho_p}{\tau^2} w + D \Delta^2_\Sigma w + \alpha\Delta_\Sigma \varphi & = m -\bsigma\bn_\Sigma \cdot \bn_\Sigma & \quad \text{in $\Sigma$},\label{eq:plate2}
\end{align}
equipped with the following boundary conditions 
\begin{align}
 \bu & = \cero & \text{on $\Gamma^{\bu}$ },\label{eq:bc1a}\\
 \bsigma\cdot\bn_\Gamma & = \cero & \text{on $\Gamma^{\bsigma}$},\label{eq:bc1b}\\
 \varphi &= 0 & \text{on $\partial\Sigma$,}\label{eq:bc2} \\
 w & = \nabla_\Sigma w \cdot \bn_{\partial\Sigma} = 0 & \text{on $\partial\Sigma$.}\label{eq:bc3}
\end{align}
Since stress can be exerted from the fluid domain onto the poroelastic plate, we also consider the following set of kinematic and dynamic transmission conditions accounting for the continuity of normal velocities and the Beavers--Joseph--Saffman--Jones  interfacial conditions for normal and tangential stress (see, e.g., \cite{showalter05,taffetani21}) 
\begin{align}
 \bu \cdot \bn_\Sigma & =   \frac{1}{\tau} w - \kappa \nabla_\Sigma\varphi\cdot\bn_\Sigma =  \frac{1}{\tau} w & \text{on $\Sigma$},\label{eq:bjs1}\\
 - (\bsigma \bn_\Sigma)\cdot \bn_\Sigma & = \varphi & \text{on $\Sigma$},\label{eq:bjs2}\\
 - (\bsigma\bn_\Sigma) \times \bn_\Sigma & = \gamma (\bu-\frac{1}{\tau}w\bn_{\Sigma})\times\bn_\Sigma =  \gamma \bu\times\bn_\Sigma & \text{on $\Sigma$},\label{eq:bjs3}
\end{align}
\end{subequations}
where $\bsigma := \mu \bnabla \bu - p \bI$ denotes the Cauchy pseudo-stress tensor associated with the bulk fluid domain, $\gamma>0$ is the slip rate coefficient (tangential resistance) rescaled with the fluid viscosity and plate permeability. In \eqref{eq:bjs1} the fluid pressure moment flux vanishes because the plate gradient and $\bn_\Sigma$ are mutually orthogonal, and in \eqref{eq:bjs3} the plate deflections vanish because of the term $\bn_\Sigma \times \bn_\Sigma$, which turn \eqref{eq:bjs3} into the classical Beavers--Joseph--Saffman condition for tangential stress encountered in Stokes--Darcy type of models. We also stress  that, from \eqref{eq:bjs2}, we can recast the right-hand side forcing term on the plate simply as $m - \varphi$.

As we consider entities defined on surfaces, apart from the surface function spaces $\rL^2(\Sigma)$ and $\rH^m(\Sigma)$, we will also use the trace space $\rH^{1/2}(\Sigma)$ and its dual $\rH^{-1/2}(\Sigma)$ as well as their vector-valued counterparts (see, e.g., \cite{elliott17}). We recall that the trace operator from $\rH^1(\Omega)$ to $\rH^{1/2}(\Sigma)$ is bounded and surjective (cf. \cite{grisvard1980boundary}), and  that the space 
$\rH^1(\Sigma)$ is dense in $\rH^{1/2}(\Sigma)$. 

In addition, and in view of the boundary conditions, we define the following functional spaces for bulk velocity, bulk pressure, surface pressure moment, and surface deflection 
\begin{gather*}
 \bH^1_\star(\Omega) := \{ \bv \in \bH^1(\Omega): \bv = \cero \ \text{on $\Gamma^{\bu}$}\}, \qquad \rL^2(\Omega), 
 \\
 \rH^1_0(\Sigma) : = \{ \psi \in \rH^1(\Sigma): \psi = 0 \ \text{on $\partial\Sigma$}\}, \qquad 
  \rH^2_0(\Sigma) : = \{ w \in \rH^2(\Sigma): w = \nabla_\Sigma w \cdot \bn_{\partial\Sigma}  = 0 \ \text{on $\partial\Sigma$}\}, 
\end{gather*}
respectively, where the boundary values are understood in the sense of traces and we adopt the notation of using the same symbol for a function and its trace.

We endow these spaces with their natural norms 
\[ \bv\mapsto \|\bv\|_{1,\Omega}, \quad q\mapsto \|q\|_{0,\Omega},\quad \psi \mapsto \|\psi\|_{1,\Sigma}, \quad \zeta \mapsto \|\zeta\|_{2,\Sigma}, \]
and furthermore we denote the graph norm in $\rL^2(\Omega)\times \rH^1_0(\Sigma)$ as $\|(q,\psi)\|^2:= \|q\|^2_{0,\Omega} + \|\psi\|^2_{1,\Sigma}$. 
 
Note that the boundary conditions for the fluid velocity need to be compatible between the part of $\Gamma$ that meets with $\partial\Sigma$. For the zero fluid pressure moment condition assumed in \eqref{eq:bc2} we cannot have $\Gamma^{\bu}$ meeting $\partial\Sigma$ since on $\Gamma^{\bu}$ we do not prescribe fluid pressure  but rather fluid velocity. Then we need to assume a domain and boundary configuration as depicted in Figure~\ref{fig:sketch}. Should different boundary conditions be considered on $\partial\Sigma$, for example a no-flux condition for plate fluid pressure moment, then we require to exchange the fluid domain sub-boundaries $\Gamma^{\bu} \rightleftarrows \Gamma^{\bsigma}$.

\begin{figure}[t!]
    \centering
    \includegraphics[width=0.5\textwidth]{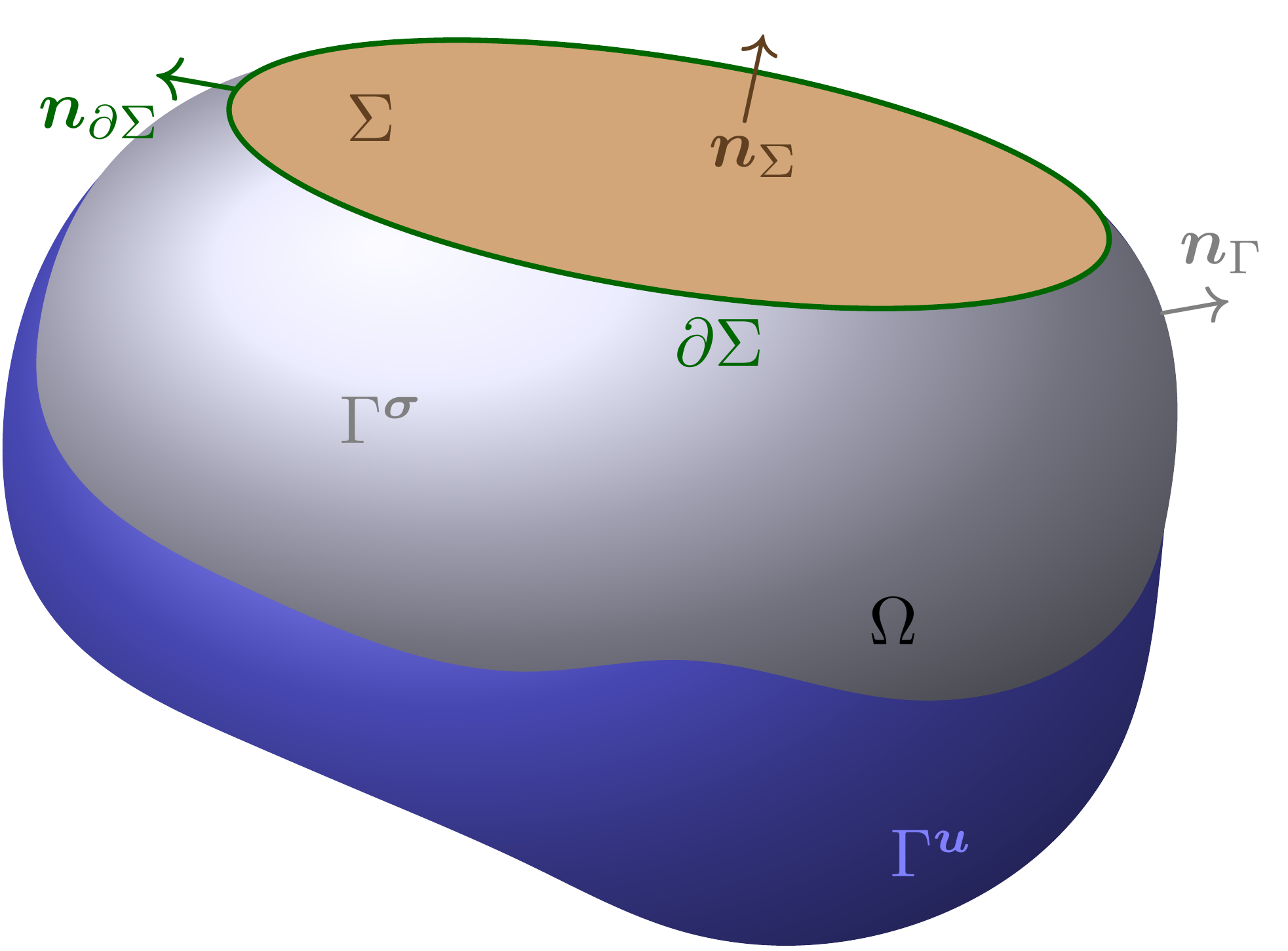}
    \caption{Sketch of the domain and boundary configuration.}
    \label{fig:sketch}
\end{figure}

A weak formulation is derived by testing equation \eqref{eq:stokes1} against $\bv\in \bH^1_\star(\Omega)$, integrating over $\Omega$, applying integration by parts and using the boundary conditions \eqref{eq:bc1a}-\eqref{eq:bc1b};  then testing \eqref{eq:stokes2} against $q\in\rL^2(\Omega)$ and integrating over $\Omega$, then testing the transmission condition \eqref{eq:bjs1} against $\psi \in \rH^1_0(\Sigma)$, testing \eqref{eq:plate1} also against $\psi \in \rH^1_0(\Sigma)$, integrating by parts using \eqref{eq:bc2} and negating that equation; and finally proceeding similarly for \eqref{eq:plate2} testing by $\zeta\in \rH^2_0(\Sigma)$ and rescaling that equation by $\frac{1}{\tau}$. We note that from the momentum balance, the remainder of integration by parts
 can be conveniently rewritten, owing to the splitting $\bsigma\bn_\Sigma = [(\bsigma\bn_\Sigma)\cdot \bn_\Sigma]\bn_\Sigma + [(\bsigma\bn_\Sigma)\times \bn_\Sigma]\times\bn_\Sigma$ and the transmission conditions \eqref{eq:bjs2}-\eqref{eq:bjs3}, as follows 
\[  I_\Sigma = -\int_\Sigma \bsigma \bn_\Sigma \cdot \bv  = \gamma\int_\Sigma (\bu\times\bn_\Sigma)\cdot(\bv\times \bn_\Sigma) + \int_\Sigma \varphi \bv\cdot \bn_\Sigma.\]
Putting all this together, we arrive at:
 Find $(\bu,p,\varphi,w)\in \bH^1_\star(\Omega)\times \rL^2(\Omega)\times \rH^1_0(\Sigma)\times \rH^2_0(\Sigma)$
such that 
\begin{subequations}\label{eq:weak}
\begin{align}
\frac{\rho_f}{\tau}\int_\Omega \bu\cdot \bv + \mu \int_\Omega \bnabla\bu:\bnabla\bv + \gamma \int_\Sigma (\bu\times \bn_\Sigma) \cdot (\bv\times \bn_\Sigma)& \nonumber\\ - \int_\Omega p\,\vdiv\bv + \int_\Sigma \varphi \bv\cdot\bn_\Sigma 
& = \int_\Omega \ff\cdot \bv \quad &&\forall \bv\in \bH^1_\star(\Omega),\label{eq:weak1}\\
- \int_\Omega q\,\vdiv\bu  &= 0 \quad &&\forall q \in \rL^2(\Omega),\label{eq:weak2}\\
\int_\Sigma \psi \bu\cdot\bn_\Sigma - \frac{1}{\tau}\int_\Sigma w\,\psi  &= 0 \quad &&\forall \psi\in \rH^1_{0}(\Sigma),\label{eq:weak3}\\
-\frac{c_0}{\tau}\int_\Sigma \varphi\psi - \frac{\alpha}{\tau}\int_\Sigma \nabla_\Sigma w \cdot \nabla_\Sigma \psi - \kappa\int_\Sigma \nabla_\Sigma\varphi\cdot\nabla_\Sigma\psi  &= -\int_\Sigma g\psi \quad &&\forall \psi \in \rH^1_{0}(\Sigma),\label{eq:weak4}\\
\frac{1}{\tau}\int_\Sigma \varphi \zeta -\frac{\alpha}{\tau} \int_\Sigma \nabla_\Sigma\varphi\cdot\nabla_\Sigma\zeta  +\frac{\rho_p}{\tau^3}\int_\Sigma w\zeta + \frac{D}{\tau} \int_\Sigma \nabla^2_\Sigma w: \nabla_\Sigma^2 \zeta &= \frac{1}{\tau}\int_\Sigma m\zeta \quad &&\forall \zeta \in \rH^2_{0}(\Sigma).\label{eq:weak5}
\end{align}\end{subequations}

Notice that the Sobolev  embeddings provide $\rH^1_0(\Sigma) \hookrightarrow \rH^{1/2}(\Sigma) \hookrightarrow \rL^2(\Sigma) \hookrightarrow \rH^{-1/2}(\Sigma) \hookrightarrow \rH^{-1}(\Sigma)$. Moreover, let us denote  by $\mathcal{R}_{\hat{s},\Sigma}$ ($\hat{s}\in \mathbb{R}$) the Riesz map between $\rH^{-\hat{s}}(\Sigma)$ and its dual $\rH^{\hat{s}}(\Sigma)$. Thus, for any  $\xi\in  \rH^1_0(\Sigma)$ and $\rho \in  \rH^{1/2}(\Sigma)$, the following duality pairings are equal to the  inner product in $\rL^2(\Sigma)$
\begin{equation}\label{eq:duality_equalities}
    \langle \xi,i_{-1/2,-1} \circ\mathcal{R}_{1/2,\Sigma}^{-1}(\rho) \rangle_{1,\Sigma} = \langle \rho,\mathcal{R}_{1/2,\Sigma}^{-1}\circ i_{1,1/2}(\xi) \rangle_{1/2,\Sigma} = \int_\Sigma \xi\,\rho, \quad \text{and} \quad \langle \xi,\mathcal{R}_{1,\Sigma}^{-1}(\xi) \rangle_{1,\Sigma}=\norm{\xi}_{1,\Sigma}^2,
\end{equation}
where $i_{-1/2,-1}$ (resp. $i_{1,1/2}$) denotes the Sobolev embedding from $\rH^{-1/2}(\Sigma)$ to $\rH^{-1}(\Sigma)$ (resp. $\rH^{1}(\Sigma)$ to $\rH^{1/2}(\Sigma)$) and the Riesz representation theorem is used in the second equality. Therefore, since $\bn_\Sigma$ is sufficiently regular so that $\bv\cdot\bn_\Sigma \in \rH^{1/2}(\Sigma)$, the last term in the left-hand side of  \eqref{eq:weak1} (and similarly, the first one in \eqref{eq:weak3}) is well-defined. 

Adding together equations \eqref{eq:weak3} and \eqref{eq:weak4} and defining the operators $\bA: \bH^1_\star(\Omega) \to [\bH^1_\star(\Omega)]'$, $\bB_1:\bH^1_\star(\Omega) \to [\rL^2(\Omega)\times\rH^1_0(\Sigma)]'$, 
$\bC_1:\rL^2(\Omega)\times\rH^1_0(\Sigma)\to [\rL^2(\Omega)\times\rH^1_0(\Sigma)]'$, 
$\bB_2,\bB_3: \rL^2(\Omega)\times\rH^1_0(\Sigma) \to [\rH^2_0(\Sigma)]'$, and 
$\bC_2:  \rH^2_0(\Sigma) \to [ \rH^2_0(\Sigma)]'$
through the following bilinear forms 
\begin{align*}
\langle \bA\bu,\bv\rangle = a(\bu,\bv) &= a^{0}(\bu,\bv) + a^{\nabla}(\bu,\bv) + a^{\Sigma}(\bu,\bv)\\
&:= \frac{\rho_f}{\tau}\int_\Omega \bu\cdot \bv + \mu \int_\Omega \bnabla\bu:\bnabla\bv + \gamma \int_\Sigma (\bu\times \bn_\Sigma) \cdot (\bv\times \bn_\Sigma), \\ 
\langle \bB_1\bv,(q,\psi)\rangle = b_1(\bv,(q,\psi)) &= b_1^{\vdiv}(\bv,(q,\psi)) + b_1^\Sigma(\bv,(q,\psi)) \\
&:= - \int_\Omega q\vdiv\bv+\int_\Sigma \psi \,\bv\cdot\bn_\Sigma, \\
\langle \bC_1(p,\varphi),(q,\psi)\rangle = c_1((p,\varphi),(q,\psi)) &= c_1^{0}((p,\varphi),(q,\psi)) + c_1^{\nabla}((p,\varphi),(q,\psi)) \\
&:= \frac{c_0}{\tau}\int_\Sigma \varphi\psi + \kappa\int_\Sigma \nabla_\Sigma\varphi\cdot\nabla_\Sigma\psi, \\
\langle \bB_2(q,\psi),\zeta\rangle = b_2((q,\psi),\zeta) &:= -\frac{\alpha}{\tau} \int_\Sigma \nabla_\Sigma\psi\cdot\nabla_\Sigma\zeta, \\
\langle \bB_3(q,\psi),\zeta\rangle = b_3((q,\psi),\zeta) &:= - \frac{1}{\tau}\int_\Sigma \zeta \psi,\\ 
\langle \bC_2 w,\zeta\rangle = c_2(w,\zeta) &= c_2^{0}(w,\zeta) + c_2^{\nabla^2}(w,\zeta)\\ &:= \frac{\rho_p}{\tau^3}\int_\Sigma w\zeta + \frac{D}{\tau} \int_\Sigma \nabla^2_\Sigma w: \nabla_\Sigma^2 \zeta, 
\end{align*}
we infer that the weak formulation of the coupled problem can be written in operator form (in the dual of the solution space) as follows 
\begin{equation}\label{eq:operator} \begin{pmatrix}
    \bA & \bB_1^* & \cero \\
    \bB_1 & - \bC_1 & \bB_2^* + \bB_3^* \\
    \cero & \bB_2\,-\bB_3  & \bC_2
\end{pmatrix}\begin{pmatrix}\bu \\ (p,\varphi) \\ w\end{pmatrix} = \begin{pmatrix}
 F \\ G \\ M    
\end{pmatrix} \quad \text{in} \quad [\bH^1_\star(\Omega)\times (\rL^2(\Omega)\times \rH^1_0(\Sigma))\times \rH^2_0(\Sigma)]', \end{equation}
where the linear functionals $F \in [\bH^1_\star(\Omega)]'$, $G \in [\rL^2(\Omega)\times \rH^1_0(\Sigma)]'$, and $M\in [\rH^2_0(\Sigma)]'$ are defined as 
\[ F(\bv):= \int_\Omega \ff\cdot \bv, \quad G((q,\psi)): = -\int_\Sigma g\psi, \quad M(\zeta): = \frac{1}{\tau}\int_\Sigma m\zeta , \]
for all $\bv\in \bH^1_\star(\Omega)$, $(q,\psi) \in \rL^2(\Omega)\times \rH^1_0(\Sigma)$, and $\zeta \in \rH^2_0(\Sigma)$. 
\section{Unique solvability of the continuous problem}\label{sec:wellp}
This section presents a theoretical analysis of the weak formulation in \eqref{eq:weak}, including detailed proofs establishing its well-posedness.
\subsection{Preliminaries} We start by stating key properties of the operators above. They follow directly from trace inequality (with continuity constant $C_T>0$ depending only on $\Sigma$ and $\Omega$), Cauchy--Schwarz, H\"older inequalities, and the norm definitions 
\begin{subequations}\label{eq:bounds}
\begin{align}
|\langle \bA\bu,\bv \rangle| & \leq \max\{\frac{\rho_f}{\tau},\mu,\gamma C_T^{2}\} \|\bu\|_{1,\Omega}\|\bv\|_{1,\Omega} \quad &&\forall  \bu,\bv \in \bH^1_\star(\Omega),\label{a:bound}\\
\langle \bA\bv,\bv \rangle & \geq \min\{\frac{\rho_f}{\tau},\mu\} \|\bv\|^2_{1,\Omega} \quad &&\forall  \bv \in \bH^1_\star(\Omega),\label{a:coer}\\
\langle \bC_1(p,\varphi),(q,\psi)\rangle & \leq \max\{\frac{c_0}{\tau},\kappa\}\|(p,\varphi)\|\|(q,\psi)\|\quad &&\forall (p,\varphi),(q,\psi) \in \rL^2(\Omega)\times\rH^1_0(\Sigma), \label{c1:bound}\\
\langle \bC_1(q,\psi),(q,\psi)\rangle & \geq   \min\{\frac{c_0}{\tau},\kappa\} \|\psi\|_{1,\Sigma}^2\geq 0
\quad &&\forall (q,\psi) \in \rL^2(\Omega)\times\rH^1_0(\Sigma), \label{c1:pos}\\
|\langle \bC_2 w,\zeta\rangle| & \leq \max \{\frac{\rho_p}{\tau^3},\frac{D}{\tau}\} \|w\|_{2,\Sigma}\|\zeta\|_{2,\Sigma} \quad &&\forall w,\zeta \in \rH^2_0(\Sigma),\label{c2:bound}\\
\langle \bC_2 \zeta,\zeta\rangle & \geq \min \{\frac{\rho_p}{\tau^3},\frac{D}{\tau}\} (\|\zeta\|^2_{0,\Sigma}+|\zeta|^2_{2,\Sigma}) \quad &&\forall \zeta \in \rH^2_0(\Sigma),\label{c2:coer}\\
\langle \bB_1\bv,(q,\psi)\rangle & \leq \max\{1,C_T\}\|\bv\|_{1,\Omega}\|(q,\psi)\| \quad &&\forall \bv \in \bH^1_\star(\Omega),(q,\psi) \in \rL^2(\Omega)\times\rH^1_0(\Sigma),\label{b1:bound}\\
\langle \bB_2(q,\psi),\zeta\rangle & \leq \frac{\alpha}{\tau} \|(q,\psi)\|\|\zeta\|_{2,\Sigma} \quad &&\forall (q,\psi) \in \rL^2(\Omega)\times\rH^1_0(\Sigma),\zeta\in \rH^2_0(\Sigma),\label{b2:bound}\\
 \langle \bB_3(q,\psi),\zeta\rangle & \leq \frac{1}{\tau} \|(q,\psi)\|\|\zeta\|_{2,\Sigma} \quad &&\forall (q,\psi) \in \rL^2(\Omega)\times\rH^1_0(\Sigma),\zeta\in \rH^2_0(\Sigma),\label{b3:bound}
\end{align}\end{subequations}
and we note that even if the inertial contributions vanish $\rho_f 
= \rho_p = 0$, the coercivity properties \eqref{a:coer} and \eqref{c2:coer} can still be shown by Poincar\'e--Friedrichs and its well-second-order variant in $\rH^2(\Sigma)$  with clamped boundary conditions (see, e.g., \cite[Chapter 6]{evans}).
It is straightforwardly shown that the linear functionals are bounded in their respective norms 
\begin{equation}\label{eq:func-bound} \|F\| \leq \|\ff\|_{0,\Omega}, \quad \|G\| \leq \|g\|_{0,\Sigma}, \quad \|M\| \leq \frac{1}{\tau}\|m\|_{0,\Sigma}.\end{equation}
We can also establish the following inf-sup condition for $\bB_1$. 
\begin{lemma}\label{lem:inf-sup}
There exists $\beta>0$ such that
\begin{equation}\label{eq:inf-sup}
    \sup_{\bv\in \bH^1_\star(\Omega)\setminus\{\cero\}}\frac{\langle\bB_1\bv,(q,\psi)\rangle  }{\|\bv\|_{1,\Omega}} \geq \beta \|(q,\psi)\| \qquad \forall (q,\psi)\in \rL^2(\Omega)\times\rH_0^1(\Sigma).
\end{equation}
\end{lemma}
\begin{proof}
From the surjectivity of the divergence operator in $\bH^1_{\Gamma^{\bu}\cup\Sigma}(\Omega)$ (see \cite{ern04}),  we know that for any $q\in \rL^2(\Omega)$ and $\psi\in \rH^1_0(\Sigma)$ with $c:=\frac{1}{|\Omega|}\int_\Sigma\psi$, 
there exists $\bv_1 \in \bH^1_{\Gamma^{\bu}\cup\Sigma}(\Omega)$ such that 
\begin{equation} \vdiv \bv_1 = - q - c\quad \text{in $\Omega$},  \qquad \bv_1\cdot\bn_\Sigma = 0\quad \text{on $\Sigma$},\qquad\|\bv_1\|_{1,\Omega} \leq \widetilde{\beta} (\|q\|_{0,\Omega}+\|\psi\|_{1,\Sigma}).\label{eq:v1}\end{equation}
Moreover, the unique solvability of the following Stokes equations 
(see, e.g., \cite{girault-raviart}) 
\begin{align*} 
- \bDelta \bv_2 +\nabla q_2  = \cero,\quad \vdiv\bv_2 = c \quad \text{in $\Omega$}, \quad
 \bv_2 = \psi\bn_\Sigma\quad \text{on $\Sigma$},\quad
\bv_2 = \cero\quad \text{on $\partial\Omega\setminus\Sigma$},
\end{align*}
provides that there exists a unique  $\bv_2  \in 
\bH_{\partial\Omega  \setminus \Sigma}^1(\Omega)\subset \bH^1_\star(\Omega)$ with the properties
\begin{equation}
\bv_2\cdot\bn_\Sigma = \psi \quad \text{on $\Sigma$} \quad\text{and}\quad  \|\bv_2\|_{1,\Omega} \leq \widehat{\beta} \|\psi\|_{1,\Sigma}.\label{eq:v2}
\end{equation}
 The combination of \eqref{eq:v1} and \eqref{eq:v2} for $\tilde{\bv}:=\bv_1+\bv_2\in \bH^1_\star(\Omega)$, together with \eqref{eq:duality_equalities} results in
\begin{align*}
    \sup_{\bv\in \bH^1_\star(\Omega)\setminus\{\cero\}}\frac{\langle\bB_1\bv,(q,\psi)\rangle  }{\|\bv\|_{1,\Omega}} \geq \frac{\langle\bB_1\tilde{\bv},(q,\psi)\rangle  }{\|\tilde{\bv}\|_{1,\Omega}} = \frac{\|q\|^2_{0,\Omega}+\|\psi\|^2_{1,\Sigma}}{\|\tilde{\bv}\|_{1,\Omega}}\geq \beta (\|q\|_{0,\Omega}+\|\psi\|_{1,\Sigma}),
\end{align*}
and consequently, concludes the proof with the existence of a positive constant $\beta:=(\sqrt{2}(\widetilde{\beta}+\widehat{\beta}))^{-1}>0$ that depends on the elliptic regularity constants 
of the auxiliary problems above.
\end{proof}

\subsection{Abstract results}
In this section we state two abstract results (the Fredholm alternative for compact operators,  and the classical Babu\v{s}ka--Brezzi theory for perturbed saddle-point problems, respectively) in a specific form that accommodates  the well-posedness analysis of \eqref{eq:operator}. Proofs of these theorems can be found in, e.g., \cite[Theorem 6.9]{gatica2025} and \cite[Lemma 3.4]{gatica11}, respectively.  

\begin{theorem}\label{th:fredholm}
  Let  $X$ be a Hilbert space and consider the linear operator $ (\cA + \cK) : X \to X'$. Assume that $\cA + \cK$ is injective,  $\cK$ is compact, and $\cA$ is self-adjoint and invertible. Then, $\cA + \cK$ is invertible. 
\end{theorem}

\begin{theorem}\label{th:perturbed}
 Let $X,Y$ be two Hilbert spaces    and three continuous bilinear forms $a(\cdot,\cdot)$ on $X\times X$, $b(\cdot,\cdot)$ on $X\times Y$, and $c(\cdot,\cdot)$ on $Y\times Y$; which define three linear continuous operators $A: X \rightarrow X'$, $B : X \rightarrow Y'$ and $C:Y \rightarrow Y'$. Suppose that 
 \begin{itemize}
     \item $a(\cdot,\cdot)$ is symmetric and coercive over $X$, i.e., $a(u,v)=a(v,u)$ and 
     \[ |a(v,v)| \geq \alpha \norm{v}_X^2  \qquad \forall v\in X,\]
     \item $b(\cdot,\cdot)$ satisfies the inf-sup condition 
     \[ \sup_{v\in X\setminus\{0\}} \frac{b(v,q)}{{\norm{v}}_X} \geq \beta {\norm{q}}_Y  \qquad \forall q \in Y,\]
\item  $c(\cdot,\cdot)$ is symmetric and positive semi-definite over $Y$, i.e.,   $c(p,q)=c(q,p)$ and 
\[ c(q,q) \geq 0  \qquad \forall q\in Y.\]
 \end{itemize}
 Then, for for every $F\in X'$ and $G\in Y'$, there exists a unique $(u,p)\in X\times Y$ satisfying
\begin{align*}
a(u,v)+ b(v,p) = F(v) &\qquad \forall v\in X,\\
b(u,q) - c(p,q) = G(q) &\qquad \forall q\in Y,
\end{align*}
as well as the following continuous dependence on data 
\[
    {\norm{u}}_X + {\norm{p}}_Y \lesssim {\norm{F}}_{X'}+{\norm{G}}_{Y'}.
\]
\end{theorem}

\subsection{Verification of well-posedness}
With respect to \eqref{eq:operator}, let us denote the product space $\bbX:=\bH^1_\star(\Omega)\times [\rL^2(\Omega)\times \rH^1_0(\Sigma)]\times \rH^2_0(\Sigma)$ and denote the left-hand side operator as $\cA + \cK:\bbX \to \bbX'$, with the linear (and, owing to the estimates \eqref{eq:bounds}, clearly bounded) operators $\cA,\cK:\bbX\to\bbX'$ 
\[ \cA := \begin{pmatrix}
    \bA & \bB_1^* & \cero \\
    \bB_1 & - \bC_1 & \cero \\
    \cero & \cero  & \bC_2 
\end{pmatrix}, \qquad \cK := \begin{pmatrix}
    \cero & \cero & \cero \\
    \cero & \cero & \bB_2^* + \bB_3^* \\
    \cero & \bB_2 -\bB_3 & \cero
\end{pmatrix}.\]
And, proceeding similarly to, e.g., \cite{gatica11,oyarzua16}, the goal of this section is to use Theorem~\ref{th:fredholm} to show that \eqref{eq:weak} is well-posed. 

\begin{lemma}\label{lem:K}
 The map   $\cK$ is compact. 
\end{lemma}
\begin{proof}
Let $\mathrm{id}: \rH^1(\Sigma)\to \rH^1(\Sigma)$ denote the identity operator and $\mathrm{i}_C$ denote the compact embedding from $\rH^2(\Sigma)$ into $\rH^1(\Sigma)$ (we could also use the compactness of the identity map of $\rH^2_0(\Sigma)$ into $\rH^1_0(\Sigma)$, cf. \cite[Theorem 8.3]{agmon09}). Furthermore, we have that 
\[-\langle [\bB_2+\bB_3](q,\psi),\zeta\rangle = -(\mathcal{R}_{1,\Sigma} \{ [\bB_2+\bB_3] (q,\psi)\},\zeta)_{1,\Sigma} =  \frac{1}{\tau}\int_\Sigma \zeta  \psi + \frac{\alpha}{\tau}\int_\Sigma \nabla_\Sigma \zeta \cdot\nabla_\Sigma \psi,\]
and the right-hand side is simply a scaled (equivalent) inner product in $\rH^1(\Sigma)$ between $\zeta$ and $\psi$, 
which implies the operator identification $\bB_2+\bB_3=\mathcal{R}_{1,\Sigma}^{-1}\circ (-\text{const.}_1\times\mathrm{id})\circ \mathrm{i}_C$, and therefore $\bB_2+\bB_3$ is a compact operator.
Note also that the same argument holds for $\bB_2-\bB_3$. Indeed, $\bB_2-\bB_3=\mathcal{R}_{1,\Sigma}^{-1}\circ (-\text{const.}_2\times\mathrm{id})\circ \mathrm{i}_C$. Then, we can assert that the map 
\begin{align*}
    \mathcal{K}: &\ \bH^1_\star(\Omega)\times [\rL^2(\Omega)\times \rH^1_0(\Sigma)]\times \rH^2_0(\Sigma)\to [\bH^1_\star(\Omega)\times (\rL^2(\Omega)\times \rH^1_0(\Sigma))\times \rH^2_0(\Sigma)]', \\
    & \langle \mathcal{K} (\bu,(p,\varphi),w),(\bv,(q,\psi),\zeta)\rangle : = \langle [\bB_2+\bB_3](q,\psi),w  \rangle + \langle [\bB_2 -\bB_3] (p,\varphi),\zeta \rangle,
\end{align*}
is indeed compact.
\end{proof}

\begin{lemma}\label{lem:A}
The map $\cA$ is self-adjoint and invertible.     
\end{lemma}
\begin{proof}
Let us denote $\vec{\bx}= (\bu,(p,\varphi),w)^{\top} \in \bbX$, and note first that the first two equations that define the problem $\cA \vec{\bx} = ( F, G, M)^{\top}$ are decoupled from the third one. In consequence, we can simply analyse the unique solvability of the two separate problems 
 \begin{equation}\label{eq:split} \begin{pmatrix}
      \bA & \bB_1^* \\
    \bB_1 & - \bC_1 \end{pmatrix}\begin{pmatrix} \bu \\ (p,\varphi) \end{pmatrix} = \begin{pmatrix}
        F \\ G
    \end{pmatrix} \quad \text{and} \quad \bC_2 w = M.\end{equation}
The first problem in \eqref{eq:split} consists in finding $(\bu,(p,\varphi)) \in \bH^1_\star(\Omega)\times [\rL^2(\Omega)\times \rH^1_0(\Sigma)]$ such that 
 \begin{equation}\label{eq:split1}\begin{aligned}
     a(\bu,\bv) + b_1(\bv,(p,\varphi)) &= F(\bv) \quad &\forall \bv \in \bH_\star^1(\Omega),\\
b_1(\bu,(q,\psi)) - c_1((p,\varphi),(q,\psi)) & = G((q,\psi)) \quad &\forall (q,\psi) \in    \rL^2(\Omega)\times \rH^1_0(\Sigma).   \end{aligned}\end{equation}
The bilinear forms $a(\cdot,\cdot)$ and $c_1(\cdot,\cdot)$ are clearly symmetric. Also, using \eqref{a:bound}, \eqref{b1:bound},  \eqref{c1:bound} and \eqref{eq:func-bound} we have that all bilinear forms and linear functionals in \eqref{eq:split1} are bounded. In addition, the coercivity of $a(\cdot,\cdot)$ over $\bH_\star^1(\Omega)$ is established in \eqref{a:coer}, the semi-positive-definiteness of $c_1(\cdot,\cdot)$ is stated in \eqref{c1:pos}, and the inf-sup condition for $b_1(\cdot,(\cdot,\cdot))$ is proven in Lemma~\ref{lem:inf-sup}. Then it suffices to apply Theorem~\ref{th:perturbed} to conclude that there exists a unique $(\bu,(p,\varphi)) \in \bH^1_\star(\Omega)\times [\rL^2(\Omega)\times \rH^1_0(\Sigma)]$ solution to \eqref{eq:split1} that satisfies 
\[ \|\bu\|_{1,\Omega} + \|(p,\varphi)\| \leq C (\|\ff\|_{0,\Omega} + \|g\|_{0,\Sigma}). \]

On the other hand, the second problem in \eqref{eq:split} consists in finding $w\in \rH^2_0(\Sigma)$ such that 
\begin{equation}\label{eq:split2} c_2(w,\zeta) = M(\zeta) \qquad \forall \zeta \in \rH^2_0(\Sigma). \end{equation}
Since $c_2(\cdot,\cdot)$ is bounded and coercive in $\rH^2_0(\Sigma)$ (cf. \eqref{c2:bound}-\eqref{c2:coer}) and $M(\cdot)$ is bounded (cf. \eqref{eq:func-bound}), a direct application of Lax--Milgram's lemma yields that there exists a unique $w\in \rH^2_0(\Sigma)$ solution to \eqref{eq:split2} and the following continuous dependence on data holds 
\[\| w\|_{2,\Sigma} \leq C\|m\|_{0,\Sigma}. \]

These steps imply that the problem $\cA \vec{\bx} = ( F, G, M)^{\top}$ is well-posed. To check that $\cA$ is self-adjoint, it suffices to recall that the bilinear forms $a(\cdot,\cdot),c_1(\cdot,\cdot),c_2(\cdot,\cdot)$ are symmetric. 
\end{proof}

\begin{lemma}\label{lem:A+K}
The operator $\cA+ \cK$ is injective.     
\end{lemma}
\begin{proof}
We consider the problem $(\cA+\cK)\vec{\bx} = (\cero,\cero,\cero)^{\top}$, and our goal is to show that $\vec{\bx}$ must be the zero vector in $\bbX$. In terms of bilinear forms, this problem is written as 
 \begin{equation}\label{eq:inje}\begin{aligned}
     a(\bu,\bv) + b_1(\bv,(p,\varphi)) & =  0  \quad \forall \bv \in \bH_\star^1(\Omega),\\
b_1(\bu,(q,\psi)) - c_1((p,\varphi),(q,\psi)) + b_2((q,\psi),w) + b_3((q,\psi),w) & = 0 \quad \forall (q,\psi) \in    \rL^2(\Omega)\times \rH^1_0(\Sigma),\\
b_2((p,\varphi),\zeta) - b_3((p,\varphi),\zeta) + c_2(w,\zeta) & = 0 \quad \forall \zeta \in \rH^2_0(\Sigma).\end{aligned}\end{equation}
For the bulk fluid pressure, we use the inf-sup condition for $b_1(\cdot,\cdot)$ \eqref{eq:inf-sup}, the first equation in \eqref{eq:inje}, and the boundedness of $a(\cdot,\cdot)$ \eqref{a:bound}, to readily  obtain the following estimate 
\begin{equation*}
\begin{aligned}
\beta \|p\|_{0,\Omega} \leq \beta  \| (p,\varphi)\| &\leq \sup_{\bv\in \bH^1_\star(\Omega)\setminus\{\cero\}} \frac{b_1(\bv,(p,\varphi))}{\|\bv\|_{1,\Omega}} \\
& \leq \sup_{\bv\in \bH^1_\star(\Omega)\setminus\{\cero\}} \frac{|a(\bu,\bv)|}{\|\bv\|_{1,\Omega}} \leq \max\{\frac{\rho_f}{\tau},\mu,C_T^2\gamma\} \|\bu\|_{1,\Omega}.\end{aligned}\end{equation*}

Next, choosing as test functions $\bv = \bu$, $(q,\psi)=(p,\varphi)$, and $\zeta = w$, adding the first and third equations in \eqref{eq:inje} and subtracting the second one, we obtain 
\[ a(\bu,\bu) + c_1((p,\varphi),(p,\varphi)) - 2b_3((p,\varphi),w) + c_2(w,w)  = 0.\]
From this relation, and using the coercivity of $a(\cdot,\cdot)$, the definitions of $c_1(\cdot,\cdot)$ and $c_2(\cdot,\cdot)$ (which form a norm in $\rH^1(\Sigma)$ and $\rH^2(\Sigma)$, respectively), together with the definition of $b_3$, we get 
\begin{align*}
 0 & \geq \min\{\frac{\rho_f}{\tau},\mu\} \|\bu\|_{1,\Omega}^2 + 
 \frac{c_0}{\tau} \|\varphi\|^2_{0,\Sigma} + \kappa \|\nabla_\Sigma \varphi\|^2_{0,\Sigma} + 
 \frac{\rho_p}{\tau^3}\|w\|^2_{0,\Sigma} + \frac{D}{\tau}\|\nabla_\Sigma^2w\|^2_{0,\Sigma} + \frac{2}{\tau}\int_\Sigma w\varphi \\
& \geq \min\{\frac{\rho_f}{\tau},\mu\} \|\bu\|_{1,\Omega}^2 + 
 \frac{c_0}{\tau} \|\varphi\|^2_{0,\Sigma} + \kappa \|\nabla_\Sigma \varphi\|^2_{0,\Sigma} + 
 \frac{\rho_p}{\tau^3}\|w\|^2_{0,\Sigma} + \frac{D}{\tau}\|\nabla_\Sigma^2w\|^2_{0,\Sigma} - \frac{2}{\tau}\| w\|_{0,\Sigma}\|\varphi\|_{0,\Sigma} \\
 & \geq \min\{\frac{\rho_f}{\tau},\mu\}  \|\bu\|_{1,\Omega}^2 + (\frac{c_0}{\tau}-\frac{1}{\epsilon \tau}) \|\varphi\|^2_{0,\Sigma} + \kappa \|\nabla_\Sigma \varphi\|^2_{0,\Sigma} + (\frac{\rho_p}{\tau^3}- \frac{\epsilon}{\tau})\|w\|^2_{0,\Sigma} + 
 \frac{D}{\tau}\|\nabla_\Sigma^2w\|^2_{0,\Sigma}, 
\end{align*}
where we have used Cauchy--Schwarz and Young's inequalities. Then, choosing $\frac{1}{c_0} < \epsilon < \frac{\rho_p}{\tau^2}$ (actually, because of the boundary conditions and equivalence between the semi-norms and norms in the spaces for plate deflection and plate pressure moment, we can take $\frac{1}{c_0} \leq \epsilon \leq \frac{\rho_p}{\tau^2}$), we can infer that $\bu = \cero$, $\varphi=0$, and $w=0$. 
\end{proof}

The proof of well-posedness of the coupled system \eqref{eq:weak} is a direct consequence of Theorem~\ref{th:fredholm}. 

\begin{theorem}[well-posedness]\label{th:wellp}
    For given $\ff\in \bL^2(\Omega)$, $g\in \rL^2(\Sigma)$, and $m \in \rL^2(\Sigma)$, there exists a unique $\vec{\bx}= (\bu,(p,\varphi),w)^{\top} \in \bbX$ solution to \eqref{eq:weak}. Moreover, there exists a positive constant $C$ such that 
    \begin{equation}\label{eq:cont-dep}
        \|\bu\|_{1,\Omega} + \|(p,\varphi)\| + \|w\|_{2,\Sigma} \leq C \left( \|\ff\|_{0,\Omega} + \|g\|_{0,\Sigma} + \|m\|_{0,\Sigma}\right).
    \end{equation}
\end{theorem}
\begin{proof}
The existence and uniqueness of $\vec{\bx}\in \bbX$ follows from Lemmas~\ref{lem:K},\ref{lem:A}, and \ref{lem:A+K} confirming the hypotheses of Theorem~\ref{th:fredholm}. The continuous dependence on data \eqref{eq:cont-dep} is carried out using similar steps as in the proof of Lemma~\ref{lem:A+K}.  That is, we use $\bv = \bu$, $(q,\psi)=(p,\varphi)$, and $\zeta = w$ as test functions in \eqref{eq:weak}. Combining the first and second resulting equations gives 
\[ a(\bu,\bu) + c_1((p,\varphi),(p,\varphi)) - [b_2+b_3]((p,\varphi),w)   = F(\bu) - G((p,\varphi)),\]
and using again the coercivity of $a(\cdot,\cdot)$, the positivity of $c_1(\cdot,\cdot)$, and the boundedness of the operators $F(\cdot)$ and $G(\cdot)$, we arrive at 
\begin{align*}
   & \min\{\frac{\rho_f}{\tau},\mu\} \|\bu\|_{1,\Omega}^2 +  \frac{c_0}{\tau} \|\varphi\|^2_{0,\Sigma} + \kappa\|\nabla_\Sigma \varphi\|^2_{0,\Sigma} \hspace{-0.1cm} - \hspace{-0.1cm} [b_2+b_3]((p,\varphi),w) \leq \|\ff\|_{0,\Omega} \|\bu\|_{1,\Omega} \hspace{-0.05cm} + \hspace{-0.05cm} \|g\|_{0,\Sigma}\|(p,\varphi)\| .
\end{align*}
On the other hand, using the third resulting equation 
\[[b_2-b_3]((p,\varphi),w) + c_2(w,w) = M(w),\] 
and invoking the coercivity of $c_2(\cdot,\cdot)$ as well as the continuity of the linear functional $M(\cdot)$, we can assert that 
\begin{equation*}
\frac{\rho_p}{\tau^3}\|w\|^2_{0,\Sigma} + \frac{D}{\tau} \|\nabla_\Sigma^2w\|^2_{0,\Sigma} + [b_2-b_3]((p,\varphi),w) \leq \|m\|_{0,\Sigma} \|w\|_{2,\Sigma}.\end{equation*}
Adding the previous two estimates and using the boundedness of $b_3(\cdot,\cdot)$ together with Young's inequality (requiring again $\epsilon$ such that $\frac{1}{c_0} < \epsilon < \frac{\rho_p}{\tau^2}$), we obtain the bound 
\begin{equation}\label{eq:aux09}  \|\bu\|_{1,\Omega} + \|\varphi\|_{1,\Sigma} + \|w\|_{2,\Sigma} \leq C(\|\ff\|_{0,\Omega} + \|g\|_{0,\Sigma} + \|m\|_{0,\Sigma} ).\end{equation}
Finally, we use again the first resulting equation, the inf-sup condition for $b_1(\cdot,\cdot)$, and the boundedness of $a(\cdot,\cdot)$, to end up with 
\begin{align*}
\beta \|p\|_{0,\Omega} \leq \beta  \| (p,\varphi)\|  &\leq \sup_{\bv\in \bH^1_\star(\Omega)\setminus\{\cero\}} \frac{b_1(\bv,(p,\varphi))}{\|\bv\|_{1,\Omega}} \\
&  = \sup_{\bv\in \bH^1_\star(\Omega)\setminus\{\cero\}} \frac{F(\bu) - a(\bu,\bv)}{\|\bv\|_{1,\Omega}} \leq \|\ff\|_{0,\Omega} + \max\{\frac{\rho_f}{\tau},\mu,C_T^2\gamma\} \|\bu\|_{1,\Omega},\end{align*}
which, together with \eqref{eq:aux09}, implies the desired result. 
\end{proof}

\section{Virtual element framework}\label{sec:vem}
This section aims to introduce the conforming \ac{vem} spaces for \eqref{eq:weak}, specifying appropriate \acp{dof} along with the definition of suitable polynomial projection operators that guarantee the computability of the method.

\paragraph{Admissible meshes} Let $\Omega$ be a contractible polyhedron with Lipschitz boundary $\partial\Om$ and $\{\Om_h\}_{h>0}$ be a sequence of decompositions of $\Om$ into general polyhedral elements $K$ of diameter $h_K$  with the mesh-size $h:=\sup_{K\in\Om_h}h_K$. This consequently provides the polygonal decomposition $\Gamma_h$ of the part of the boundary $\Gamma$ and $\Sigma_h$ of the surface $\Sigma$. 
\par Let $\cF_h$ (resp. $\cE_h$)  denote the set of the faces (resp. edges) of the polyhedral decomposition $\Om_h$. The following mesh assumptions are considered throughout this paper. We assume that there exists a universal constant $\rho>0$ such that
\begin{enumerate}[label={(\bfseries M\arabic*)}]
    \item\label{M1} each element $K$ is star-shaped with respect to a ball of radius greater than equal to $\rho h_K$,
    \item\label{M2} every face $F$  of $K$ is of diameter $h_F$ and is star-shaped with respect to a disk of radius greater than equal to $\rho h_K$,
    \item\label{M3} every edge $e$ of $K$ has length $h_e\geq \rho h_K$. 
\end{enumerate}
The local (resp. global) number of vertices, edges and faces are denoted by $N_v^K, N_e^K, N_f^K$ (resp. $N_v, N_e, N_f$). An important consequence of the mesh assumptions is that the global number of faces $N_f$ can be uniformly bounded by $N_\delta \geq 4$ in 3D depending only on $\rho$  (refer to \cite[Lemma~1.12]{di2020hybrid} for a proof).  Let $\rP_r(\cD)$ (resp. $\bP_r(\cD), \bbP_r(\cD)$) denote the space of scalar (resp. vector, tensor) valued polynomials of degree $\leq r$ on a domain $\cD$. 
Note that here the domain $\cD$ can be of dimension $1,2$ or $3$. We use the following polynomial decomposition  on each polyhedron $K$ 
\[
    \bP_r(K) = \nabla_K \rP_{r+1}(K) \oplus \bG_r^\perp(K),
\]
where the complement $\bG_r^\perp(K)$ contains the polynomials of the form $\bx^\perp\rP_{r-1}(K)$ with $\bx^\perp:=(x_2,-x_1)$ for $d=2$ and  $\bx\wedge \bP_{r-1}(K)$ for $d=3$.  Let $\bx_{\cD}$ and $h_\cD$ be the barycentre and the diameter of a domain $\cD$ respectively. Define the space of scaled monomials of degree $\leq k$ for $k\geq 0$ on a domain $\cD$ as
\[\rM_k(\cD) = \left\{\Big(\frac{\bx-\bx_\cD}{h_\cD}\Big)^{\bt}: 0\leq |\bt|\leq k\right\}\]
for a multi-index $\bt = (t_1,\dots,t_d)$ with $|\bt|:=t_1+\dots+t_d$. The notation $\bM_k(\cD)$ (resp. $\bbM_k(\cD)$) stand for vector and tensor valued scaled monomials. Also the space $\bM_k^\perp(\cD)$ is the scaled basis of the complement $\bG_k^\perp(\cD)$ such that $\nabla_\cD\rM_{k+1}(\cD)\oplus\bM_k^\perp(\cD)$ forms a scaled basis of $\bP_k(\cD)$.  It is well-known in \ac{vem} literature that the \acp{dof} scale like $\approx 1$ and hence the scaled monomials are considered as the basis of any polynomial space involved in the definition of the \acp{dof}.
\paragraph{Polynomial projections}
For each polyhedron or face $\cD$, we introduce the following projections:
\begin{itemize}
    \item  For any $k\geq 0$, the $\rL^2$  projection $\Pi_k^{0,\cD}: \rL^2(\cD)\to\rP_k(\cD)$ is defined, for any $v\in \rL^2(\cD)$, by
    \begin{align}\label{L2-proj}
        \int_\cD \Pi_k^{0,\cD}v \chi_k = \int_\cD v\chi_k\quad\text{for all}\;\chi_k\in\rP_k(\cD)
    \end{align}
    with analogous definitions of $\bPi_k^{0,\cD}$ and $\bbPi_k^{0,\cD}$ for vector and tensor functions.
    \item  For any $k\geq 1$, the $\rH^1$-seminorm projection $\Pi_k^{\nabla,\cD}: \rH^1(\cD)\to\rP_k(\cD)$ is defined, for any $v\in \rH^1(\cD)$, by
    \begin{subequations}\label{nabla_proj}
    \begin{align}
        \int_\cD \nabla\Pi_k^{\nabla,\cD}v\cdot \nabla\chi_k &= \int_\cD \nabla v\cdot\nabla\chi_k\quad\text{for all}\;\chi_k\in\rP_k(\cD),\\
        \int_{\partial\cD}\nabla\Pi_k^{\nabla,\cD}v &=  \int_{\partial\cD}\nabla v
    \end{align}\end{subequations}
    with analogous definition of $\bPi_k^{\nabla,\cD}$ for vector functions. 
    \item For any $k\geq 2$, the $\rH^2$-seminorm projection $\Pi_k^{\nabla^2,\cD}: \rH^2(\cD)\to\rP_k(\cD)$ is defined, for any $v\in \rH^2(\cD)$, by
    \begin{subequations}\label{nabla2_proj}
    \begin{align}
        \int_\cD \nabla^2\Pi_k^{\nabla^2,\cD}v : \nabla^2\chi_k &= \int_\cD \nabla^2 v:\nabla^2\chi_k\quad\text{for all}\;\chi_k\in\rP_k(\cD),\\
        \overline{\Pi_k^{\nabla^2,\cD}v} = \overline{v} \quad&\text{and}\quad \overline{\nabla\Pi_k^{\nabla^2,\cD}v} = \overline{\nabla v},
        \end{align}\end{subequations}
        where 
        is the Hessian matrix (of second-order derivatives) and $\overline{v}$ is the average $\frac{1}{N_K}\sum_{i=1}^{N_v^K}v(z_i)$ of the values of $v$ at the vertices $z_i$ of $K$.
\end{itemize}
\paragraph{Virtual element spaces for the 3D Stokes equations}
The definition of the enhanced 3D VE space in this paper is followed from \cite{beirao20}. We set the polynomial degree $k\geq 2$ in the rest of this paper.  For each face $F\in\partial K$, the VE space $\rB^k(F)$  locally solves the Poisson equation with Dirichlet boundary conditions and is defined by 
\begin{align*}
    \rB^k(F):=\begin{dcases}
    \begin{rcases}
        v\in \rH^1(F):& v|_{\partial F}\in C^0(\partial F),\;v|_e\in\rP_k(e)\quad\text{for all}\;e\in\partial F, \;\\ &\Delta_F v\in\rP_{k+1}(F),\;(v-\Pi_k^{\nabla,F}v,\chi_{k+1})_F = 0\;\\&\text{for all}\;\chi_{k+1}\in\rP_{k+1}(F)\setminus\rP_{k-2}(F)
    \end{rcases}
\end{dcases}
\end{align*}
with $\bB^k(F)=[\rB^k(F)]^3$ and  the boundary space $\bB^k(\partial K)=[\rB^k(\partial K)]^3$ for
\[\rB^k(\partial K):=\{v\in C^0(\partial K): v|_F\in\rB^k(F)\quad\text{for all}\;F\in\partial K\}.
\]
First we define an extended local VE space $\tV_h(K)$  for each $K\in\Om_h$ as
\begin{align*}
    \tV_h^k(K):=\begin{dcases}
        \begin{rcases}
            \bv\in \bH^1(K): &\bv|_{\partial K}\in \bB^k(\partial K),\;\Delta \bv+\nabla s\in \bG_k^\perp(K)\\&\text{for some}\;s\in \rL^2_0(K),\;\vdiv \bv\in\rP_{k-1}(K)
        \end{rcases},
    \end{dcases}
\end{align*}
and the  enhanced local VE  space $\bV_h(K)$ with additional orthogonality condition as
\[
    \bV_h^k(K):=\begin{dcases}
        \begin{rcases}
            \bv\in\tV_h^k(K): (\bv-\bPi_k^{\nabla,K}\bv,\bg_k^\perp)_{0,K} =0\quad\text{for all}\; \bg_k^\perp\in \bG_k^\perp(K)\setminus \bG_{k-2}^\perp(K)
        \end{rcases}.
    \end{dcases}
\]
The global VE space is defined by
\[\bV_h^k:=\{\bv_h\in \bH^1_\star(\Omega): \bv_h|_K\in\bV_h^k(K)\quad\text{for all}\;K\in\Om_h\}. \]
The \acp{dof} for $\bv\in\bV_h^k(K)$ and  $K\in\Om_h$ can be taken as follows:
\begin{enumerate}[label={(\bfseries D\arabic*$\bv$)}]
    \item\label{D1v} Edge moments: the values of $\bv$ at the vertices of $K$ and at $k-1$ distinct points on every edge $e$ of $K$,
    \item\label{D2v} Face moments: the normal and tangential components of $\bv$ as
    \[\frac{1}{h_F}\int_F (\bv\cdot\bn_K^F)\chi_{k-2}\quad\text{and}\quad \frac{1}{h_F}\int_F \bv|_F\cdot\bchi_{k-2}\]
    for all $\chi_{k-2}\in\cM_{k-2}(F)$ and $\bchi_{k-2}\in[\cM_{k-2}(F)]^2$.
    \item\label{D3v} Cell moments: the curl and div part of $\bv$ as
    \[\frac{1}{h_K}\int_K\bv\cdot \bchi_{k-2}^\perp\quad\text{and}\quad \frac{1}{h_K}\int_K(\vdiv \bv)\chi_{k-1}\]
    for all $\bchi_{k-2}^\perp\in\bM_{k-2}^\perp(K)$ and $\chi_{k-1}\in\rM_{k-1}(K)\setminus\mathbb{R}$. 
\end{enumerate}
It is shown in \cite[Proposition~3.1]{beirao20} that the above set of \acp{dof} is unisolvent for the space $\bV_h^k(K)$.
\begin{lemma}[computable projections on $\bV_h^k$]
   The $\bH^1$ face and element projections $\bPi_k^{\nabla,F}:\bB^k(F)\to\bP_k(F)$ and $\bPi_k^{\nabla,K}:\bV_h^k(K)\to\bP_k(K)$, respectively, also the $\bL^2$ face projection  $\bPi_{k+1}^{0,F}:\bB^k(F)\to\bP_{k+1}(F)$ and the $\bL^2$ element projections $\bPi_k^{0,K}:\bV_h^k(K)\to\bP_k(K)$ along with  $\bbPi_{k-1}^{0,K}:\nabla(\bV_h^k(K))\to\bbP_{k-1}(K)$ are computable in terms of \acp{dof} \ref{D1v}-\ref{D3v}.
\end{lemma} 
\begin{proof}
    The detailed proof can be found in \cite[Proposition~5.1]{beirao20}.
\end{proof}
The  local discrete pressure space $\rQ_h^{k-1}(K)$ is nothing but the space of polynomials $\rP_{k-1}(K)$ and  we can take the \acp{dof} as
\begin{enumerate}[label={({\bfseries D\arabic*}$q$)}]
    \item Cell moments: For any $q\in \rQ_h^{k-1}(K)$ and for all $\chi_{k-1}\in\rM_{k-1}(K)$,
    \[\frac{1}{h_K}\int_K q\;\chi_{k-1}.\]
\end{enumerate}
The global discrete pressure space is the set of piecewise polynomials of degree $\leq k-1$ on $\Om_h$, that is,
\[\rQ_h^{k-1}:=\{q_h\in \rL^2(\Omega): q_h|_K\in \rQ_h^{k-1}(K)\quad\text{for all}\;K\in\Om_h\}.\]
\paragraph{Virtual element spaces for the 2D poroelastic plate equations}
To approximate the displacement space, we consider the local conforming VE space $\rW_h^k(F)$  as a set of solutions to a biharmonic problem over $F$ with clamped boundary conditions on $\partial F$, and  it is defined, for $k\geq 2$ and $r=\max\{k,3\}$, by 
	\[
	\rW_h^k(F):= \begin{dcases}
		\begin{rcases}
			&\hspace{-0.35cm}w  \in H^2(F)\cap C^1(\partial F) : \Delta^2_\Sigma w \in\rP_k(F),\; w|_e\in \rP_r(e)\;\text{and}\;\nabla_\Sigma w|_e\cdot\bn^e_F\in\rP_{k-1}(e)\\&\qquad\forall\;e\in\partial F,\;\text{and}\; (w-\Pi_k^{\nabla^2,F}w,\chi_k)_F=0\quad\forall\;\chi_k\in\rP_k(F)\setminus\rP_{k-4}(F)
		\end{rcases}.
	\end{dcases}
	\]
 The following \acp{dof} are unisolvent for the space $\rW_h^k(F)$. For any $w\in \rW_h^k(F)$,
 \begin{enumerate}[label={({\bfseries D\arabic*}$w$)}]
     \item\label{D1w} Edge moments: 
     \begin{itemize}
         \item The values $w(z)$ and $h_z \nabla w(z)$ ($h_z$ is characteristic length associated with $z$) of zero and first derivatives of $w$ at the vertices $z$ of $F$, 
         \item $\displaystyle \frac{1}{h_e}\int_e\partial_{\bn} w\;\chi_{k-3}\quad\forall\;\chi_{k-3}\in\rM_{k-3}(e),\;e\in\partial F$,
         \item $\displaystyle \frac{1}{h_e}\int_e w\;\chi_{k-4}\quad\forall\;\chi_{k-4}\in\rM_{k-4}(e),\;e\in\partial F$,
     \end{itemize}  
     \item\label{D2w} Face moments: 
     \[\frac{1}{h_F}\int_F w\;\chi_{k-4}\quad\forall\;\chi_{k-4}\in\rM_{k-4}(F).\]
 \end{enumerate}
 The global displacement VE space  is defined by
 \[\rW_h^k:=\{w_h\in \rH^2_0(\Sigma) : w_h|_F\in \rW_h^k(F)\quad\text{for all}\;F\in\Sigma_h\}.\]
 
 \begin{lemma}[computable projections on $W_h^k$ \cite{khot23}]
     The $\rH^2$-projection $\Pi_k^{\nabla^2,F}:\rW_h^k(F)\to\rP_k(F)$ and the $\rL^2$-projections $\Pi_k^{0,F}:\rW_h^k(F)\to\rP_k(F)$  together with $\bPi_{k-1}^{0,F}:\nabla_\Sigma(\rW_h^k(F))\to\bP_{k-1}(F)$ are computable in terms of the \acp{dof} \ref{D1w}-\ref{D2w}.
 \end{lemma}
 To approximate the pressure space, one can consider the space $\rB^k(F)$, but note that, it is a super-enhanced space defined particularly to have the computable $\rL^2$ face projection $\Pi_{k+1}^{0,F}$ at hand.  Since here only computable $\rL^2$-projection $\Pi_\ell^{0,F}$ is required, we define the local enhanced space, for $\ell\geq 1$,  as
 \[\rR_h^\ell(F):= \begin{dcases}
	\begin{rcases}
		&r \in \rH^1(F)\cap C^0(\partial F): \Delta_\Sigma r \in \rP_\ell(F),\; r|_e\in\rP_\ell(e)\quad\forall\;e\in\partial F,\\&\qquad\text{and}\; (r-\Pi^{\nabla,F}_\ell r,\chi_{\ell})_F=0\quad\forall\;\chi_\ell\in\rP_\ell(F)\setminus\rP_{\ell-2}(F)
	\end{rcases}.
\end{dcases}
\]
The following \acp{dof}  are defined such that the triplet $(F,\rR_h^{\ell}(F),\{\text{\ref{D1r}-\ref{D2r}}\})$ 
forms a finite element in the sense of Ciarlet. For any $r\in \rR_h^\ell(F)$,
\begin{enumerate}[label={({\bfseries D\arabic*}$r$)}]
    \item \label{D1r} Edge moments: values of $r$ at the vertices $z$ of $F$ and at the $\ell-1$ distinct points on each edge $e$ of $F$,
    \item\label{D2r} Face moments: 
    \[\frac{1}{h_F}\int_F r\;\chi_{\ell-2}\quad\forall\chi_{\ell-2}\in\cM_{\ell-2}(F).\]
\end{enumerate}
The global pressure VE space is defined by 
\[\rR_h^\ell:=\{r_h\in \rH^1_0(\Omega): r_h|_F\in \rR_h^\ell(F)\quad\text{for all}\;F\in\Sigma_h\}.\]

\begin{lemma}[computable projections on $R_h^\ell$]
    The $\rH^1$-projection $\Pi_\ell^{\nabla,F}:\rR_h^\ell(F)\to\rP_\ell(F)$ and the $\rL^2$-projections $\Pi_\ell^{0,F}:\rR_h^\ell(F)\to\rP_\ell(F)$ along with $\bPi_{\ell-1}^{0,F}:\nabla_{\Sigma}(\rR_h^\ell(F))\to\bP_{\ell-1}(F)$ are computable in terms of the \acp{dof} \ref{D1r}-\ref{D2r}.
\end{lemma}
\paragraph{Discrete bilinear forms}
With $\ell = k-1$, we set  the global discrete operators $A_h:\bV_h^k\to[\bV_h^k]'$, $B_{1h}:\bV_h^k\to[\rQ_h^{k-1}\times \rR_h^{k-1}]'$, $C_{1h}:\rQ_h^{k-1}\times \rR_h^{k-1}\to [\rQ_h^{k-1}\times \rR_h^{k-1}]'$, $B_{2h},B_{3h}:\rQ_h^{k-1}\times \rR_h^{k-1}\to[\rW_h^k]'$ and $C_{2h}:\rW_h^k\to[\rW_h^k]'$ as the summation over local discrete counterparts. In particular, for any $(\bu_h,(p_h,\varphi_h),w_h)$ and $(\bv_h,(q_h,\psi_h),\zeta_h)$ in $ \bV_h^k\times(\rQ_h^{k-1}\times \rR_h^{k-1})\times \rW_h^k$,
\begin{align*}
   \langle \bA_h\bu_h,\bv_h\rangle = a_h(\bu_h,\bv_h)&=\sum_{K\in\Om_h} a_h^K(\bu_h,\bv_h),\\
   \langle \bB_{1h}\bv_h,(q_h,\psi_h)\rangle =b_{1h}(\bv_h,(q_h,\psi_h))&=\sum_{K\in\Om_h} b_{1h}^K(\bv_h,(q_h,\psi_h)) , \\
\langle \bC_{1h}(p_h,\varphi_h),(q_h,\psi_h)\rangle =c_{1h}((p_h,\varphi_h),(q_h,\psi_h))&=\sum_{F\in\Sigma_h} c_{1h}^F((p_h,\varphi_h),(q_h,\psi_h)), \\
\langle \bB_{2h}(q_h,\psi_h),\zeta_h\rangle =b_{2h}((q_h,\psi_h),\zeta_h)&=\sum_{F\in\Sigma_h} b_{2h}^F((q_h,\psi_h),\zeta_h), \\  \langle \bB_{3h}(q_h,\psi_h),\zeta_h\rangle =b_{3h}((q_h,\psi_h),\zeta_h)&=\sum_{F\in\Sigma_h} b_{3h}^F((q_h,\psi_h),\zeta_h),\\ 
\langle \bC_{2h} w_h,\zeta_h\rangle =c_{2h}(w_h,\zeta_h)&=\sum_{F\in\Sigma_h} c_{2h}^F(w_h,\zeta_h),
\end{align*}
where the local discrete bilinear forms are defined as follows:
\begin{subequations}\label{discrete_forms}
\begin{align}
a_h^K(\bu_h,\bv_h) &:=  a_h^{0,K}(\bu_h,\bv_h) + a_h^{\nabla,K}(\bu_h,\bv_h)+a_h^{\Sigma,K}(\bu_h,\bv_h)\notag \\
&:= \frac{\rho_f}{\tau}\int_K \bPi_k^{0,K}\bu_h\cdot \bPi_k^{0,K}\bv_h +S_1^{0,K}((1-\bPi_k^{0,K})\bu_h,(1-\bPi_k^{0,K})\bv_h)\notag \\
&\quad +\, \mu \int_K \bbPi_{k-1}^{0,K}(\bnabla\bu_h):\bbPi_{k-1}^{0,K}(\bnabla\bv_h) + S_1^{\nabla,K}((1-\bPi_k^{\nabla,K})\bu_h,(1-\bPi_k^{\nabla,K})\bv_h)\notag \\
&\quad +\, \gamma \sum_{F\in\partial K\cap\Sigma_h}\int_F (\bPi_{{k}}^{0,F}\bu_h\times \bn_\Sigma) \cdot (\bPi_{{k}}^{0,F}\bv_h\times \bn_\Sigma), \\ 
b_{1h}^K(\bv_h,(q_h,\psi_h)) &:= b_{1h}^{\vdiv,K}(\bv_h,(q_h,\psi_h)) + b_{1h}^{\Sigma,K}(\bv_h,(q_h,\psi_h))\notag \\
&:= -\, \int_K q_h\vdiv\bv_h+ \sum_{F\in\partial K\cap\Sigma_h}\int_F \Pi_{k-1}^{0,F}\psi_h \,\bPi_{k}^{0,F}\bv_h\cdot\bn_\Sigma, \\ 
c_{1h}^F((p_h,\varphi_h),(q_h,\psi_h)) &:= {c_{1h}^{0,F}((p_h,\varphi_h),(q_h,\psi_h)) + c_{1h}^{\nabla,F}((p_h,\varphi_h),(q_h,\psi_h))}\notag \\
&:= \frac{c_0}{\tau} \int_F \Pi_{k-1}^{0,F}\varphi_h\Pi_{k-1}^{0,F}\psi_h + S_2^{0,F}((1-\Pi_{k-1}^{0,F})\varphi_h,(1-\Pi_{k-1}^{0,F})\psi_h)\notag \\
&\quad +\,  \kappa \hspace{-0.05cm} \int_F \hspace{-0.05cm} \bPi_{k-2}^{0,F}(\nabla_\Sigma\varphi_h)\cdot \bPi_{k-2}^{0,F}(\nabla_\Sigma\psi_h) \hspace{-0.05cm} + \hspace{-0.05cm} S_2^{\nabla,F}((1-\Pi_{k-1}^{\nabla,F})\varphi_h,(1-\Pi_{k-1}^{\nabla,F})\psi_h), \\
b_{2h}^F((q_h,\psi_h),\zeta_h) &:= -\frac{\alpha}{\tau} \int_F \bPi_{{k-2}}^{0,F}(\nabla\psi_h)\cdot\bPi_{{k-1}}^{0,F}(\nabla\zeta_h), \\ 
b_{3h}^F((q_h,\psi_h),\zeta_h) &:= - \frac{1}{\tau}\int_F\Pi_{{k}}^{0,F}\zeta_h \,\Pi_{{k-1}}^{0,F}\psi_h, \\
c_{2h}^F(w_h,\zeta_h) &:= {c_{2h}^{0,F}(w_h,\zeta_h) + c_{2h}^{\nabla^2,F}(w_h,\zeta_h)}\notag \\
&:= \frac{\rho_p}{\tau^3}\int_F \Pi_{{k}}^{0,F}w_h\Pi_{{k}}^{0,F}\zeta_h + S_3^{0,F}((1-\Pi_k^{0,F})w_h,(1-\Pi_k^{0,F})\zeta_h)\notag \\
&\quad +\, \hspace{-0.15cm}\frac{D}{\tau}\hspace{-0.1cm} \int_F \hspace{-0.1cm} \bbPi_{k-2}^{0,F}(\nabla^2_\Sigma w_h)\hspace{-0.05cm}: \hspace{-0.05cm}\bbPi_{k-2}^{0,F}(\nabla^2 \zeta_h)\hspace{-0.05cm} + \hspace{-0.05cm}S_3^{\nabla^2,F}((1-\Pi_k^{\nabla^2,F})w_h,(1-\Pi_k^{\nabla^2_\Sigma,F})\zeta_h).
\end{align}
\end{subequations}
We assume that the stabilisation terms $S_1^{0,K},S_1^{\nabla,K}:\bV_h^k(K)\times \bV_h^k(K)\to\mathbb{R}$, $S_2^{\nabla,F}:\rR_h^{k-1}(F)\times \rR_h^{k-1}(F) \to \mathbb{R}$, and $S_3^{\nabla^2,F}:\rW_h^k(K)\times \rW_h^k(K)\to\mathbb{R}$ are positive semi-definite inner products for any polyhedron $K\in\Om_h$ and face $F\in\Sigma_h$, and there exist positive constants $C_{s0}$, $C_{s1}$, $C_{s2}$, $C_{s3}$, $C_{s4}$, $C_{s5}$ (independent of $h$ and $K, F$) such that
\begin{subequations}
\begin{gather}
    C_{s0}^{-1}\frac{\rho_f}{\tau}\|\bv_h\|^2_{K}\leq S^{0,K}_{1}(\bv_h,\bv_h)\leq C_{s0}\frac{\rho_f}{\tau}\|\bv_h\|^2_{K},\label{s0}\\
     C_{s1}^{-1}\mu|\bv_h|^2_{1,K}\leq S^{\nabla,K}_{1}(\bv_h,\bv_h)\leq C_{s1}\mu|\bv_h|^2_{1,K},\label{s1}\\
     C_{s2}^{-1}\frac{c_0}{\tau}\|\psi_h\|^2_{F}\leq S^{0,F}_{2}(\psi_h,\psi_h)\leq C_{s2}\frac{c_0}{\tau}\|\psi_h\|^2_{F},\label{s2}\\
      C_{s3}^{-1}\kappa|\psi_h|^2_{1,F}\leq S^{\nabla,F}_{2}(\psi_h,\psi_h)\leq C_{s3}\kappa|\psi_h|^2_{1,F},\label{s3}\\
      C_{s4}^{-1}\frac{\rho_p}{\tau^3}\|\zeta_h\|^2_{F}\leq S^{0,F}_{3}(\zeta_h,\zeta_h)\leq C_{s4}\frac{\rho_p}{\tau^3}\|\zeta_h\|^2_{F},\label{s4}\\
      C_{s5}^{-1}\frac{D}{\tau}|\zeta_h|^2_{2,F}\leq S^{\nabla^2,F}_{3}(\zeta_h,\zeta_h)\leq C_{s5}\frac{D}{\tau}|\zeta_h|^2_{2,F}.\label{s5}
\end{gather}\end{subequations}
\textit{Discrete problem}. Find $(u_h,(p_h,\varphi_h),w_h)\in \bV_h^k\times(\rQ_h^{k-1}\times \rR_h^{k-1})\times \rW_h^k$ such that
\begin{subequations}\label{dis-problem}
\begin{align}
        a_{h}(\bu_h,\bv_h)+b_{1h}(\bv_h,(p_h,\varphi_h)) &= F_h(\bv_h), \\
        b_{1h}(\bu_h,(q_h,\psi_h))-c_{1h}((p_h,\varphi_h),(q_h,\psi_h))+(b_{2h}+b_{3h})((q_h,\psi_h),w_h) &= G_h((q_h,\psi_h)),\\
        (b_{2h}- b_{3h})((p_h,\varphi_h),\zeta_h)+c_{2h}(w_h,\zeta_h) &= M_h(\zeta_h) , 
\end{align}
\end{subequations}
for all $\bv_h\in\bV_h^k$, $(q_h,\psi_h)\in \rQ_h^{k-1}\times \rR_h^{k-1}$, $\zeta_h \in \rW_h^k$, whith the discrete right-hand sides:
\begin{align*}
    F_h(\bv_h)=\sum_{K\in \Omega_h}F_h^K(\bv_h)&:=\sum_{K\in \Omega_h}(\ff,\bPi_k^{0,K}\bv_h)_K,\\
    G_h((q_h,\psi_h))=\sum_{F\in\Sigma_h}G_h^F((q_h,\psi_h))&:=-\sum_{F\in \Sigma_h}(g,\Pi_{k-1}^{0,F}\psi_h)_F,\\
    M_h(\zeta_h)=\sum_{F\in\Sigma_h}M_h^F(\zeta_h)&:=\frac{1}{\tau}\sum_{F\in\Sigma_h}(m,\Pi_{k}^{0,F}\zeta_h)_F.
\end{align*}
\section{Unique solvability of the discrete problem}\label{sec:wellp-h} 
Elementary algebra with Cauchy--Schwarz inequalities, the stability estimates of projection operators, the bounds \eqref{s0}-\eqref{s3}, the trace inequality, provide the continuity (of all) and coercivity (of some) of the discrete operators
\begin{subequations}
\begin{align}
|\langle \bA_h\bu_h,\bv_h \rangle| & \leq \max\{\frac{\rho_f}{\tau}(1+C_{s0}),\mu (1+C_{s1}),{\gamma C_T^2}\} \|\bu_h\|_{1,\Omega}\|\bv_h\|_{1,\Omega},\label{ah:bound}\\
\langle \bA_h\bv_h,\bv_h \rangle & \geq \min\{\frac{\rho_f}{\tau}C_{s0}^{-1},\mu C_{s1}^{-1}\} \|\bv_h\|^2_{1,\Omega},\label{ah:coer}\\
\langle \bC_{1h}(p_h,\varphi_h),(q_h,\psi_h)\rangle & \leq \max\{\frac{c_0}{\tau}(1+C_{s2}),\kappa (1+C_{s3})\}\|(p_h,\varphi_h)\|\|(q_h,\psi_h)\|, \label{c1h:bound}\\
\langle \bC_{1h}(q_h,\psi_h),(q_h,\psi_h)\rangle & \geq \min\{\frac{c_0}{\tau}C_{s2}^{-1},\kappa C_{s3}^{-1}\}\|\psi_h\|_{1,\Sigma}^2\geq 0  
, \label{c1h:pos}\\
|\langle \bC_{2h} w_h,\zeta_h\rangle| & \leq \max \{\frac{\rho_p}{\tau^3}(1+C_{s4}),\frac{D}{\tau} (1+C_{s5})\} \|w_h\|_{2,\Sigma}\|\zeta_h\|_{2,\Sigma},\label{c2h:bound}\\
\langle \bC_{2h} \zeta_h,\zeta_h\rangle & \geq \min\{\frac{\rho_p}{\tau^3}C_{s4}^{-1},\frac{D}{\tau}C_{s5}^{-1}\}(\|\zeta_h\|^2_{0,\Sigma}+ |\zeta_h|^2_{2,\Sigma}),\label{c2h:coer}\\
\langle \bB_{1h}\bv_h,(q_h,\psi_h)\rangle & \leq \max\{1,{C_T }\}\|\bv_h\|_{1,\Omega}\|(q_h,\psi_h)\|,\label{b1h:bound}\\
\langle \bB_{2h}(q_h,\psi_h),\zeta_h\rangle & \leq \frac{\alpha}{\tau} \|(q_h,\psi_h)\|\|\zeta_h\|_{2,\Sigma},\label{b2h:bound}\\
 \langle \bB_{3h}(q_h,\psi_h),\zeta_h\rangle & \leq \frac{1}{\tau} \|(q_h,\psi_h)\|\|\zeta_h\|_{2,\Sigma},.\label{b3h:bound}
\end{align}\end{subequations}
for all $\bu_h,\bv_h \in \bV_h^k$, $(p_h,\varphi_h),(q_h,\psi_h) \in \rQ_h^{k-1}\times \rR_h^{k-1}$, $w_h,\zeta_h \in \rW_h^k$. 

The following proposition concerns classical polynomial approximation theory. It is formulated for scalar functions, but also applies to vectorial functions.
\begin{proposition}[polynomial approximation \cite{brenner2008mathematical}] \label{prop:est-poly}
    Given $K\in \Omega_h$ and $F\in \Sigma_h$, assume that $v\in  \rH^{{\overline{s}}}(K)$ and $\psi \in \rH^{{\overline{s}}}(F)$, with  $1 \leq {\overline{s}} \leq k+1$. Then, there exist $v_\pi \in \rP_k(K)$, $v^F_\pi:=v_\pi|_F,\psi_\pi \in \rP_k(F)$, and a positive constant $C_{\mathrm{apx}}$ (that depends exclusively  on $\rho$ from \ref{M1}) such that for $0\leq d\leq {\overline{s}}$ the following estimates hold
	\begin{align*}
        |v - v_\pi|_{d,K} &\leq C_{\mathrm{apx}} h_K^{{\overline{s}}-d}|v|_{{\overline{s}},K},\\
        |v - v^F_\pi|_{d,F} &\leq C_{\mathrm{apx}} h_K^{{\overline{s}}-d-\frac{1}{2}}|v|_{{\overline{s}},K},\\
        |\psi - \psi_\pi|_{d,F} &\leq C_{\mathrm{apx}} h_F^{{\overline{s}}-d}|v|_{{\overline{s}},F}.
	\end{align*} 
\end{proposition}

Next, we focus on deriving a discrete inf-sup condition for the operator $\bB_{1h}$. We start by introducing a quasi-interpolation operator for the Stokes problem as follows.
\begin{proposition}[{Stokes} quasi-interpolator \cite{dkrr25}]\label{prop:int}
   Let $\bv\in \bH^1_\star(\Omega)\cap \bH^{s+1}(\Om)$  for $0\leq s \leq k$. Under the mesh assumptions, there exist $\mathbf{I}_h^\mathrm{S}\bv\in\bV_h^k$ and the positive constant $C_{\mathbf{I}^{\mathrm{S}}}$ (independent of $h$) such that 
   \[\|\bv-\mathbf{I}_h^\mathrm{S}\bv\|_{0,K}+h_K|\bv-\mathbf{I}_h^\mathrm{S}\bv|_{1,K}\leq  C_{\mathbf{I}^{\mathrm{S}}} h_K^{s+1}|\bv|_{s+1,D(K)},\]
   where $D(K)$ denotes the union of the polyhedral elements in $\Om_h$ intersecting $K$, for all $K\in \Omega_h$. 
\end{proposition}

The construction of a novel minimal regularity Stokes Fortin-like interpolation is provided next, which in turn enables the proof of the discrete inf-sup condition via the orthogonality property.
\begin{proposition}[{Stokes} Fortin operator]\label{prop:fortin}
    For $k\geq 2$, there exists  a linear operator ${\bar{\mathbf{I}}_h^{\mathrm{S}}}:\bH^1_\star(\Omega)\to\bV_h^k$ satisfying the following  properties: 
 \begin{subequations}   \begin{enumerate}
        \item Orthogonality: For $\bv\in\bH^1_\star(\Omega)$,
        \begin{align}
        \int_\Omega\vdiv(\bv-{\bar{\mathbf{I}}_h^{\mathrm{S}}}\bv)q_h=0\quad\forall\;q_h\in Q_h^{k-1}. \label{Fortin:ortho}
        \end{align}
        \item Boundedness: There exists a positive constant ${\overline{C}_{\bar{\mathbf{I}}^{\mathrm{S}}}}$, independent of $h$, such that 
        \begin{align}
        \|{\bar{\mathbf{I}}_h^{\mathrm{S}}}\bv\|_{1,K}\leq {\overline{C}_{\bar{\mathbf{I}}^{\mathrm{S}}}}\|\bv\|_{1,K}\quad\forall\;\bv\in\bH^1_\star(\Omega).\label{Fortin:bound}
        \end{align}
        \item Error estimate: For $\bv\in \bH^1_\star(\Omega)\cap\bH^{{s+1}}(\Omega)$ and $0\leq s \leq k$, there exists a positive constant ${C_{\bar{\mathbf{I}}^{\mathrm{S}}}}$ independent of $h$ such that for all $K\in \Omega_h$
        \begin{align}
        \|\bv-{\bar{\mathbf{I}}_h^{\mathrm{S}}}\bv\|_{0,K}+h_K|\bv-{\bar{\mathbf{I}}_h^{\mathrm{S}}}\bv|_{1,K}\leq {C_{\bar{\mathbf{I}}^{\mathrm{S}}}} h^{s+1}_K|\bv|_{s+1,K}.\label{Fortin:est}
        \end{align}
    \end{enumerate}\end{subequations}
\end{proposition}
\begin{proof}
\textit{Step 1}. Given a $\bv\in\bH^1_\star(\Omega)$, we first aim to construct $\overline{\bv}_h\in\bV_h^{k}$ such that it is orthogonal to piecewise constants, that is, 
\begin{align}
(\vdiv(\bv-\overline{\bv}_h),q_0)_\Omega=0\quad\forall\;q_0\in\rP_0(\Om_h).\label{est:vh}
\end{align}
Let $\bPi_1\bv = \mathbf{I}_h^\mathrm{S}\bv$ for the interpolation $\mathbf{I}_h^\mathrm{S}\bv\in\bV_h^k$ from Proposition~\ref{prop:int}. Define $\bPi_2\bv\in \bV_h^{k}(K)$ for all $K\in\Om_h$ through the \acp{dof} as  edge moments $(\mathbf{D1}\bPi_2\bv)$ and the  cell moments $(\mathbf{D3}\bPi_2\bv)$ to be zero and face moments $(\mathbf{D2}\bPi_2\bv)$ to be equal to
\begin{align}
    \frac{1}{h_F}\int_F(\bPi_2\bv\cdot\bn_K^F) = \frac{1}{h_F}\int_F(\bv\cdot\bn_K^F) \quad\text{and}\quad \frac{1}{h_F}\int_F (\bPi_2\bv)|_F = \frac{1}{h_F}\int_F \bv|_F.\label{5.1.15}
\end{align}
Note from the definition of $\bV_h^{k}$, that $\bPi_2\bv$ locally solves the following Stokes problem:
 \begin{align*}
     -\Delta\bPi_2\bv-\nabla s = \0\;\text{for some}\;s\in \rL^2_{0}(K),&\quad
     \vdiv(\bPi_2\bv) = \Pi_0^{0,K}\vdiv(\bv)\;\text{in}\;K, \quad 
     \bPi_2\bv = \bw_2\;\text{on}\;\partial K,
 \end{align*}
  where $\bw_2\in \bB^{k}(\partial K)$  is defined uniquely through above \acp{dof}. The continuous dependence of the data for Stokes  with the dependence of a constant on mesh regularity parameter $\rho$ and dimension $n$, and not on domain $K$ specifically for star-shaped domains follows from  \cite[Exercise~III.3.5 and Lemma~III.3.1]{galdi2011introduction}. This implies that 
 \begin{align}
     \|\bPi_2\bv\|_{1,K}\lesssim \|\Pi_0^{0,K}\vdiv(\bv)\|_{0,K}+\|\bw_2\|_{1/2,\partial K}\lesssim |\bv|_{1,K}+h_K^{-1/2}\|\bw_2\|_{\partial K}\label{5.1.16}
 \end{align}
 with the $\rL^2$ stability of $\Pi_0^{0,K}$ and the inverse estimate for conforming VE functions in  $\bB^k(F)$ from \cite{chen2018some} in the last step. If $\varphi_j$  denotes the basis of $\bB^k(F)$ of dimension $N_{\text{dof}}^{k,F}$ for $F\in\partial K$ and for $j=1,\dots,N_{\text{dof}}^{k,F}$, then we can write $\bw_2$ in terms of basis functions $\{\varphi_j\}$ and the scaling $\|\varphi_j\|_F\approx h_F$ for 3D from \cite{sutton2017virtual} lead to
 \begin{align}
 \|\bw_2\|_F \lesssim \bigg|\frac{1}{h_F}\int_F (\bPi_2\bv)\bigg|h_F=\bigg|\frac{1}{h_F}\int_F \bv\bigg|h_F\leq \|\bv\|_F\lesssim h_F^{-1/2}\|\bv\|_{0,K}+h_F^{1/2}|\bv|_{1,K}\label{5.1.17}
 \end{align}
with \eqref{5.1.15} in the second step, the Cauchy--Schwarz inequality in the third step and the trace inequality in the last step. The combination of \eqref{5.1.16}-\eqref{5.1.17} results in 
\begin{align}
\|\bPi_2\bv\|_{1,K}\lesssim h_K^{-1}\|\bv\|_{0,K}+|\bv|_{1,K}. \label{stab:pi2}
\end{align}
Then define $\overline{\bv}_h:=\bPi_1\bv+\bPi_2(\bv-\bPi_1\bv)$. Note that the definition of $\bPi_2\bv$ provides $\int_F\overline{\bv}_h\cdot\bn_K^F=\int_F\bv\cdot\bn_K^F$ and an integration by parts proves \eqref{est:vh}. The triangle inequality, the boundedness of $\bPi_2\bv$ and  Proposition~\ref{prop:int} show
\begin{align}
    \|\bv-\overline{\bv}_h\|_{1,K}\lesssim \|\bv-\bPi_1\bv\|_{1,K}+ h_K^{-1}\|\bv-\bPi_1\bv\|_{0,K}\lesssim h_K^s\|\bv\|_{s+1,D(K)}.\label{5.1.19}
\end{align}
\textit{Step 2}. Given a $\bv\in\bH^1_\star(\Omega)$ and $\overline{\bv}_h\in\bV_h^{k}$, the next aim is construct a bubble function $\widetilde{\bv}_h\in\bV_h^k$. We set the \acp{dof} of $\widetilde{\bv}_h$ as zero for all edge, face and cell moments except the ones with respect to the divergence part. In particular, 
\begin{align}
\int_K\vdiv(\widetilde{\bv}_h)\chi_{k-1} = \int_K\vdiv(\bv-\overline{\bv}_h)\chi_{k-1}\quad\text{for all}\; \chi_{k-1}\in\rP_{k-1}(K)\setminus\mathbb{R}.\label{5.1.20}
\end{align}
In other words, the definition of the space $\bV_h^k(K)$ implies that $\widetilde{\bv}_h$ locally solves the following Stokes problem:
 \begin{align*}
     -\Delta\widetilde{\bv}_h-\nabla \widetilde{s} = \0\quad\text{for some}\;\widetilde{s}\in \rL^2(K),&\quad
     \vdiv(\widetilde{\bv}_h) = \Pi_{k-1}^{0,K}\vdiv(\bv-\overline{\bv}_h)\quad\text{in}\;K,\quad 
     \widetilde{\bv}_h = \0\quad\text{on}\;\partial K.
 \end{align*}
An analogous argument as in \cite{dkrr25} shows  the existence of  a constant $C$ independent of the domain $K$ such that 
 \begin{equation}
 \|\widetilde{\bv}_h\|_{1,K}\leq C\|\Pi_{k-1}^{0,K}\vdiv(\bv-\overline{\bv}_h)\|_{0,K}\leq C\|\vdiv(\bv-\overline{\bv}_h)\|_{0,K}\leq C\|\bv-\overline{\bv}_h\|_{1,K}. 
 \label{est:vh_tilde}
 \end{equation}
\textit{Step 3 (proof of 1 and 2)}. Finally, given $\bv\in\bH^1_\star(\Omega)$, define ${\bar{\mathbf{I}}_h^{\mathrm{S}}}\bv=\overline{\bv}_h+\widetilde{\bv}_h\in\bV_h^k$. This with \eqref{5.1.20} for $q_h\in\rP_{k-1}(\Om_h)\setminus\rP_0(\Om_h)$ and \eqref{est:vh} for $q_h\in\rP_0(\Om_h)$ conclude the required orthogonality for $q_h\in Q_h^{k-1}$. To the end, the triangle inequality {$\|{\bar{\mathbf{I}}_h^{\mathrm{S}}}\bv\|_{1,K}\leq \|\overline{\bv}_h-\bv\|_{1,K}+\|\bv\|_{1,K}+\|\widetilde{\bv}_h\|_{1,K}$} followed by estimates \eqref{5.1.19} and \eqref{est:vh_tilde} for all $K\in\Om_h$ conclude the proof of boundedness of ${\bar{\mathbf{I}}_h^{\mathrm{S}}}\bv$.\\

\noindent\textit{Step 4 (proof of 3)}. For the $\bH^1$-error estimate, the triangle inequality leads to
\begin{align}
{\|\bv-{\bar{\mathbf{I}}_h^{\mathrm{S}}}\bv\|_{1,K}\leq \|\bv-\overline{\bv}_h\|_{1,K}+\|\widetilde{\bv}_h\|_{1,K}\leq (1+C) \|\bv-\overline{\bv}_h\|_{1,K}\lesssim h^s\|\bv\|_{s+1,K}}\label{5.1.24}
\end{align}
with \eqref{est:vh_tilde} and \eqref{5.1.19} in the last two inequalities. For the $\bL^2$-error estimate, observe that the Poincar\'e inequality can be applied to $\widetilde{\bv}_h$ on each $K\in\Om_h$ since it is zero along the boundary $\partial K$, and hence the estimate \eqref{5.1.24} shows that 
\begin{align}
    {\|\widetilde{\bv}_h\|_{0,K}\lesssim h|\widetilde{\bv}_h|_{1,K}\lesssim h^{s+1}\|\bv\|_{s+1,K}.}\label{vhtilde:l2}
\end{align}
The triangle inequality and the Poincar\'e--Friedrichs inequality for $\bv-\bPi_2\bv$ since $\int_{\partial  K} \bv-\bPi_2\bv=\0$ from \cite{huang2023some}, lead to
\begin{align*}
    \|\bPi_2\bv\|_{0,K}&\lesssim h_K|\bv-\bPi_2\bv|_{1,K}+\|\bv\|_{0,K}\leq h_K(|\bv|_{1,K}+|\bPi_2\bv|_{1,K})+\|\bv\|_{0,K} 
    \lesssim h_K|\bv|_{1,K}+\|\bv\|_{0,K},
\end{align*}
with again the triangle inequality in the second step and \eqref{stab:pi2} in the last step. This bound, the triangle inequality, and Proposition~\ref{prop:int} prove that
\[
    \|\bv-\overline{\bv}_h\|_{0,K}\lesssim \|\bv-\bPi_1\bv\|_{0,K}+h_K|\bv-\bPi_1\bv|_{1,K}\lesssim h_K^{s+1}\|\bv\|_{s+1,D(K)}.
\]
This and \eqref{vhtilde:l2} in the triangle inequality ${\|\bv-{\bar{\mathbf{I}}_h^{\mathrm{S}}}\bv\|_{0,K}\leq \|\bv-\overline{\bv}_h\|_{0,K}+\|\widetilde{\bv}_h\|_{0,K}}$ conclude the proof of $L^2$-error estimate.
\end{proof}

{We also introduce interpolation estimates for the poroelastic plate in the following result.}
\begin{proposition}[poroelastic plate interpolation \cite{khot23}]\label{prop:int_plate}
For all $F\in \Sigma_h$. Let $\psi\in\rH^1_0(\Sigma)\cap\rH^{r_1}(\Sigma)$ and $\zeta\in\rH^2_0(\Sigma)\cap\rH^{r_2}(\Sigma)$ with $1\leq r_1 \leq k$ and $2\leq r_2 \leq k+1$. Then, there exist interpolation operators $\mathrm{I}_{1,h}^{\mathrm{P}}:\rH^1_0(\Sigma)\cap\rH^{r_1}(\Sigma)\rightarrow \rR_h^{k-1}$ and $\mathrm{I}_{2,h}^{\mathrm{P}}:\rH^2_0(\Sigma)\cap\rH^{r_2}(\Sigma)\rightarrow \rW_h^k$, such that
\begin{align*}
|\psi-\mathrm{I}_{1,h}^{\mathrm{P}}\psi|_{j_1,F} &\leq C_{\mathrm{I}^{\mathrm{P}}_{1}} h^{r_1-j_1}_F |\psi|_{r_1,F}, \quad 0\leq j_1 \leq 1, \\
|\zeta-\mathrm{I}_{2,h}^{\mathrm{P}}\zeta|_{j_2,F} &\leq C_{\mathrm{I}^{\mathrm{P}}_{2}} h^{r_2-j_2}_F |\zeta|_{r_2,F}, \quad 0\leq j_2 \leq 2.
\end{align*}
\end{proposition}
Finally, we establish the discrete inf-sup for the operator $\bB_{1h}$.
\begin{theorem}[discrete inf-sup]\label{dis-inf-sup}
    For sufficiently small mesh-size $h$, there exists a positive constant $\beta'$ independent of $h$ such that
    \[\sup_{\bv_h\in \bV_h^k\setminus\{\cero\}}\frac{b_{1h}(\bv_h,(q_h,\psi_h))}{\|\bv_h\|_{1,\Om}}\geq \beta'\|(q_h,\psi_h)\|\qquad\forall(q_h,\psi_h)\in \rQ_h^{k-1}\times \rR_h^{k-1}.\]
\end{theorem}
\begin{proof}
    Following the proof of Lemma~\ref{lem:inf-sup}, given $\psi_h\in R_h^{k-1}$, $q_h\in Q_h^{k-1}$ and $c_h:=\frac{1}{|\Omega|}\int_\Sigma\psi_h$, there exists $\widetilde{\bv}_1 \in \bH^1_{\Gamma^{\bu}\cup\Sigma}(\Omega)$ such that 
    \begin{equation}
        \vdiv \widetilde{\bv}_1 = - q_h - c_h\quad \text{in $\Omega$},  \quad \widetilde{\bv}_1 \cdot \bn_\Sigma = 0\quad \text{on $\Sigma$},\quad\|\widetilde{\bv}_1\|_{1,\Omega} \leq \widetilde{\beta}' (\|q_h\|_{0,\Omega}+\|\psi_h\|_{1,\Sigma}),\label{eq:v1h}
    \end{equation}
    and $\widetilde{\bv}_2  \in \bH_{\partial\Omega  \setminus \Sigma}^1(\Omega)\subset \bH^1_\star(\Omega)$ such that
    \begin{align*} 
        - \bDelta \widetilde{\bv}_2 +\nabla q_{2h}  = \cero, \quad  \vdiv\widetilde{\bv}_2 = c_h \quad \text{in $\Omega$},\quad \widetilde{\bv}_2 = \psi_h \bn_\Sigma\quad \text{on $\Sigma$}, \quad \widetilde{\bv}_2 = \cero\quad \text{on $\partial\Omega\setminus\Sigma$}
    \end{align*}
    satisfying  the properties
    \begin{equation}
        \widetilde{\bv}_2\cdot\bn_\Sigma = \psi_h \quad \text{on $\Sigma$} \quad\text{and}\quad  \|\widetilde{\bv}_2\|_{1,\Omega} \leq \widehat{\beta}' \|\psi_h\|_{1,\Sigma}.\label{eq:v2h}
    \end{equation}
   Next, we define $\widetilde{\bv} := \widetilde{\bv}_1 + \widetilde{\bv}_2$, and  note that $\overline{\mathbf{I}}_{h}^{\mathrm{S}} \widetilde{\bv} \in \bV_h^k$ and $\vdiv(\overline{\mathbf{I}}_{h}^{\mathrm{S}} \widetilde{\bv}) = \Pi_{k-1}^{0}(\vdiv \widetilde{\bv}) = q_h$. Thus, the combination \eqref{eq:v1h}-\eqref{eq:v2h} leads to
    \begin{align}\label{first_part_disc_inf-sup}
        \sup_{\bv_h\in \bV_h^k\setminus\{\cero\}}\!\!\!\frac{\langle\bB_{1h}\bv_h,(q_h,\psi_h)\rangle  }{\|\bv_h\|_{1,\Omega}} &\geq \frac{\langle\bB_{1h}\overline{\mathbf{I}}_{h}^{\mathrm{S}} \widetilde{\bv},(q_h,\psi_h)\rangle  }{\|\overline{\mathbf{I}}_{h}^{\mathrm{S}} \widetilde{\bv}\|_{1,\Omega}}  \notag\\
        & \geq \frac{1}{\overline{C}_{\bar{\mathbf{I}}^{\mathrm{S}}}\|\tilde{\bv}\|_{1,\Omega}}\left(\|q_h\|^2_{0,\Omega}+\displaystyle{\sum_{K\in\Omega_h}\sum_{F\in\partial K\cap\Sigma_h}\int_F \Pi_{k-1}^{0,F}\psi_h \,\bPi_{k}^{0,F}(\overline{\mathbf{I}}_{h}^{\mathrm{S}} \widetilde{\bv})\cdot\bn_\Sigma}\right).
    \end{align}
    Then, we rewrite the remaining term as follows 
    \begin{align*}
     \sum_{K\in\Omega_h}\sum_{F\in\partial K\cap\Sigma_h}\int_F \Pi_{k-1}^{0,F}\psi_h \,\bPi_{k}^{0,F}(\overline{\mathbf{I}}_{h}^{\mathrm{S}} \widetilde{\bv})\cdot\bn_\Sigma &= \sum_{K\in\Omega_h} \left(-\sum_{F\in\partial K\cap\Sigma_h} \int_F \Pi_{k-1}^{0,F}\psi_h \,\bPi_{k}^{0,F}(\widetilde{\bv}-\overline{\mathbf{I}}_{h}^{\mathrm{S}} \widetilde{\bv})\cdot\bn_\Sigma \right. \notag \\
        &\quad - \sum_{F\in\partial K\cap\Sigma_h}\int_F \left(\psi_h - \Pi_{k-1}^{0,F}\psi_h\right)\, \bPi_{k}^{0,F}\widetilde{\bv}\cdot\bn_\Sigma \notag\\
        &\quad - \sum_{F\in\partial K\cap\Sigma_h}\int_F \psi_h \left(\widetilde{\bv}-\bPi_{k}^{0,F}\widetilde{\bv}\right)\cdot \bn_\Sigma \notag \\
        &\quad \left. + \sum_{F\in\partial K\cap\Sigma_h}\int_F \psi_h\, \widetilde{\bv}\cdot\bn_\Sigma \right) \notag \\
        &=: \sum_{K\in\Omega_h} \left(T_1 + T_2 + T_3 + T_4\right).
    \end{align*}
    For the first term $T_1$, the Cauchy--Schwarz inequality provides 
    \begin{align}\label{eq:T1}
        -\sum_{K\in\Omega_h} T_1 &\leq \sum_{K\in\Omega_h}\sum_{F\in\partial K\cap\Sigma_h}\|\Pi_{k-1}^{0,F}\psi_h\|_{0,F} \|\bPi_{k}^{0,F}(\widetilde{\bv}-\overline{\mathbf{I}}_{h}^{\mathrm{S}} \widetilde{\bv})\cdot\bn_\Sigma\|_{0,F}\nonumber\\
        &\leq \sum_{K\in\Omega_h}\sum_{F\in\partial K\cap\Sigma_h}\|\psi_h\|_{0,F} \|\widetilde{\bv}-\overline{\mathbf{I}}_{h}^{\mathrm{S}} \widetilde{\bv}\|_{0,F} \nonumber \\
        &\leq  C_T \sum_{K\in\Omega_h}\sum_{F\in\partial K\cap\Sigma_h}\|\psi_h\|_{1,F}\left(h_K^{-\frac{1}{2}}\|\widetilde{\bv}-\overline{\mathbf{I}}_{h}^{\mathrm{S}} \widetilde{\bv}\|_{0,K}+h_K^{\frac{1}{2}}|\widetilde{\bv}-\overline{\mathbf{I}}_{h}^{\mathrm{S}} \widetilde{\bv}|_{1,K} \right)\nonumber \\ 
        &\leq C_T C_{\bar{\mathbf{I}}^{\mathrm{S}}} \sum_{K\in\Omega_h}\sum_{F\in\partial K\cap\Sigma_h}\|\psi_h\|_{1,F}h^{\frac{1}{2}}_K\|\widetilde{\bv}\|_{1,K}\nonumber \\
        &\leq C_1 h^{\frac{1}{2}} \left(\|q_h\|_{0,\Om}\|\psi_h\|_{1,\Sigma} + \|\psi_h\|^2_{1,\Sigma}\right)
    \end{align}
    with the continuity of $\Pi_{k-1}^{0,F}$, and $\bPi_{k}^{0,F}$ in the $\|\cdot\|_{0,F}$ norm for the second step, the trace inequality in the third step, \eqref{Fortin:est} in the fourth step, and \eqref{eq:v1h}-\eqref{eq:v2h} in the last step with $C_1:=C_T
    C_{\bar{\mathbf{I}}^{\mathrm{S}}}(\widetilde{\beta}'+\widehat{\beta}') $. For the second term $T_2$, the Cauchy--Schwarz inequality  lead to
    \begin{align*}
        -\sum_{K\in\Omega_h}T_2 &\leq  \sum_{K\in\Omega_h}\sum_{F\in\partial K\cap\Sigma_h}\|\psi_h - \Pi_{k-1}^{0,F}\psi_h\|_{0,F} \|\bPi_{k}^{0,F}\widetilde{\bv}\cdot\bn_\Sigma\|_{0,F}\nonumber \\
        &\leq C_{\text{apx}} \sum_{K\in\Omega_h}\sum_{F\in\partial K\cap\Sigma_h} h_F \|\psi_h\|_{1,F}\|\widetilde{\bv}\cdot\bn_\Sigma\|_{0,F} \nonumber \\
        &\leq C_{\text{apx}}  h \|\psi_h\|^2_{1,\Sigma},
    \end{align*}
    where we used Proposition~\ref{prop:est-poly} and the continuity of $\bPi_{k}^{0,F}$ in the $\|\cdot\|_{0,F}$ norm for the second inequality, and the relation $\widetilde{\bv}\cdot\bn_\Sigma=\psi_h$ together with \eqref{eq:duality_equalities} for the last inequality. For the third term $T_3$, similar arguments with $C_2:=C_{\text{apx}}(\widetilde{\beta}'+\widehat{\beta}')$ yield 
    \begin{align*}
        -\sum_{K\in\Omega_h}T_3 &\leq \sum_{K\in\Omega_h}\sum_{F\in\partial K\cap\Sigma_h} \|\psi_h\|_{0,F} \|(\widetilde{\bv}-\bPi_{k}^{0,F}\widetilde{\bv})\cdot \bn_\Sigma\|_{0,F}\\
        &\leq \sum_{K\in\Omega_h}\sum_{F\in\partial K\cap\Sigma_h} \|\psi_h\|_{1,F} \|\widetilde{\bv}-\bPi_{k}^{0,F}\widetilde{\bv}\|_{0,F} \nonumber\\
        &\leq C_{\text{apx}}\sum_{K\in\Omega_h}\sum_{F\in\partial K\cap\Sigma_h} \|\psi_h\|_{1,F} h^{\frac{1}{2}}_K\|\widetilde{\bv}\|_{1,K} \nonumber\\
        &\leq C_2  h^{\frac{1}{2}} \left(\|q_h\|_{0,\Om}\|\psi_h\|_{1,\Sigma} + \|\psi_h\|^2_{1,\Sigma}\right).
    \end{align*}
    Finally, for the last term $T_4$, using $\widetilde{\bv}\cdot\bn_\Sigma=\psi_h$ and \eqref{eq:duality_equalities}, we obtain
\begin{align}\label{eq:T4}
        \sum_{K\in\Omega_h}T_4 &=  \|\psi_h\|^2_{1,\Sigma}.
    \end{align}
    Thus, putting together \eqref{eq:T1}-\eqref{eq:T4} and invoking Young's inequality, we can assert the following bound
    \begin{align}
        \sum_{K\in\Omega_h}\left( T_1 + T_2 + T_3 + T_4 \right) & \geq (1-\frac{3}{2}(C_1+C_2) h^{\frac{1}{2}} -C_{\text{apx}}h) \|\psi_h\|^2_{1,\Sigma} -\frac{1}{2}(C_1+C_2) h^{\frac{1}{2}} \|q_h\|_{0,\Om}^2.\label{eq:T1-T4}
    \end{align}
    Hence, for sufficiently small $h<h_0$ with $h_0:=\frac{1}{2}\min\{\frac{4}{9}(C_1+C_2)^{-2},C_{\text{apx}}^{-1}\}$, the combination of \eqref{first_part_disc_inf-sup}  and \eqref{eq:T1-T4} imply that for some $\beta_0>0$
    \begin{align*}
        \sup_{\bv_h\in \bV_h^k\setminus\{\cero\}}\!\!\!\frac{\langle\bB_{1h}\bv_h,(q_h,\psi_h)\rangle  }{\|\bv_h\|_{1,\Omega}} &\geq
        \frac{\beta_0}{\overline{C}_{\bar{\mathbf{I}}^{\mathrm{S}}}}\frac{\|q_h\|^2_{0,\Omega} + \|\psi_h\|_{1,\Sigma}^2}{\|\tilde{\bv}\|_{1,\Omega}} \geq \beta'  \|(q_h,\psi_h)\|.
    \end{align*}
    Therefore, the inf-sup condition holds with $\beta':=\beta_0(\sqrt{2}(\widetilde{\beta}'+\widehat{\beta}')\overline{C}_{\bar{\mathbf{I}}^{\mathrm{S}}})^{-1}>0$.
\end{proof}
Note that the unique solvability of the discrete coupled system \eqref{dis-problem} is a direct consequence of Theorem~\ref{th:fredholm} applied to the discrete product space $\bbX_h:=\bV_h^k\times (\rQ_h^{k-1}\times \rR_h^{k-1})\times \rW_h^k$.
\begin{theorem}[discrete well-posedness]\label{th:disc-well-posed}
     Assuming that $h$ is sufficiently small, there exists unique discrete solution $\vec{x}_h=(\bu_h,(p_h,\varphi_h),w_h)^{\top}\in \bbX_h$ to the discrete problem \eqref{dis-problem} assuming the continuous dependence of data provided by
     \begin{equation}\label{eq:dis-dep}
         \|\bu_h\|_{1,\Omega} + \|(p_h,\varphi_h)\| + \|w_h\|_{2,\Sigma} \leq C( \|\ff\|_{\Omega} + \|g\|_{\Sigma} + \|m\|_{\Sigma}).
     \end{equation}
\end{theorem}
 \begin{proof}
       For the existence and uniqueness in the finite dimensional setting,  it suffices to prove that the solution of the homogeneous discrete problem is trivial. Given homogeneous data $\ff=\0, g =0$ and $m=0$,  let $\vec{\bx}_h$ be a solution of \eqref{dis-problem}. Then we follow the analogous arguments as in Lemma~\ref{lem:A+K} to obtain
    \begin{subequations}\label{dis-problem-zero-rhs}
    \begin{align}
            a_{h}(\bu_h,\bv_h)+b_{1h}(\bv_h,(p_h,\varphi_h)) &= 0,\\
            b_{1h}(\bu_h,(q_h,\psi_h))-c_{1h}((p_h,\varphi_h),(q_h,\psi_h))+(b_{2h}+b_{3h})((q_h,\psi_h),w_h) &= 0,\\
            (b_{2h}- b_{3h})((p_h,\varphi_h),\zeta_h)+c_{2h}(w_h,\zeta_h) &= 0. 
    \end{align}
    \end{subequations}
    for all $\bv_h\in\bV_h^k$, $(q_h,\psi_h)\in \rQ_h^{k-1}\times \rR_h^{k-1}$, and $\zeta_h \in \rW_h^k$. The discrete inf-sup condition from Theorem~\ref{dis-inf-sup}, the first equation in \eqref{dis-problem-zero-rhs}, and the boundedness of $a_{h}(\cdot,\cdot)$ \eqref{ah:bound}, imply the following estimate 
\begin{align*}
\beta \|p_h\|_{0,\Omega} \leq \beta  \| (p_h,\varphi_h)\| &\leq \sup_{\bv_h\in \bV_h^k\setminus\{\cero\}} \frac{b_{1h}(\bv_h,(p_h,\varphi_h))}{\|\bv_h\|_{1,\Omega}} \notag\\
& \leq \sup_{\bv_h\in \bV_h^k\setminus\{\cero\}} \frac{|a_h(\bu_h,\bv_h)|}{\|\bv_h\|_{1,\Omega}}\notag\\
& \leq \max\{\frac{\rho_f}{\tau}(1+C_{s0}),\mu (1+C_{s1}),{\gamma C_T}\} \|\bu\|_{1,\Omega}.\end{align*}
       Next, elementary algebra with  the test functions $\bv_h=\bu_h$, $(q_h,\psi_h)=(p_h,\varphi_h),$ and
       $\zeta_h=w_h$ in \eqref{dis-problem-zero-rhs}  leads to
       \[a_h(\bu_h,\bu_h)+c_{1h}((p_h,\varphi_h),(p_h,\varphi_h))-2b_{3h}((p_h,\varphi_h),w_h)+c_{2h}(w_h,w_h)=0.\]
       The coercivity of $a_h$ from \eqref{ah:coer} and the lower bounds of $c_{1h}$ and $c_{2h}$   from \eqref{c1h:pos} and  \eqref{c2h:coer}, and the boundedness of $b_{3h}$ from \eqref{b3h:bound}  in the above identity show
       \begin{align*}
           0\geq& \min\{\frac{\rho_f}{\tau C_{s0}},\frac{\mu}{C_{s1}}\} \|\bu_h\|_{1,\Omega}^2+\frac{c_0}{\tau C_{s2}}\|\varphi_h\|^2_{0,\Sigma}\hspace{-0.1cm}+\frac{\kappa}{C_{s3}} |\varphi_h|_{1,\Sigma}^2 + \frac{\rho_p}{\tau^3 C_{s4}}\|w_h\|^2_{0,\Sigma}+\frac{D}{\tau C_{s5}} |w_h|^2_{2,\Sigma}-\frac{1}{\tau} \|\varphi_h\|_{0,\Sigma}\|w_h\|_{0,\Sigma}\\
           & \geq \min\{\frac{\rho_f}{\tau C_{s0}},\frac{\mu}{C_{s1}} \}\|\bu_h\|_{1,\Omega}^2+(\frac{c_0}{\tau C_{s2}}-\frac{1}{\epsilon \tau})\|\varphi_h\|^2_{0,\Sigma}+\frac{\kappa}{C_{s3}}|\varphi_h|_{1,\Sigma}^2 + (\frac{\rho_p}{\tau^3 C_{s4}}-\frac{\epsilon}{\tau})\|w_h\|^2_{0,\Sigma}+\frac{D}{\tau C_{s5}} |w_h|^2_{2,\Sigma}.
        \end{align*}
        Finally, and similarly to the continuous case, we can infer that $\bu_h = \cero$, $\varphi_h = 0$, and $w_h=0$ by choosing $\epsilon'$ such that $\frac{C_{s2}}{c_0} < \epsilon' < \frac{\rho_p}{\tau^2C_{s4}}$. The bound in \eqref{eq:dis-dep} follows exactly as in the proof of Theorem~\ref{th:wellp}.  
 \end{proof}
 
\section{A priori error estimates}\label{sec:error}
This section establishes the convergence of the \ac{vem} from Section~\ref{sec:vem}. We start by defining the total error,  total discrete interpolation error, and  total continuous interpolation error as
\begin{align*}
\text{e}_{\vec{x}_h}&:=\text{e}_{\bu_h}+\text{e}_{p_h}+\text{e}_{w_h}+\text{e}_{\varphi_h}:=\norm{\bu-\bu_h}_{1,\Omega}+\norm{p-p_h}_{0,\Omega}+\norm{w-w_h}_{2,\Sigma}+\norm{\varphi-\varphi_h}_{1,\Sigma}, \\ 
\mathrm{e}_{\vec{\bx}_h^*} &:= \text{e}_{\bu_h^*}+\text{e}_{p_h^*}+\text{e}_{w_h^*}+\text{e}_{\varphi_h^*}:= \|\bu_h-\bar{\mathbf{I}}_{h}^{\mathrm{S}} \bu\|_{1,\Omega} + \|p_h - \Pi_{k-1}^{0,K} p\|_{0,\Omega} + \|\varphi_h - \mathrm{I}_{1,h}^{\mathrm{P}} \varphi\|_{1,\Sigma} + \|w_h - \mathrm{I}_{2,h}^{\mathrm{P}} w\|_{2,\Sigma}, \\ 
\mathrm{e}_{\vec{\bx}^*} &:= \text{e}_{\bu^*}+\text{e}_{p^*}+\text{e}_{w^*}+\text{e}_{\varphi^*} := \|\bu-\bar{\mathbf{I}}_{h}^{\mathrm{S}} \bu\|_{1,\Omega} + \|p - \Pi_{k-1}^{0,K} p\|_{0,\Omega} + \|\varphi - \mathrm{I}_{1,h}^{\mathrm{P}} \varphi\|_{1,\Sigma} + \|w - \mathrm{I}_{2,h}^{\mathrm{P}} w\|_{2,\Sigma}, 
\end{align*}
respectively. The following result provides a bound of $\text{e}_{\vec{x}_h}$ in terms of the data approximation, polynomial approximation, and interpolation errors.
\begin{theorem}[energy-error estimate]\label{lemm:errorForApriori}
    Assuming that $h$ is sufficiently small, let $\vec{x}\in \bbX$ and $\vec{x}_h\in \bbX_h$ be the unique solutions to \eqref{eq:weak} and \eqref{dis-problem}, respectively. Then, the following estimate holds
    \begin{align*}
    \mathrm{e}_{\vec{x}_h} &\leq C_{\mathrm{e}} \sum_{K\in \Omega_h} \bigg[ \|F^K - F_h^K\|_{(\bV_h^k(K))'} + \|\bu-\bPi_k^{0,K}\bu\|_{1,K} + \|\bu-\bPi_k^{\nabla,K}\bu\|_{1,K}  + \|p-\Pi_{k-1}^{0,K}p\|_{0,K} \\
    & \quad +\, \|\bu-\bar{\mathbf{I}}_{h}^{\mathrm{S}}\bu\|_{1,K} + \sum_{F\in \partial K \cap \Sigma_h} \biggl(\|G^F - G_h^F\|_{(\rQ_h^{k-1}(K)\times \rR_h^{k-1}(F))'} + \|M^F-M_h^F\|_{(\rW_h^k(F))'} \\
    & \quad +\, \|\bu-\bPi_{k}^{0,F}\bu\|_{0,F} + \|\varphi-\Pi_{k-1}^{0,F}\varphi\|_{1,F} + \|\varphi-\Pi_{k-1}^{\nabla,F}\varphi\|_{1,F} + \|w-\Pi_{k}^{0,F}w\|_{2,F} \\
    & \quad +\, \|w-\Pi_{k}^{\nabla^2,F}w\|_{2,F} + \|w-\Pi_{k}^{\nabla^2,F}w\|_{2,F} + \|\varphi-\mathrm{I}_{1,h}^{\mathrm{P}}\varphi\|_{1,F} + \|w-\mathrm{I}_{2,h}^{\mathrm{P}}w\|_{2,F}\bigg)\bigg].
    \end{align*}
\end{theorem}
\begin{proof}
The continuos problem \eqref{eq:weak} and the discrete problem \eqref{dis-problem} show that $\vec{\bx}_h^* = (\bu_h-\bar{\mathbf{I}}_{h}^{\mathrm{S}} \bu,(p_h - \Pi_{k-1}^{0,K} p,\varphi_h - \mathrm{I}_{1,h}^{\mathrm{P}} \varphi),w_h - \mathrm{I}_{2,h}^{\mathrm{P}} w)^{\top} \in \bV_h^k\times (Q_h^{k-1}\times R_h^{k-1})\times W_h^k$ is the unique solution to
\begin{align*}
        a_{h}(\bu_h-\bar{\mathbf{I}}_{h}^{\mathrm{S}} \bu,\bv_h)+b_{1h}(\bv_h,(p_h - \Pi_{k-1}^{0} p,\varphi_h - \mathrm{I}_{1,h}^{\mathrm{P}} \varphi))&= \check{F}(\bv_h),\\
        b_{1h}(\bu_h-\bar{\mathbf{I}}_{h}^{\mathrm{S}} \bu,(q_h,\psi_h))-c_{1h}((p_h - \Pi_{k-1}^{0} p,\varphi_h - \mathrm{I}_{1,h}^{\mathrm{P}} \varphi),(q_h,\psi_h))& \notag \\ +(b_{2h}+b_{3h})((q_h,\psi_h),w_h - \mathrm{I}_{2,h}^{\mathrm{P}} w)&=\check{G}((q_h,\psi_h)),\\
        (b_{2h}-b_{3h})((p_h - \Pi_{k-1}^{0} p,\varphi_h - \mathrm{I}_{1,h}^{\mathrm{P}} \varphi),\zeta_h)+c_{2h}(w_h - \mathrm{I}_{2,h}^{\mathrm{P}} w,\zeta_h)&= \check{M}(\zeta_h),
\end{align*}
for all $\bv_h\in\bV_h^k$, $(q_h,\psi_h)\in \rQ_h^{k-1}\times \rR_h^{k-1}$, $\zeta_h \in \rW_h^k$. Here, the global polynomial projection of $p$ is given by $\left. \Pi_{k-1}^{0} p\right|_K := \Pi_{k-1}^{0,K} p$. Whereas, the new discrete right-hand sides are defined by
\begin{subequations}\label{check_RHS}
\begin{align}
        \check{F}(\bv_h) &:= \left[a(\bu,\bv_h) - a_{h}(\bar{\mathbf{I}}_{h}^{\mathrm{S}} \bu, \bv_h)\right]
        \notag \\
        & \quad +\, \left[b_1(\bv_h,(p,\varphi)) - b_{1h}(\bv_h,(\Pi_{k-1}^{0} p,\mathrm{I}_{1,h}^{\mathrm{P}} \varphi))\right] \notag \\
        & \quad +\, (F_h - F)(\bv_h)\notag \\
        &=: \check{F}_1 + \check{F}_2 + \check{F}_3,\label{check_F}\\
        \check{G}((q_h,\psi_h)) &:= \left[b_1(\bu, (q_h,\psi_h)) - b_{1h}(\bar{\mathbf{I}}_{h}^{\mathrm{S}} \bu, (q_h,\psi_h))\right]\notag \\
        &\quad +\, \left[-c_1((p,\varphi),(q_h,\psi_h)) + c_{1h}((\Pi_{k-1}^{0} p,\mathrm{I}_{1,h}^{\mathrm{P}} \varphi),(q_h,\psi_h))\right] \notag \\
        &\quad +\, \left[b_2((q_h.\psi_h),w) - b_{2h}((q_h,\psi_h),\mathrm{I}_{2,h}^{\mathrm{P}} w)\right]\notag \\
        &\quad +\, \left[b_3((q_h.\psi_h),w) - b_{3h}((q_h,\psi_h),\mathrm{I}_{2,h}^{\mathrm{P}} w)\right] \notag \\
        &\quad +\, (G_h-G)((q_h,\psi_h))\notag \\
        &=: \check{G}_1 + \check{G}_2 + \check{G}_3 + \check{G}_4 + \check{G}_5,\label{check_G}\\
        \check{M}(\zeta_h) &:= \left[b_2((p,\varphi),\zeta_h) - b_{2h}((\Pi_{k-1}^{0} p,\mathrm{I}_{1,h}^{\mathrm{P},F} \varphi),\zeta_h)\right] \notag \\
        &\quad +\, \left[-b_3((p,\varphi),\zeta_h) + b_{3h}((\Pi_{k-1}^{0} p,\mathrm{I}_{1,h}^{\mathrm{P}} \varphi),\zeta_h)\right] \notag \\
        &\quad +\, \left[c_2(w,\zeta_h)-c_{2h}(\mathrm{I}_{2,h}^{\mathrm{P}} w,\zeta_h)\right]\notag \\
        &\quad +\, (M_h-M)(\zeta_h) \notag \\
        &=: \check{M}_1 + \check{M}_2 + \check{M}_3 + \check{M}_4.\label{check_M}
\end{align}
\end{subequations}
Therefore, the discrete dependence on data \eqref{eq:dis-dep} implies that
\begin{align*}
\mathrm{e}_{\vec{\bx}_h^*} \leq C \left( \|\check{F}\|_{(\bV_h^k)'} + \|\check{G}\|_{(\rQ_h^{k-1}\times \rR_h^{k-1})'} + \|\check{M}\|_{(\rW_h^k)'}\right).
\end{align*}
Next we aim at bounding the functionals  in \eqref{check_RHS}. The consistency with respect to polynomials of the stabilised terms that define the local  forms in \eqref{discrete_forms} are given as follows for all $K\in \Omega_h$, and $F\in\Sigma_h$
\begin{subequations}\label{const}
\begin{align}
a_h^{0,K}(\bPi_k^{0,K}\bu,\bv_h) &= a^{0,K}(\bPi_k^{0,K}\bu,\bv_h), \label{const_a11}\\  a_h^{\nabla,K}(\bPi_k^{\nabla,K}\bu,\bv_h) &= a^{\nabla,K}(\bPi_k^{\nabla,K}\bu,\bv_h),\label{const_a12}\\
c_{1h}^{0,F}((p_h,\Pi_{k-1}^{0,F}\varphi),(q_h,\psi_h)) &= c_{1}^{0,F}((p_h,\Pi_{k-1}^{0,F}\varphi),(q_h,\psi_h)),\label{const_c11}\\
c_{1h}^{\nabla,F}((p_h,\Pi_{k-1}^{\nabla,F}\varphi),(q_h,\psi_h)) &= c_{1}^{\nabla,F}((p_h,\Pi_{k-1}^{\nabla,F}\varphi),(q_h,\psi_h)),\label{const_c12}\\
c_{2h}^{0,F}(\Pi_k^{0,F}w,\zeta_h) &= c_{2}^{0,F}(\Pi_k^{0,F}w,\zeta_h),\label{const_c21}\\ c_{2h}^{\nabla^2,F}(\Pi_k^{\nabla^2,F}w,\zeta_h) &= c_{2}^{\nabla^2,F}(\Pi_k^{\nabla^2,F}w,\zeta_h). \label{const_c22}
\end{align}
\end{subequations}
The previous identities follow from the polynomial projections defined in \eqref{L2-proj}-\eqref{nabla2_proj}, and the discrete bilinear forms given in \eqref{discrete_forms}. In addition, the following equality holds due to the triple scalar rule and the anti-commutative properties
\begin{align}\label{normal_term}
    \int_F (\bPi_{k}^{0,F}\bu\times \bn_{\Sigma}^F) \cdot \left((\bv_h-\bPi_{k}^{0,F}\bv_h)\times \bn_{\Sigma}^F\right) &= - \int_F (\bPi_{k}^{0,F}\bu\times \bn_{\Sigma}^F) \cdot \left(\bn_{\Sigma}^F \times (\bv_h-\bPi_{k}^{0,F}\bv_h)\right) \notag \\
    &= -\int_F \left((\bPi_{k}^{0,F}\bu\times \bn_{\Sigma}^F)\times \bn_\Sigma\right) \cdot \left(\bv_h-\bPi_{k}^{0,F}\bv_h)\right) \notag \\
    &= 0.
\end{align}
Next, the triangle inequality, \eqref{const_a11}-\eqref{const_a12}, \eqref{normal_term}, and the idempotency of the polynomial projections applied to $\check{F}_1$, and $\check{F}_2$ lead to
\begin{align*}
    \left|\check{F}_1 \right| &\leq \sum_{K\in \Omega_h} \left| a^{0,K}(\bu-\bPi_k^{0,K}\bu,\bv_h) - a_h^{0,K}(\bar{\mathbf{I}}_{h}^{\mathrm{S}} \bu-\bPi_k^{0,K}\bu, \bv_h) \right| \\
    &\quad + \sum_{K\in \Omega_h} \left| a^{\nabla,K}(\bu-\bPi_k^{\nabla,K}\bu,\bv_h) - a_h^{\nabla,K}(\bar{\mathbf{I}}_{h}^{\mathrm{S}} \bu-\bPi_k^{\nabla,K}\bu, \bv_h) \right| \\
    &\quad + \sum_{K\in \Omega_h} \left| a^{\Sigma,K}(\bu-\bPi_{k}^{0,F}\bu,\bv_h) - a_h^{\Sigma,K}(\bar{\mathbf{I}}_{h}^{\mathrm{S}} \bu-\bPi_{k}^{0,F}\bu,\bv_h) \right|, \\
    \left|\check{F}_2 \right| &\leq \sum_{K\in \Omega_h} \left| b_1^{\vdiv,K}(\bv_h, (p-\Pi_{k-1}^{0,K}p,\mathrm{I}_{1,h}^{\mathrm{P}}\varphi))  \right|\\
    &\quad + \sum_{K\in \Omega_h} \left| b_{1}^{\Sigma,K}(\bv_h,(p,\varphi-\Pi_{k-1}^{0,F}\varphi)) - b_{1h}^{\Sigma,K}(\bv_h,(\Pi_{k-1}^{0,K} p,\mathrm{I}_{1,h}^{\mathrm{P}}\varphi-\Pi_{k-1}^{0,F}\varphi) ) \right|.
\end{align*}
Using \eqref{a:bound} and \eqref{b1:bound} for $a(\cdot,\cdot)$ and $b_{1h}(\cdot,\cdot)$ (resp. \eqref{ah:bound} and \eqref{b1h:bound} for $a_h(\cdot,\cdot)$ and $b_{1h}(\cdot,\cdot)$), we can readily obtain that
\begin{align}\label{eq:checkFBound}
    \|\check{F}\|_{(\bV_h^k)'} &\leq C_{\check{F}} \hspace{-0.05cm} \sum_{K\in \Omega_h} \biggl( \|F^K - F_h^K\|_{(\bV_h^k(K))'} + \|\bu-\bPi_k^{0,K}\bu\|_{1,K} + \|\bu-\bPi_k^{\nabla,K}\bu\|_{1,K}\hspace{-0.05cm} + \hspace{-0.05cm}\|\bu-\bar{\mathbf{I}}_{h}^{\mathrm{S}}\bu\|_{1,K} \notag\\
    &\quad +\, \|p-\Pi_{k-1}^{0,K}p\|_{0,K} + \!\!\!\sum_{F\in \partial K \cap \Sigma_h} \bigl( \|\bu-\bPi_{k}^{0,F}\bu\|_{0,F} +\, \|\varphi-\Pi_{k-1}^{0,F}\varphi\|_{1,F} + \|\varphi-\mathrm{I}_{1,h}^{\mathrm{P}}\varphi\|_{1,F} \bigr) \biggr)
\end{align}
with $C_{\check{F}} = \max\{\frac{\rho_f}{\tau}(1+C_{s0}),\mu (1+C_{s1}),\gamma C_T,1,C_T\}$. To address \eqref{check_G}, the triangle inequality, basic manipulations, \eqref{const_c11}-\eqref{const_c12}, and the definitions in \eqref{L2-proj}-\eqref{nabla_proj} imply that
\begin{align*}
    \left|\check{G}_1 \right| &\leq \sum_{K\in \Omega_h} \left| b_1^{\vdiv,K}(\bu-\bar{\mathbf{I}}_{h}^{\mathrm{S}} \bu, (q_h,\psi_h))  \right|\\
    &\quad + \sum_{K\in \Omega_h} \left| b_{1}^{\Sigma,K}(\bu - \bPi_{k}^{0,F}\bu,(q_h,\psi_h)) - b_{1h}^{\Sigma,K}(\bar{\mathbf{I}}_{h}^{\mathrm{S}} \bu - \bPi_{k}^{0,F}\bu,(q_h,\psi_h) ) \right|,\\
    \left|\check{G}_2 \right| &\leq \sum_{F\in \partial K \cap \Sigma_h} \left| c_1^{0,F}((p,\varphi-\Pi_{k-1}^{0,F}\varphi),(q_h,\psi_h)) - c_{1h}^{0,F}((\Pi_{k-1}^{0,K}p,\mathrm{I}_{1,h}^{\mathrm{P}}\varphi-\Pi_{k-1}^{0,F}\varphi),(q_h,\psi_h)) \right| \\
    &\quad + \sum_{F\in \partial K \cap \Sigma_h} \left| c_1^{\nabla,F}((p,\varphi-\Pi_{k-1}^{\nabla,F}\varphi),(q_h,\psi_h)) - c_{1h}^{\nabla,F}((\Pi_{k-1}^{0,K}p,\mathrm{I}_{1,h}^{\mathrm{P}}\varphi-\Pi_{k-1}^{\nabla,F}\varphi),(q_h,\psi_h)) \right|, \\
    \left|\check{G}_3 \right| &\leq \frac{\alpha}{\tau} \sum_{F\in \partial K \cap \Sigma_h} \left| -\int_F \nabla_\Sigma \psi_h \cdot \nabla_\Sigma w + \int_F \bPi_{k-2}^{0,F}(\nabla_\Sigma\psi_h)\cdot\bPi_{k-1}^{0,F}(\nabla_\Sigma (\mathrm{I}_{2,h}^{\mathrm{P}} w)) \right|\\
    &=\hspace{-0.05cm}\frac{\alpha}{\tau} \hspace{-0.05cm} \sum_{F\in \partial K \cap \Sigma_h} \left| -\hspace{-0.05cm} \int_F \hspace{-0.05cm} \nabla_\Sigma \psi_h \cdot \nabla_\Sigma (w-\Pi_{k}^{\nabla,F}w)\hspace{-0.05cm} + \hspace{-0.05cm}\int_F \hspace{-0.05cm} \bPi_{k-2}^{0,F}(\nabla_\Sigma\psi_h)\cdot \bPi_{k-1}^{0,F}(\nabla_\Sigma (\mathrm{I}_{2,h}^{\mathrm{P}} w -\Pi_{k}^{\nabla,F}w)) \right|,\\
    \left|\check{G}_4 \right| &\leq \frac{1}{\tau} \sum_{F\in \partial K \cap \Sigma_h} \left| -\int_F w \psi_h + \int_F \Pi_{k}^{0,F}(\mathrm{I}_{2,h}^{\mathrm{P}} w) \Pi_{k-1}^{0,F} \psi_h \right|\\
    &=\frac{1}{\tau} \sum_{F\in \partial K \cap \Sigma_h} \left| -\int_F (w-\Pi_{k}^{0,F}w) \psi_h + \int_F \Pi_{k}^{0,F}(\mathrm{I}_{2,h}^{\mathrm{P}} w-\Pi_{k}^{0,F}w) \Pi_{k-1}^{0,F} \psi_h \right|.
\end{align*}
In turn, the bounds in \eqref{c1:bound}, \eqref{b1:bound}, \eqref{b2:bound}, and \eqref{b3:bound} for $c_1(\cdot,\cdot)$, $b_1(\cdot,\cdot)$, $b_2(\cdot,\cdot)$, and $b_3(\cdot,\cdot)$ (resp. \eqref{c1h:bound}, \eqref{b1h:bound}, \eqref{b2h:bound}, and \eqref{b3h:bound} for $c_{1h}(\cdot,\cdot)$, $b_{1h}(\cdot,\cdot)$, $b_{2h}(\cdot,\cdot)$, and $b_{3h}(\cdot,\cdot)$) inspire the following estimate
\begin{align*}
    \|\check{G}\|_{(\rQ_h^{k-1}\times \rR_h^{k-1})'} &\leq C_{\check{G}} \sum_{K\in \Omega_h} \biggl( \|\bu-\bar{\mathbf{I}}_{h}^{\mathrm{S}}\bu\|_{1,K} + \sum_{F\in \partial K \cap \Sigma_h} \bigl( \|G^F - G_h^F\|_{(\rQ_h^{k-1}(K)\times \rR_h^{k-1}(F))'}  \notag \\
    &\quad  +\, \|\bu-\bPi_{k}^{0,F}\bu\|_{0,F} +  \|\varphi-\Pi_{k-1}^{0,F}\varphi\|_{1,F} + \|\varphi-\Pi_{k-1}^{\nabla,F}\varphi\|_{1,F} + \|\varphi-\mathrm{I}_{1,h}^{\mathrm{P}}\varphi\|_{1,F} \notag\\
    &\quad  +\, |w-\Pi_{k}^{0,F}w\|_{2,F} + \|w-\Pi_{k}^{\nabla,F}w\|_{2,F} + \|w-\mathrm{I}_{2,h}^{\mathrm{P}}w\|_{2,F}\bigr) \biggr)
\end{align*}
with $C_{\check{G}} = \max\{\frac{c_0}{\tau}(1+C_{s2}),\kappa (1+C_{s3}),1,C_T,\frac{\alpha}{\tau},\frac{1}{\tau}\}$. Finally, the $\check{M}_i$'s ($i=1,2,3$) are bounded as follows
\begin{align*}
    \left|\check{M}_1 \right| &\leq \frac{\alpha}{\tau} \sum_{F\in \partial K \cap \Sigma_h} \left| -\int_F \nabla_\Sigma \varphi \cdot \nabla_\Sigma \zeta_h + \int_F \bPi_{k-2}^{0,F}(\nabla_\Sigma (\mathrm{I}_{1,h}^{\mathrm{P}} \varphi))\cdot\bPi_{k-1}^{0,F}(\nabla_\Sigma \zeta_h) \right|\\
    &=\hspace{-0.05cm}\frac{\alpha}{\tau}\hspace{-0.05cm} \sum_{F\in \partial K \cap \Sigma_h} \left| -\int_F \nabla_\Sigma (\varphi-\Pi_{k-1}^{\nabla,F}\varphi) \cdot \nabla_\Sigma w + \int_F \bPi_{k-2}^{0,F}(\nabla_\Sigma(\mathrm{I}_{1,h}^{\mathrm{P}} \varphi-\Pi_{k-1}^{\nabla,F}\varphi))\cdot \bPi_{k-1}^{0,F}(\nabla_\Sigma \zeta_h) \right|,\\
    \left|\check{M}_2 \right| &\leq \frac{1}{\tau} \sum_{F\in \partial K \cap \Sigma_h} \left| \int_F \zeta_h \varphi - \int_F \Pi_{k}^{0,F}\zeta_h \Pi_{k-1}^{0,F} (\mathrm{I}_{1,h}^{\mathrm{P}}\varphi)\right|\\
    &=\frac{1}{\tau} \sum_{F\in \partial K \cap \Sigma_h} \left| \int_F \zeta_h(\varphi-\Pi_{k-1}^{0,F}\varphi) + \int_F \Pi_{k}^{0,F}\zeta_h \Pi_{k-1}^{0,F}(\mathrm{I}_{1,h}^{\mathrm{P}} \varphi- \Pi_{k-1}^{0,F}\varphi) \right|, \\
    \left|\check{M}_3 \right| &\leq \sum_{F\in \partial K \cap \Sigma_h} \left| c_2^{0,F}(w-\Pi_k^{0,K}w,\zeta_h) - c_{2h}^{0,F}(\mathrm{I}_{2,h}^{\mathrm{P}} w-\Pi_{k}^{0,F}w,\zeta_h) \right| \\
    &\quad + \sum_{F\in \partial K \cap \Sigma_h} \left| c_2^{\nabla^2,F}(w-\Pi_k^{\nabla^2,K}w,\zeta_h) - c_{2h}^{\nabla^2,F}(\mathrm{I}_{2,h}^{\mathrm{P}} w-\Pi_{k}^{\nabla^2,F}w,\zeta_h),\zeta_h) \right|,
\end{align*}
where \eqref{L2-proj}-\eqref{nabla2_proj}, and \eqref{const_c21}-\eqref{const_c22} were used. From \eqref{b2:bound}, \eqref{b3:bound}, and \eqref{c2:bound} for $b_2(\cdot,\cdot)$, $b_3(\cdot,\cdot)$, and $c_2(\cdot,\cdot)$ (resp. \eqref{b2h:bound}, \eqref{b3h:bound}, and \eqref{c2h:bound} for $b_{2h}(\cdot,\cdot)$, $b_{3h}(\cdot,\cdot)$, and $c_{2h}(\cdot,\cdot)$) we readily see
\begin{align}\label{eq:checkMBound}
    \|\check{M}\|_{(\rW_h^k)'} &\leq C_{\check{M}} \sum_{K\in \Omega_h} \sum_{F\in \partial K \cap \Sigma_h} \bigl( \|M^F-M_h^F\|_{(\rW_h^k(F))'} + \|\varphi-\Pi_{k-1}^{0,F}\varphi\|_{1,F} + \|\varphi-\Pi_{k-1}^{\nabla,F}\varphi\|_{1,F} \notag \\
    & \quad +\, \|\varphi-\mathrm{I}_{1,h}^{\mathrm{P}}\varphi\|_{1,F} + \|w-\Pi_{k}^{0,F}w\|_{2,F} + \|w-\Pi_{k}^{\nabla^2,F}w\|_{2,F} + \|w-\mathrm{I}_{2,h}^{\mathrm{P}}w\|_{2,F}\bigl)
\end{align}
with $C_{\check{M}} = \max \{\frac{\rho_p}{\tau^3}(1+C_{s4}),\frac{D}{\tau} (1+C_{s5}),\frac{\alpha}{\tau}\}$. The proof is completed after  specifying the constant $C_{\mathrm{e}}:=\max \{ C_{\check{F}}, C_{\check{G}}, C_{\check{M}}\}$ together with \eqref{eq:checkFBound}-\eqref{eq:checkMBound}, and applying the inverse triangle inequality $\mathrm{e}_{\vec{x}_h}-\mathrm{e}_{\vec{x}^*}\leq \mathrm{e}_{\vec{x}_h^*}$. 
\end{proof}
We finalise by stating the convergence result of the virtual element scheme.
\begin{corollary}[convergence rates]\label{th:convergence}
   Under the small $h$ assumption, let $\vec{x}=(\bu,(p,\varphi),w)^{\top}\in \bH^{s+1}(\Omega)\times (\rH^{s}(\Omega)\times \rH^{r_1}(\Sigma))\times \rH^{r_2}(\Sigma)$ and $\vec{x}_h\in \bbX_h$ be the unique solutions to \eqref{eq:weak} and \eqref{dis-problem}, respectively. Moreover, assume that $\mathbf{f}\in \bH^{s-1}(\Omega)$, $g\in \rH^{r_1-2}(\Sigma)$, and $m\in\rH^{r_2-4}(\Sigma)$ with $0\leq s\leq k$, $1\leq r_1 \leq k$, and $2\leq r_2 \leq k+1$. Then, the total error decays with the following rate
    $$\mathrm{e}_{\vec{x}_h} \leq \hat{C}_{\mathrm{e}} h^{\hat{r}} \left( |\bu|_{s+1,\Omega} + |p|_{s,\Omega} + |\varphi|_{r_1-1,\Sigma} + |w|_{r_2-2,\Sigma} + |\mathbf{f}|_{s-1,\Omega} + |g|_{r_2-2,\Sigma} + |m|_{r_1-4,\Sigma} \right), $$
    where $\hat{r}=\min\{s,r_1-1,r_2-2\}$ and $\hat{C}_{\mathrm{e}} = C_{\mathrm{e}}\max\{C_{\bar{\mathbf{I}}^{\mathrm{S}}},C_{1,\mathrm{I}^P},C_{2,\mathrm{I}^P},C_{\mathrm{apx}}\}>0$.
\end{corollary}
\begin{proof}
    The data approximation terms can be bounded similarly as in \cite[Lemma 3.2]{daveiga15-stokes} and  \cite[Theorem 4.1]{khot23}: For all $K\in \Omega_h$ and all $F\in  \Sigma_h$, 
    \begin{align*}
        \|F^K - F_h^K\|_{(\bV_h^k(K))'} &\leq h^s |\mathbf{f}|_{s-1,K}, \\
        \|G^F - G_h^F\|_{(\rQ_h^{k-1}(K)\times \rR_h^{k-1}(F))'} &\leq h^{r_2-1}|g|_{r_2-2,F},\\
        \|M^F-M_h^F\|_{(\rW_h^k(F))'} &\leq h^{r_1-2}|m|_{r_1-4,F}.
    \end{align*}
    This, together with Propositions~\ref{prop:est-poly}, \ref{prop:fortin}, and \ref{prop:int_plate} applied to Theorem~\ref{lemm:errorForApriori}, finishes the proof.
\end{proof}
\section{Implementation}\label{sec:implementation}
The numerical implementation has been done with the library \texttt{VEM++} \cite{dassi2023vem++}, a \texttt{C++} based \ac{vem} solver. The code has a hierarchical structure, which considers 3D polyhedral elements that contain 2D polygonal faces (an element in 2D), all the corresponding edges, and points. Since the faces of each polyhedron are tightly integrated into the definition and corresponding \acp{dof} of a 3D VE space, defining an independent 2D VE space on the same mesh becomes non-trivial.  

To avoid this complexity at the time, we separately implemented the bulk part corresponding to \eqref{eq:weak1}-\eqref{eq:weak2} and the surface part in \eqref{eq:weak3}-\eqref{eq:weak5}. \texttt{VEM++} provides functionality to generate a 2D mesh from a given surface of a 3D mesh with the class \texttt{vem::mesh2dGeneratorFromMesh3d}. The resulting 2D and 3D geometries are then linked by the class \texttt{vem::mesh2dAnd3dConnectorData}, which establishes a bijective mapping between their corresponding points and elements. A schematic overview of this process is shown in Figure~\ref{fig:sketch_classes}. Then, the coupled problem \eqref{dis-problem} is solved with an optimised Picard iteration (fixed-point). The rest of this section is dedicated to prove the convergence of this alternative method to the target solution.

\begin{figure}[t!]
    \centering
    \includegraphics[width=0.8\textwidth,trim={0.4cm .5cm 5.cm 0.5cm},clip]{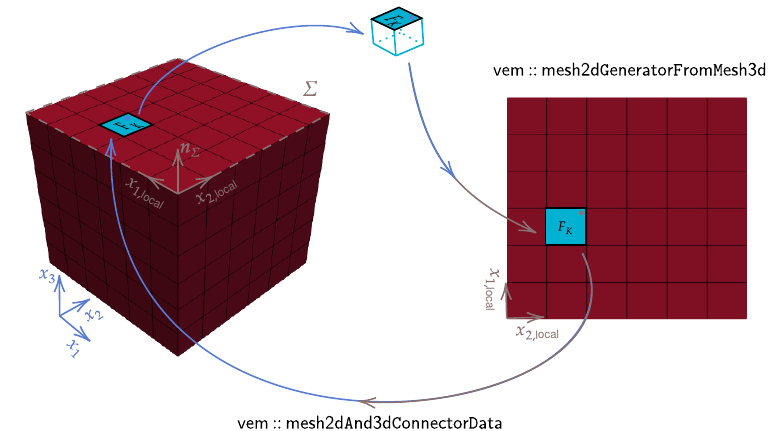}
    \caption{Diagram for creation of a 2D mesh from the surface of a 3D mesh, the bijective map between entities in 3D and 2D is shown with blue and grey arrows.}\label{fig:sketch_classes}
\end{figure}

Define the two sub-problems to be solved as follows
\begin{subequations}
\begin{align}
    \begin{pmatrix}
        \bA_h & (\bB_{1h}^{\vdiv})^* \\
        \bB_{1h}^{\vdiv} & \cero
    \end{pmatrix}
    \begin{pmatrix}
        \bu_h \\
        p_h 
    \end{pmatrix} 
    &= 
    \begin{pmatrix}
        F_h - (\bB_{1h}^{\Sigma,\vartheta_h})^* \\ 
        \cero 
    \end{pmatrix} 
    & \text{in} \quad (\bV_h^k \times \rQ_h^{k-1})',\label{eq:operatorBulk} \\
    \begin{pmatrix}
        - \bC_{1h} & \bB_{2h}^* + \bB_{3h}^* \\
        \bB_{2h}\,-\bB_{3h}  & \bC_{2h}
    \end{pmatrix}
    \begin{pmatrix}
    \varphi_h \\ 
    w_h
    \end{pmatrix}
    &= 
    \begin{pmatrix}
    G_h - \bB_{1h}^{\Sigma,\br_h}\\ 
    M_h    
    \end{pmatrix} 
    & \text{in} \quad (\rR_h^{k-1}\times \rW_h^k)'. \label{eq:operatorPlate} 
    \end{align}\end{subequations}
Note that the coupling terms $(\bB_{1h}^{\Sigma,\psi_h})^*$ and $\bB_{1h}^{\Sigma,\bv_h}$ consider their respective coupling variables $\vartheta_h$, and $\br_h$ as input data. The well-posedness
of these decoupled problems is provided next.
\begin{lemma}\label{dec_well_posed}
    Given the data $\ff\in \bL^2(\Omega)$, $g\in \rL^2(\Sigma)$, $m \in \rL^2(\Sigma)$, and the coupling terms $\vartheta_h\in \rR_h^{k-1}$, and $\br_h \in \bV_h^k$, there exist unique solutions $(\bu_h,p_h)^{\top} \in \bV_h^{h}\times \rQ_h^{k-1}$, and $(\varphi_h,w_h)^{\top} \in \rR_h^{k-1} \times \rW_h^{k}$ to \eqref{eq:operatorBulk} and \eqref{eq:operatorPlate}, respectively. Moreover, there exist two positive constants $\hat{C}_1$ and $\hat{C}_2$ such that 
    \begin{subequations}\begin{align}
        \|\bu_h\|_{1,\Omega} + \|p_h\|_{0,\Omega} &\leq \hat{C}_1 \left( \|\ff\|_{0,\Omega} + {C_T}\|\vartheta_h\|_{1,\Sigma}\right),\label{continuous_dec_bulk}\\
        \|\varphi_h\|_{1,\Sigma} + \|w_h\|_{2,\Sigma} &\leq \hat{C}_2 \left( \|g\|_{0,\Sigma} + \|m\|_{0,\Sigma} + {C_T}\|\br_h\|_{1,\Omega}\label{continuous_dec_plate}\right).
    \end{align}\end{subequations}
\end{lemma}
\begin{proof}
    The operator $\bA_h$ is coercive and bounded from \eqref{ah:bound}, and \eqref{ah:coer}. Note from \eqref{b1h:bound} that the operator $\bB_{1h}^{\vdiv}$ is bounded and satisfies the discrete inf-sup condition provided in \eqref{first_part_disc_inf-sup}. Therefore, Theorem~\ref{th:perturbed} provides the bound in \eqref{continuous_dec_bulk}. On the other hand, the stiffness matrix in \eqref{eq:operatorPlate} can be split as follows:
    \begin{align*}
        \begin{pmatrix}
        - \bC_{1h} & \bB_{2h}^* + \bB_{3h}^* \\
        \bB_{2h}\,-\bB_{3h}  & \bC_{2h}
    \end{pmatrix} 
    =
    \begin{pmatrix}
        - \bC_{1h} & \cero \\
        \cero  & \bC_{2h}
    \end{pmatrix}
    +
    \begin{pmatrix}
        \cero & \bB_{2h}^* + \bB_{3h}^* \\
        \bB_{2h}\,-\bB_{3h}  & \cero
    \end{pmatrix}.
    \end{align*}
    This observation, together with \eqref{c1h:pos}, \eqref{c2h:coer}, Lemma~\ref{lem:K}, and Lemma~\ref{lem:A+K} applied to Theorem~\ref{th:fredholm} imply the bound in \eqref{continuous_dec_plate}.
\end{proof}

The previous result motivates the definition of two solution operators 
\begin{gather*}
    S_{1h}: \rR_h^{k-1}\rightarrow \bV_h^k \times \rQ_h^{k-1}, \, \vartheta_h\mapsto (S_{11h}(\vartheta_h),S_{12h}(\vartheta_h))^{\top} := (\bu_h,p_h)^{\top},\\
    S_{2h}: \bV_h^k\rightarrow \rR_h^{k-1} \times \rW_h^k, \, \br_h\mapsto (S_{21h}(\br_h),S_{22h}(\br_h))^{\top} := (\varphi_h,w_h)^{\top}. 
\end{gather*}
In particular, the fully coupled problem \eqref{dis-problem} is equivalent to solving the following fixed-point equation:
\begin{equation}\label{fixed-point-formulation}
    \text{Find } \varphi_h\in \rR_h^{k-1}, \text{ such that } \mathcal{A}_h(\varphi_h) = \varphi_h,
\end{equation}
where $\mathcal{A}_h: \rR_h^{k-1} \rightarrow \rR_h^{k-1}$ is defined as $\mathcal{A}_h(\varphi_h) := S_{21h}(S_{11h}(\varphi_h))$. Lemma~\ref{dec_well_posed} shows that $\mathcal{A}_h$ is well-defined. The following theorem provides the well-posedness of \eqref{dis-problem} via an equivalent fixed-point argument. 

\begin{theorem}\label{th:fixed-point}
    Assume that $h$ is sufficiently small. Then, the operator $\mathcal{A}_h$ has a unique fixed-point $\varphi_h\in \cS_h$ and the continuous dependence on data \eqref{eq:dis-dep} holds.
\end{theorem}
\begin{proof}
    To prove existence of a fixed point, we invoke the Brouwer fixed-point theorem. Let $ \cS_h:= \{\vartheta_h \in \rR_h^{k-1} : \|\vartheta_h\|_{1,\Sigma} \leq \hat{C}_2 ( \|g\|_{0,\Sigma} + \|m\|_{0,\Sigma} + {C_T}\|\bu_h\|_{1,\Omega})\}$ (where $\bu_h$ satisfies \eqref{continuous_dec_plate}). It is easy to see that $\cS_h$ is convex and compact  set (closed and bounded ball in $\rR_h^{k-1}$), and $\cA_h(\cS_h)\subseteq \cS_h$.  Indeed, for any $\varphi_h\in \cS_h$,
\begin{align*}
    \|\mathcal{A}_h(\varphi_h)\|_{1,\Sigma} &\leq \|S_{21h}(S_{11h}(\varphi_h))\|_{1,\Sigma} + \|S_{22h}(S_{11h}(\varphi_h))\|_{2,\Omega}\leq  \hat{C}_2 \{ \|g\|_{0,\Sigma} + \|m\|_{0,\Sigma} + {C_T}\|\bu_h\|_{1,\Omega}\}.
\end{align*} 
It remains to show that $\mathcal{A}_h$ is Lipschitz continuous, and hence continuous. To do this, given $(\bu_{1h},p_{1h})^{\top},(\bu_{2h},p_{2h})^{\top}\in \bV_h^k\times \rQ_h^{k-1}$ solutions to \eqref{eq:operatorBulk} with respect to $\vartheta_{1h},\vartheta_{2h}\in \rR_h^{k-1}$ and $(\varphi_{1h},w_{1h})^{\top},(\varphi_{2h},w_{2h})^{\top}\in \rR_h^{k-1}\times \rW_h^k$ solutions to \eqref{eq:operatorPlate} with respect to $\br_{1h},\br_{2h}\in \bV_h^k$. Thus, we obtain the following sub-systems from the subtraction of these problems:
    \begin{align*}
    \begin{pmatrix}
        \bA_h & (\bB_{1h}^{\vdiv})^* \\
        \bB_{1h}^{\vdiv} & \cero
    \end{pmatrix}
    \begin{pmatrix}
        \bu_{1h}-\bu_{2h} \\
        p_{1h}-p_{2h} 
    \end{pmatrix} 
    &= 
    \begin{pmatrix}
        - (\bB_{1h}^{\Sigma,(\vartheta_{1h}-\vartheta_{2h})})^* \\ 
        \cero 
    \end{pmatrix} 
    \quad &&\text{in} \quad (\bV_h^k \times \rQ_h^{k-1})',\\
    \begin{pmatrix}
        - \bC_{1h} & \bB_{2h}^* + \bB_{3h}^* \\
        \bB_{2h}\,-\bB_{3h}  & \bC_{2h}
    \end{pmatrix}
    \begin{pmatrix}
    \varphi_{1h}-\varphi_{2h} \\ 
    w_{1h}-w_{2h}
    \end{pmatrix} 
    &= 
    \begin{pmatrix}
    - \bB_{1h}^{\Sigma,(\br_{1h}-\br_{2h})}\\ 
    \cero    
    \end{pmatrix} 
    \quad &&\text{in} \quad (\rR_h^{k-1}\times \rW_h^k)'. 
    \end{align*}
    Lemma~\ref{dec_well_posed} provides that
    \begin{align*}
        \|\cA_h(\vartheta_{1h}-\vartheta_{2h})\|_{1,\Sigma} &= \|S_{21h}(S_{11h}(\vartheta_{1h}-\vartheta_{2h}))\|_{1,\Sigma}\leq \hat{C}_2{C_T} \|S_{11h}(\vartheta_{1h}-\vartheta_{2h})\|_{1,\Omega} \leq \hat{C}_1 \hat{C}_2 {C_T^2} \|\vartheta_{1h}-\vartheta_{2h}\|_{1,\Sigma}.
    \end{align*}
Hence, the Brouwer fixed-point theorem implies that there exists a fixed-point solution to \eqref{fixed-point-formulation} in the set $\cS_h$. The uniqueness of this solution and continuous dependence on data follows from the equivalence between the fixed-point formulation \eqref{fixed-point-formulation} and \eqref{dis-problem}, together with Lemma~\ref{th:disc-well-posed} and the small $h$ assumption.
\end{proof}

\section{Computational results}\label{sec:results}
This section presents numerical results illustrating the properties of the proposed discrete scheme (cf. Section~\ref{sec:vem}). We show the optimal behaviour of the method under different polyhedral meshes. Finally, we simulate a simple application-oriented problem.  

The total computable error is defined as usual in the \ac{vem} framework using the local polynomial approximation of the discrete solutions as follows
\begin{align*}
    \bar{\mathrm{e}}_{\vec{\bx}^*} &:= \bar{\text{e}}_{\bu^*}+\bar{\text{e}}_{p^*}+\bar{\text{e}}_{w^*}+\bar{\text{e}}_{\varphi^*}:= \|\bu-\bPi_{k}^{\nabla,K} \bu_h\|_{1,\Omega} + \|p -  p_h\|_{0,\Omega} + \|\varphi - \Pi_{k-1}^{0,F} \varphi\|_{1,\Sigma} + \|w - \Pi_{k}^{\nabla^2,F} w\|_{2,\Sigma}.
\end{align*}
In addition, the experimental order of convergence $r(\cdot)$ applied to either error $\bar{\textnormal{e}}$ of the refinement $1\leq j$ are computed from the formula $r(\bar{\textnormal{e}})^{j+1} = \log\left(\bar{\textnormal{e}}^{j+1}/\bar{\textnormal{e}}^{j}\right)/\log\left(h^{j+1}/h^{j}\right)$, where the $h^j$ denotes the mesh size on either bulk or plate, depending on the context. The fixed-point algorithm is set with a tolerance of $10^{-10}$ applied to the $\ell^2$-norm of the increments, defined as the difference between the \acp{dof} at the iteration $i$ and $i-1$ of the fixed-point algorithm. In turn, the stabilisation term $S_1^{\nabla,E}(\bu_h,\bv_h)$ follows the ``diagonal recipe" introduced in \cite{dassi2017stab}, while for the remaining terms we simply use the well-known \texttt{DOFI-DOFI} stabilisation.

\subsection{Example 1 (convergence rates under uniform mesh refinement)}
We consider the four different discretisations as depicted in Figure~\ref{fig:meshes} of the domain $\Omega = (0,1)^3$ with the sub-boundaries defined by the sets $\Gamma_{\bu} = \left\{ (x_1,x_2,x_3)\in \partial \Omega : x_3\leq 1/2\right\}$, $\Gamma_{\bsigma} = \left\{ (x_1,x_2,x_3)\in \partial \Omega : 1/2 < x_3 < 1\right\}$, and $\Sigma = \left\{ (x_1,x_2,x_3)\in \partial \Omega : x_3 = 1\right\}$. Note that, $\bn_\Sigma = (0,0,1)$. We set unity model parameters  and define the manufactured solutions by
\begin{gather*}
  \bu(x_1,x_2,x_3) = \left( \cos(x_3)\sin(x_2), \cos(x_1)\sin(x_3), \cos(x_2)\sin(x_1)\right), \quad p(x_1,x_2,x_3) = \sin(2\pi x_1)\sin(2\pi x_2),\\
  w(x_1,x_2,x_3) = \bu(x_1,x_2,x_3)\cdot\bn_\Sigma, \quad \varphi(x_1,x_2,x_3) = -\bsigma(x_1,x_2,x_3)\bn_\Sigma\cdot \bn_\Sigma.
\end{gather*}
Note that the right-hand sides $\mathbf{f},g,m$ are sufficiently smooth, as they are derived from the prescribed manufactured solutions. Moreover, the non-homogeneous boundary conditions require a minor adjustment to the linear functionals. Such a modification does not affect the analysis presented in this paper.
\begin{figure}[!t]
    \centering
    \subfigure[Cube. \label{fig:cube}]{\includegraphics[width=0.244\textwidth,trim={7.cm 1.85cm 8.05cm 2.15cm},clip]{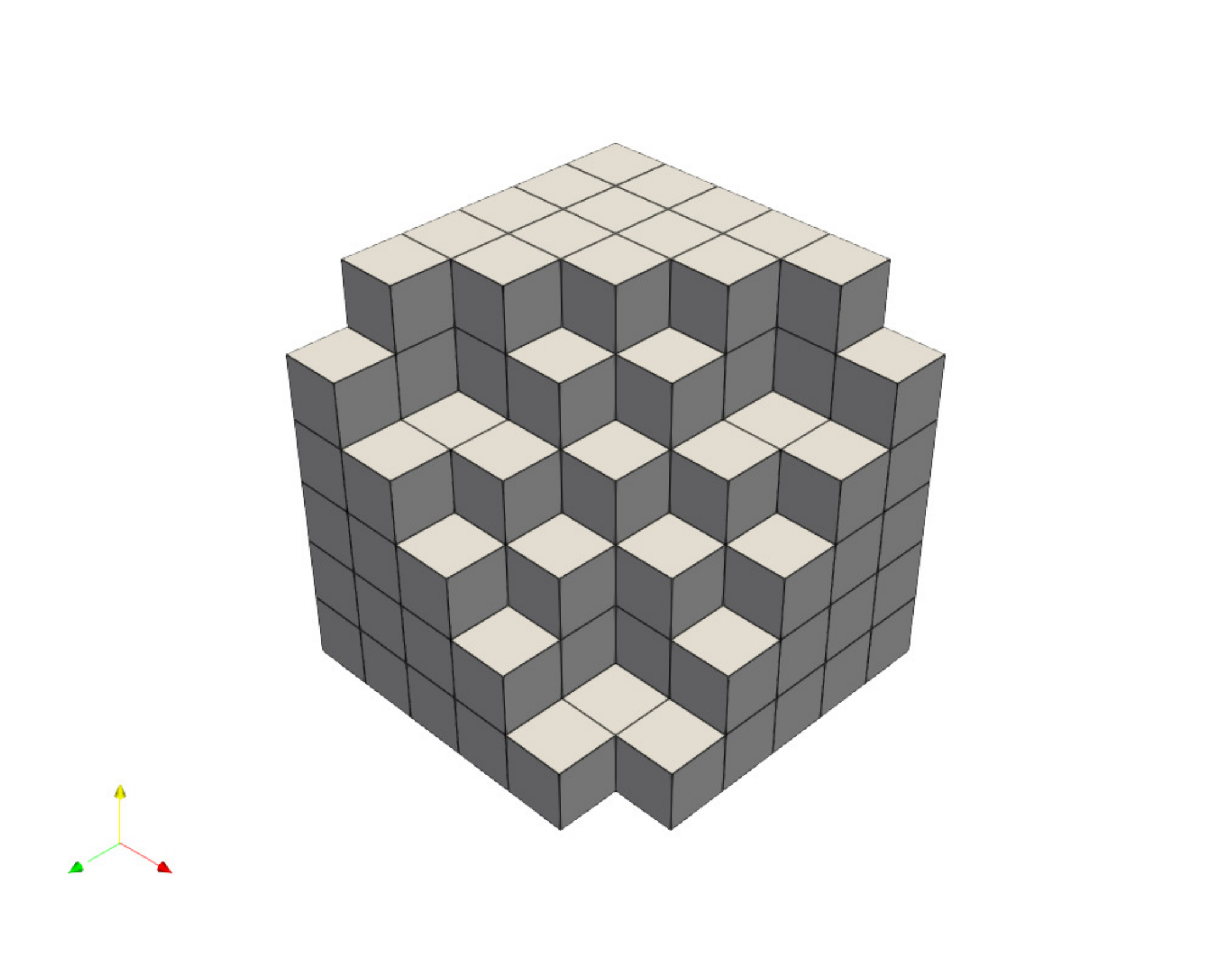}}  
    \subfigure[Octa. \label{fig:octa}]{\includegraphics[width=0.244\textwidth,trim={7.cm 1.85cm 8.05cm 2.15cm},clip]{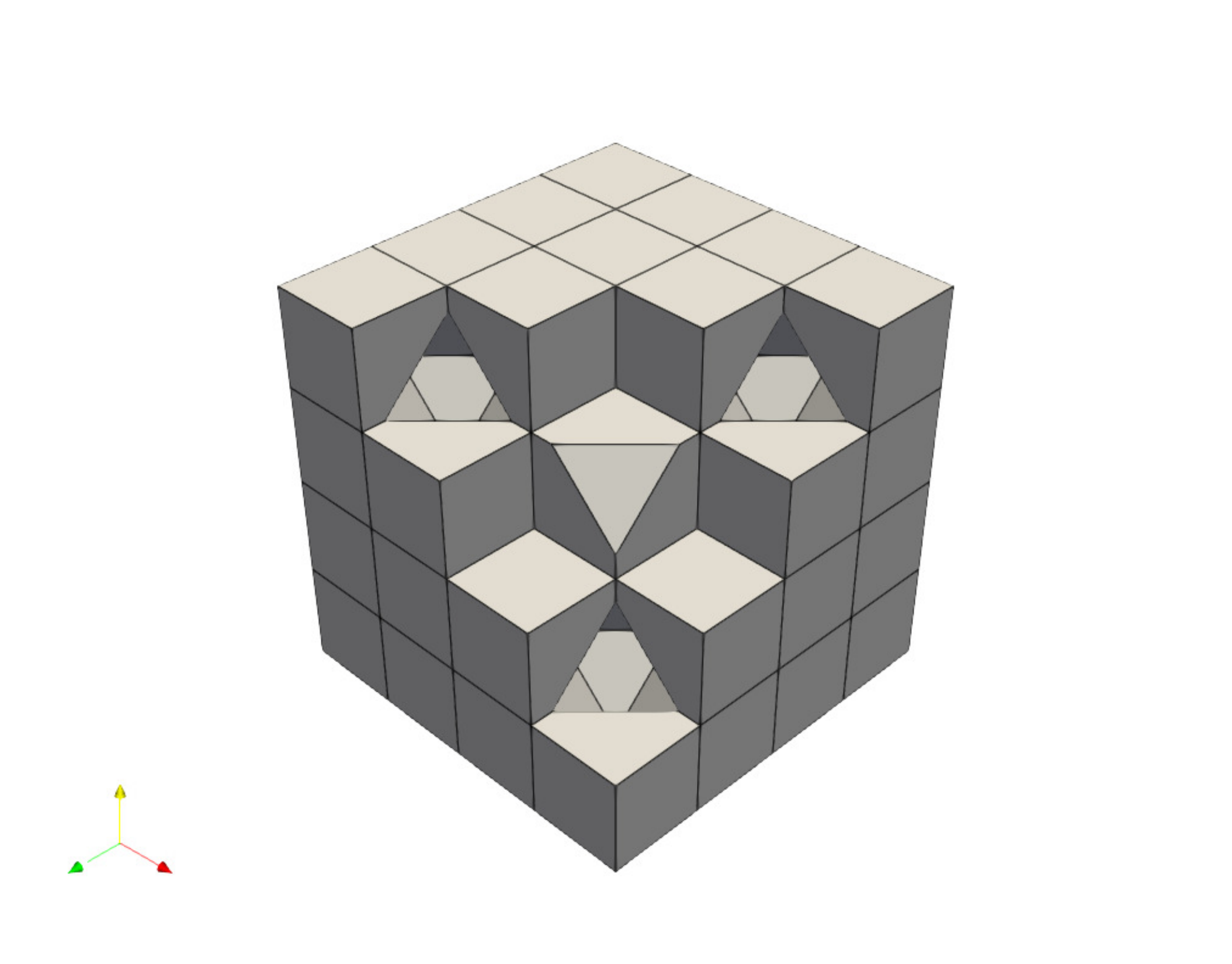}} 
    \subfigure[Voro. \label{fig:voro}]{\includegraphics[width=0.244\textwidth,trim={7.cm 1.85cm 8.05cm 2.15cm},clip]{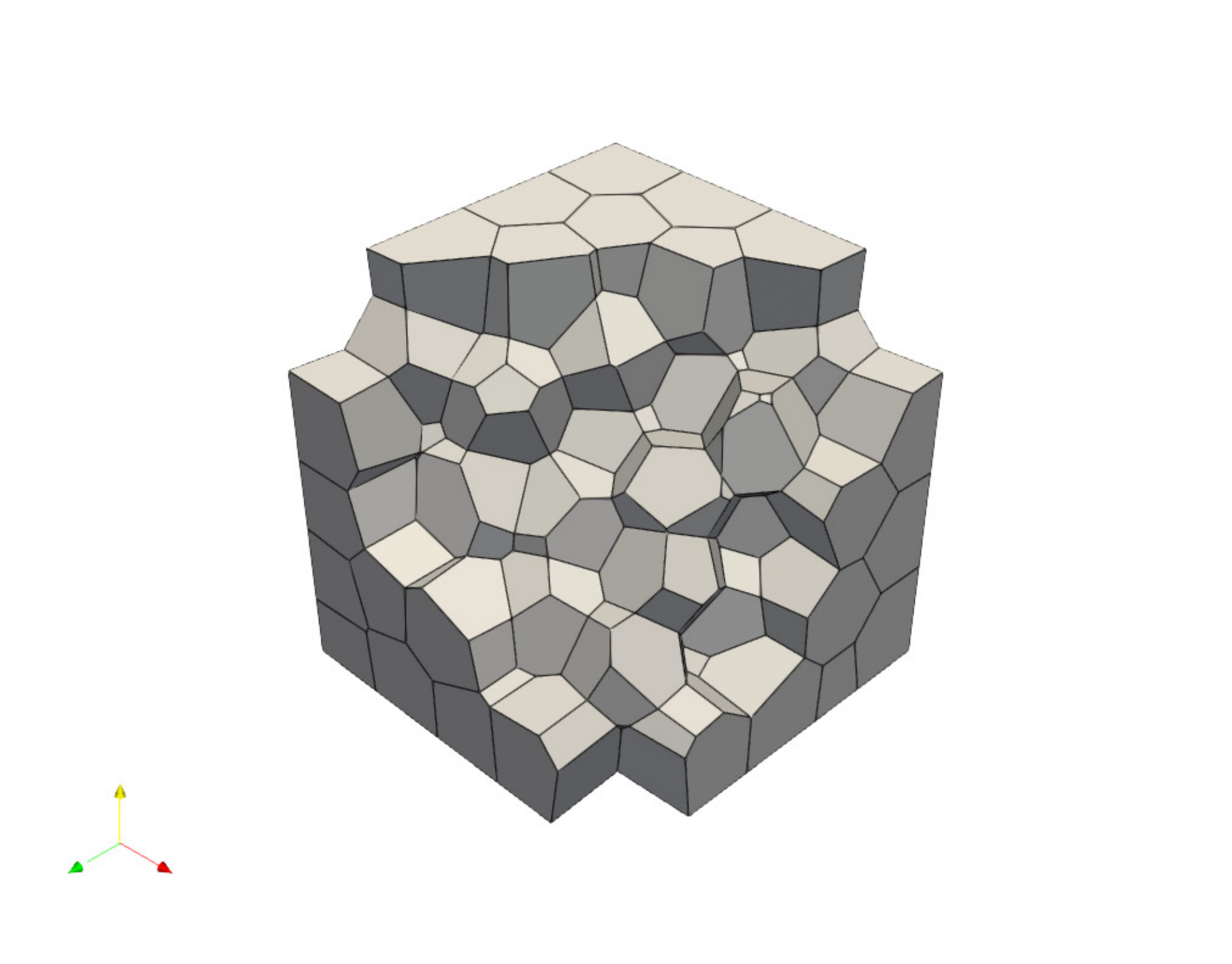}}
    \subfigure[Nine. \label{fig:nove}]{\includegraphics[width=0.244\textwidth,trim={7.cm 1.85cm 8.05cm 2.15cm},clip]{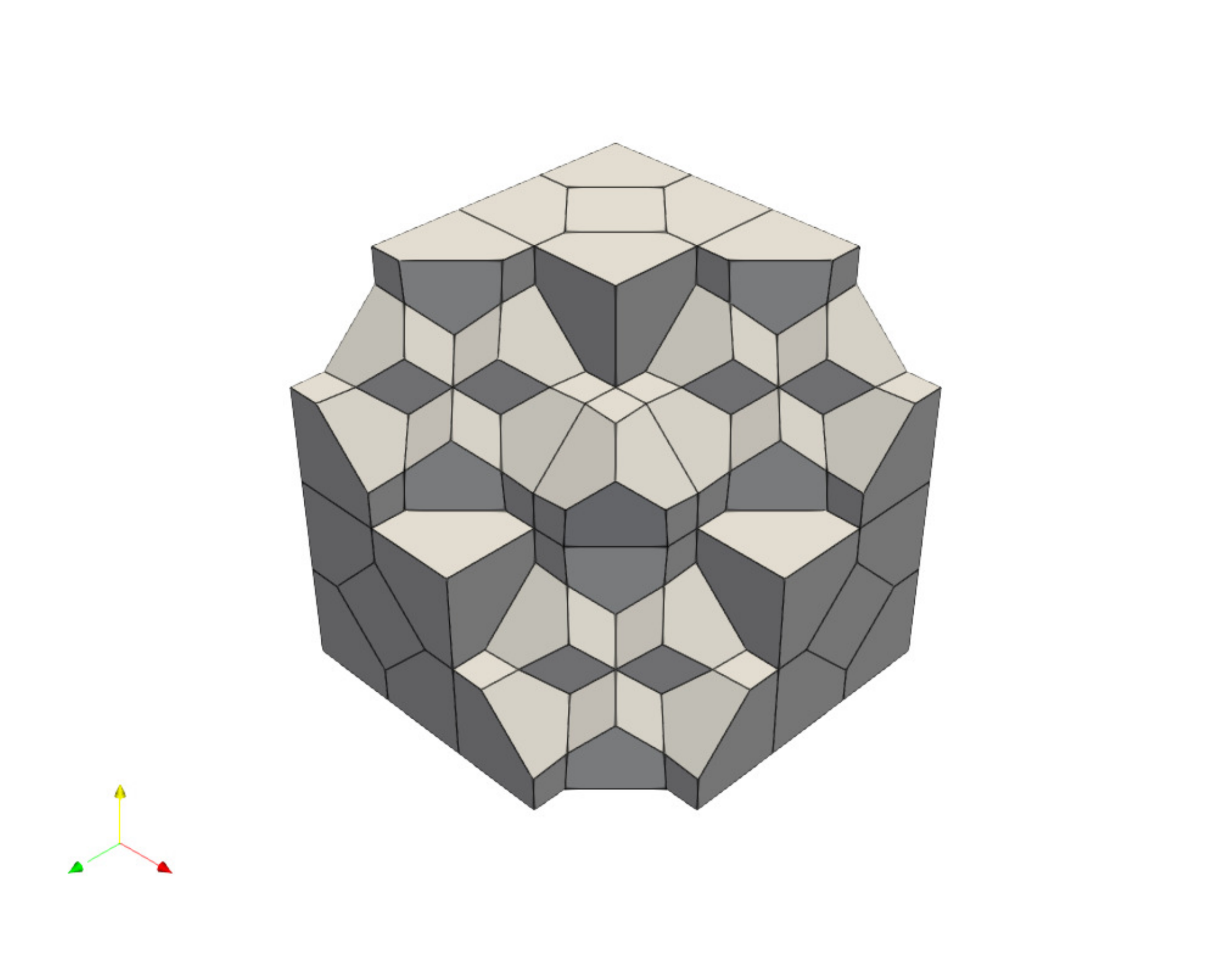}}
    \caption{Example 1. Cross-section of a variety of meshes used in the accuracy verification  test.}\label{fig:meshes}
\end{figure}
\begin{figure}[!t]
    \centering
    \subfigure[Streamlines of $\bPi_{k}^{\nabla}\bu_h$.]{\includegraphics[width=0.49\textwidth,trim={5.8cm 0.cm 2.cm 2.cm},clip]{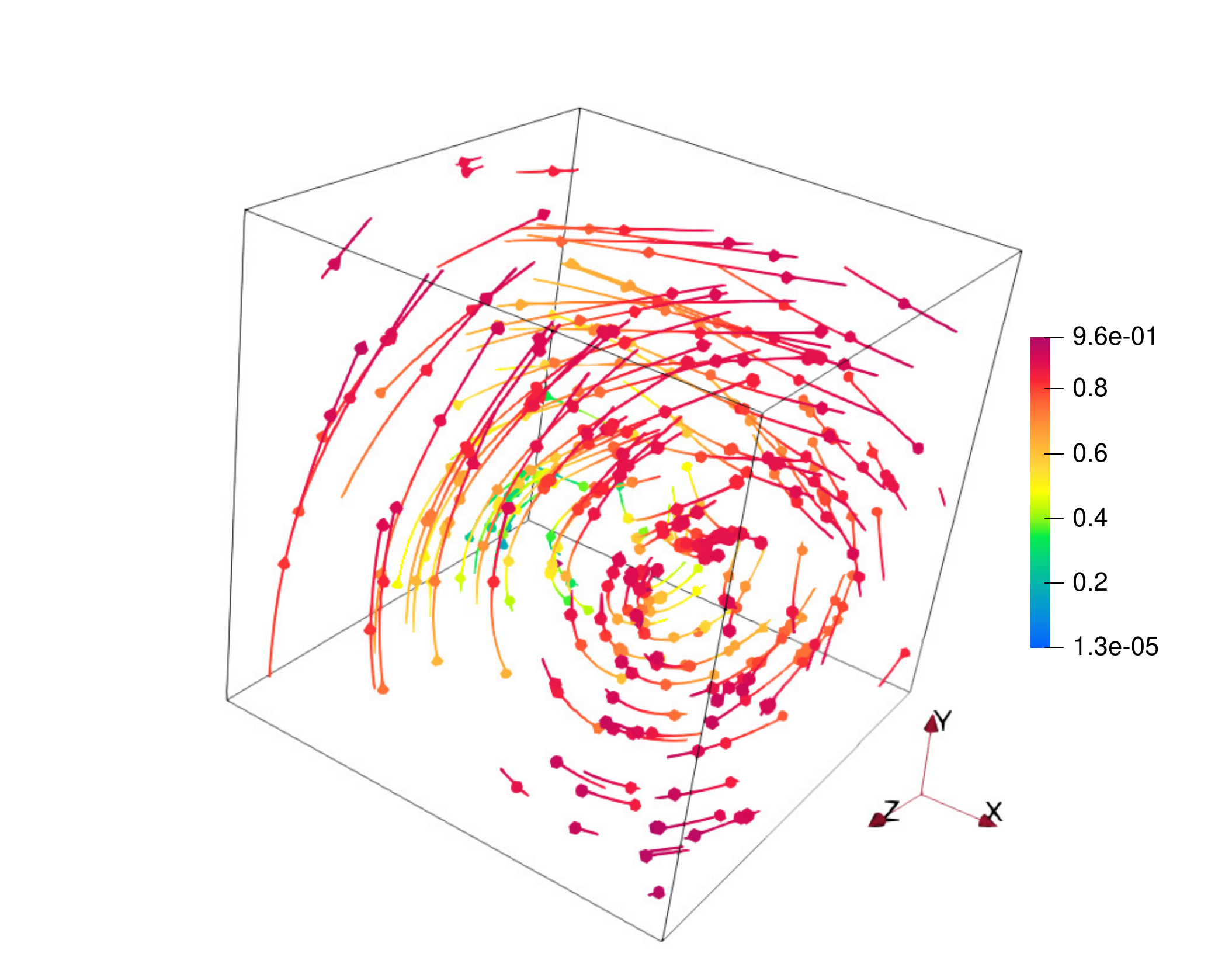}}  
    \subfigure[ $p_h$.]{\includegraphics[width=0.49\textwidth,trim={5.8cm 0.cm 2.cm 2.cm},clip]{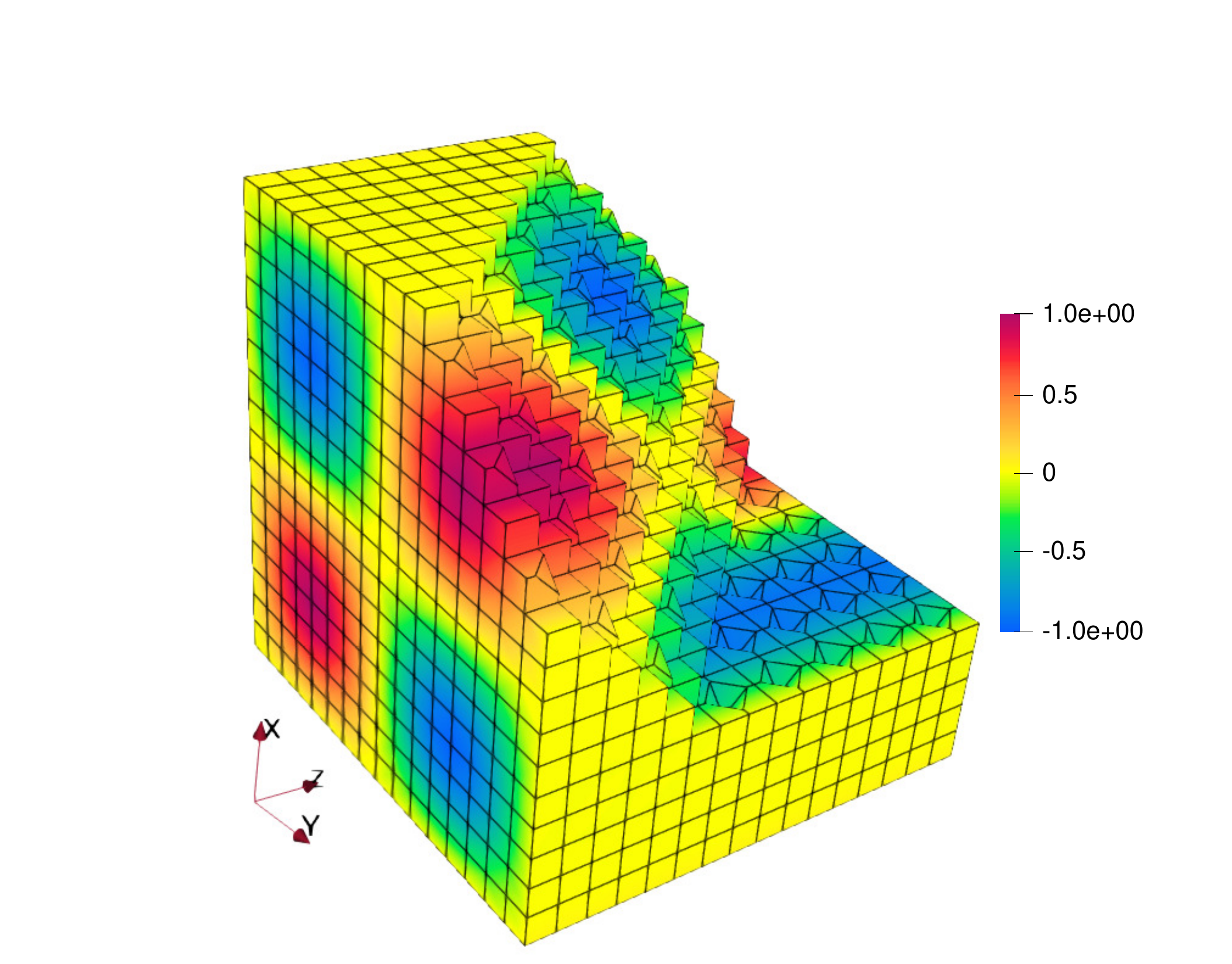}} 
    \subfigure[$\Pi_{k}^{\nabla^2}w_h$.]{\includegraphics[width=0.49\textwidth,trim={6.5cm 3.cm 1.5cm 3.5cm},clip]{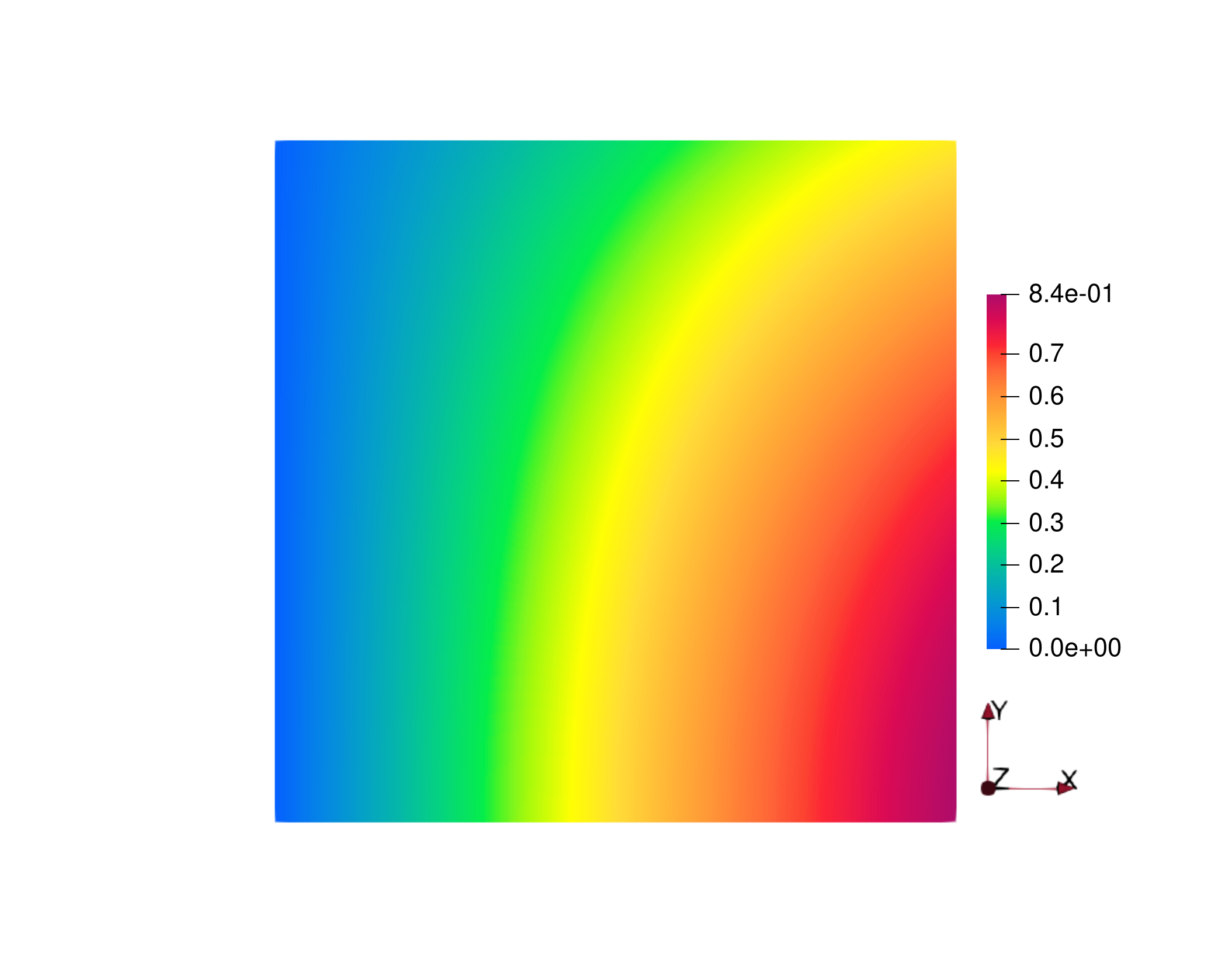}}
    \subfigure[$\Pi_{k-1}^{\nabla}\varphi_h$.]{\includegraphics[width=0.49\textwidth,trim={6.5cm 3.cm 1.5cm 3.5cm},clip]{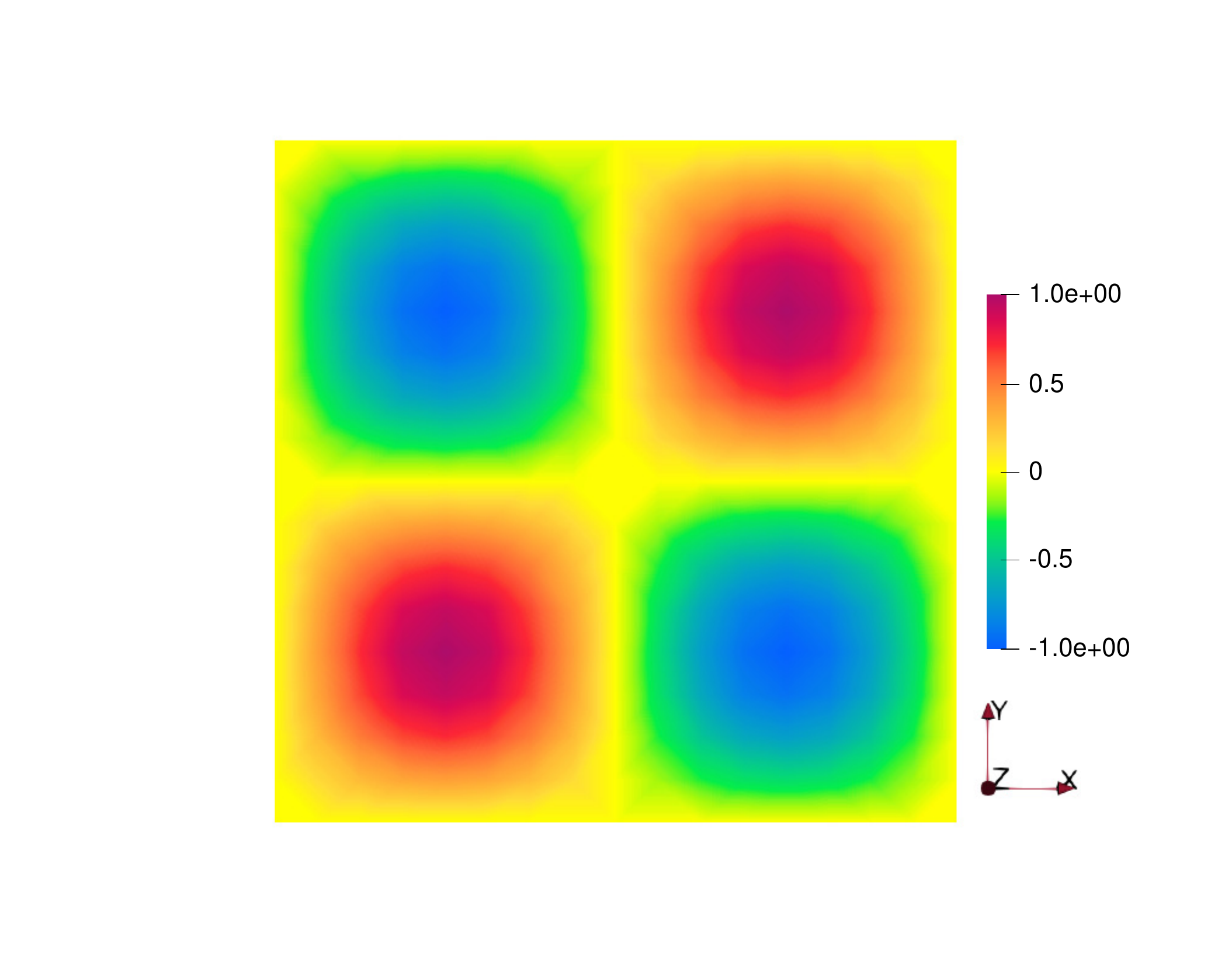}}
    \caption{Example 1. Snapshots of the variables of interest for the Octa mesh in the last refinement step.}\label{fig:manufacturedSols}
\end{figure}

\begin{table}[t!]
\setlength{\tabcolsep}{2pt}
\begin{center}
\begin{tabular}{| c | c | c | c | c | c | c | c | c | c | c | c | c | c |}
\hline
{$\cT_h$} & {$h_{\text{Bulk}}$} & {$h_{\text{Plate}}$} & {$\bar{\mathrm{e}}_{\vec{\bx}^*}$} & {$r(\bar{\mathrm{e}}_{\vec{\bx}^*})$} & {$\bar{\mathrm{e}}_{\bu^*}$} & {$r(\bar{\mathrm{e}}_{\bu^*}$\!)} & {$\bar{\mathrm{e}}_{p^*}$} & {$r(\bar{\mathrm{e}}_{p^*})$} & {$\bar{\mathrm{e}}_{w^*}$} & {$r(\bar{\mathrm{e}}_{w^*})$} & {$\bar{\mathrm{e}}_{\varphi^*}$} & {$r(\bar{\mathrm{e}}_{\varphi^*})$} & {it} \\
\hline 
\hline
\multirow{4}{*}{\rotatebox{90}{Cube}} 
& 8.66e-01 & 5.00e-01 & 1.19e-00 & $\star$    & 6.37e-02 & $\star$    & 6.34e-01 & $\star$   & 2.53e-01 & $\star$    & 9.76e-01 & $\star$   & 4\\
& 4.33e-01 & 2.50e-01 & 6.67e-01 & 0.84 & 3.95e-02 & 0.69 & 2.28e-01 & 1.47 & 1.97e-01 & 0.36 & 5.93e-01 & 0.72 & 4\\
& 2.89e-01 & 1.67e-01 & 4.38e-01 & 1.04 & 1.75e-02 & 2.01 & 1.07e-01 & 1.88 & 1.27e-01 & 1.08 & 4.04e-01 & 0.94 & 4\\
& 2.17e-01 & 1.25e-01 & 3.25e-01 & 1.03 & 9.64e-03 & 2.06 & 6.09e-02 & 1.94 & 8.75e-02 & 1.30 & 3.07e-01 & 0.96 & 4\\
\hline
\hline
\multirow{4}{*}{\rotatebox{90}{Octa}} 
& 8.33e-01 & 5.00e-01 & 1.19e-00 & $\star$   & 6.28e-02 & $\star$   & 6.38e-01 & $\star$    & 2.53e-01 & $\star$    & 9.76e-01 & $\star$    & 4\\
& 4.16e-01 & 2.50e-01 & 6.68e-01 & 0.84 & 3.96e-02 & 0.67 & 2.32e-01 & 1.46 & 1.97e-01 & 0.36 & 5.93e-01 & 0.72 & 4\\
& 2.08e-01 & 1.25e-01 & 3.24e-01 & 1.04 & 9.65e-03 & 2.04 & 5.73e-02 & 2.02 & 8.75e-02 & 1.17 & 3.07e-01 & 0.95 & 4\\
& 1.04e-01 & 6.25e-02 & 1.60e-01 & 1.02 & 2.34e-03 & 2.05 & 1.45e-02 & 1.98 & 3.58e-02 & 1.29 & 1.56e-01 & 0.98 & 4\\
\hline
\hline
\multirow{4}{*}{\rotatebox{90}{Voro}} 
& 5.77e-01 & 3.33e-01 & 8.39e-01 & $\star$   & 4.64e-02 & $\star$   & 3.47e-01 & $\star$   & 2.28e-01 & $\star$   & 7.28e-01 & $\star$   & 4\\
& 3.46e-01 & 2.04e-01 & 5.23e-01 & 0.93 & 2.24e-02 & 1.42 & 1.44e-01 & 1.73 & 1.52e-01 & 0.83 & 4.78e-01 & 0.86 & 5\\
& 1.73e-01 & 1.01e-01 & 2.58e-01 & 1.02 & 5.13e-03 & 2.13 & 3.57e-02 & 2.01 & 6.16e-02 & 1.28 & 2.48e-01 & 0.93 & 5\\
& 1.09e-01 & 6.41e-02 & 1.64e-01 & 0.97 & 1.99e-03 & 2.04 & 1.41e-02 & 2.00 & 4.01e-02 & 0.95 & 1.59e-01 & 0.99 & 5\\
\hline
\hline
\multirow{4}{*}{\rotatebox{90}{Nine}} 
& 7.19e-01 & 4.47e-01 & 1.17e-00 & $\star$   & 7.53e-02 & $\star$   & 5.34e-01 & $\star$   & 3.53e-01 & $\star$   & 9.76e-01 & $\star$  & 3\\
& 3.59e-01 & 2.24e-01 & 6.93e-01 & 0.75 & 4.00e-02 & 0.91 & 1.84e-01 & 1.54 & 2.26e-01 & 0.64 & 6.27e-01 & 0.64 & 4\\
& 1.80e-01 & 1.12e-01 & 2.98e-01 & 1.22 & 7.78e-03 & 2.36 & 4.22e-02 & 2.12 & 7.53e-02 & 1.59 & 2.85e-01 & 1.14 & 4\\
& 8.98e-02 & 5.59e-02 & 1.48e-01 & 1.00 & 1.89e-03 & 2.03 & 1.06e-02 & 1.98 & 3.25e-02 & 1.20 & 1.44e-01 & 0.97 & 4\\
\hline
\end{tabular}
\end{center}
\vspace{0.25cm}
\caption{Example 1. Convergence history and Picard iteration count for a variety of meshes.}
\label{tab:convergenceCombined}
\end{table}

The error history is reported in Table~\ref{tab:convergenceCombined}, where we observe an asymptotic $O(h^{k-1})$ decay of $\bar{\mathrm{e}}_{\vec{\bx}^*}$ as predicted by Corollary~\ref{th:convergence} for all the proposed meshes listed in Figure~\ref{fig:meshes}. In addition, we provided a detailed account of the computable error for the variables of interest, obtaining their corresponding optimal rates of $O(h^k)$ for $\bu$, $p$, and $O(h^{k-1})$ for $w$, $\varphi$, respectively. The last column shows the number of iterations required by the fixed-point implementation.  Snapshots of the variables of interest (projected to the respective polynomial spaces) are shown in Figure~\ref{fig:manufacturedSols} for the Octa mesh (see  Figure~\ref{fig:octa}) in the last refinement step.

\subsection{Example 2 (immune isolation using encapsulation with \ac{snm})}
The islets of Langerhans, or simply islets, are a cluster of endocrine cells within the pancreas that play a central role in regulating metabolism. Among these, $\beta$-cells are specialised in producing and secreting insulin, a hormone essential for maintaining healthy blood glucose levels in the human body. In Type 1 diabetes (T1D), the body's immune system mistakenly targets and destroys these $\beta$-cells. The body loses its ability to make insulin, causing blood sugar levels to rise uncontrollably, which has long-term consequences such as cardiovascular disease, nerve damage, and kidney damage, to mention a few.

An alternative to exogenous insulin administration is the transplantation of pancreatic islets to restore natural insulin production. However, this approach faces several challenges, including a limited supply of suitable donors and the need for lifelong immunosuppression to prevent rejection. Moreover, the autoimmune nature of T1D compromises the long-term effectiveness of the treatment. 

Islet encapsulation arises as a protective strategy that uses a semi-permeable membrane to shield healthy islets from the host's immune system. This membrane allows the exchange of glucose, insulin, nutrients, and small molecules, enabling the survival and proper function of the transplanted cells while preventing their destruction \cite{Song2016}. We focus on a simplification of the scheme presented in Figure~\ref{fig:illustration_device}. We adapt our model to this application by considering only blood flow coming from the top channel of the device (regarded as the bulk subdomain) and the coupling with an idealised 2D \ac{snm} (regarded as the poroelastic plate). Similar simulations considering the isolation chamber as a full poroelastic medium can be found in \cite{Bukac2024}. 

\begin{figure}[t!]
    \centering
    \includegraphics[width=\textwidth,trim={0.cm 0.cm 0.cm 0.cm},clip]{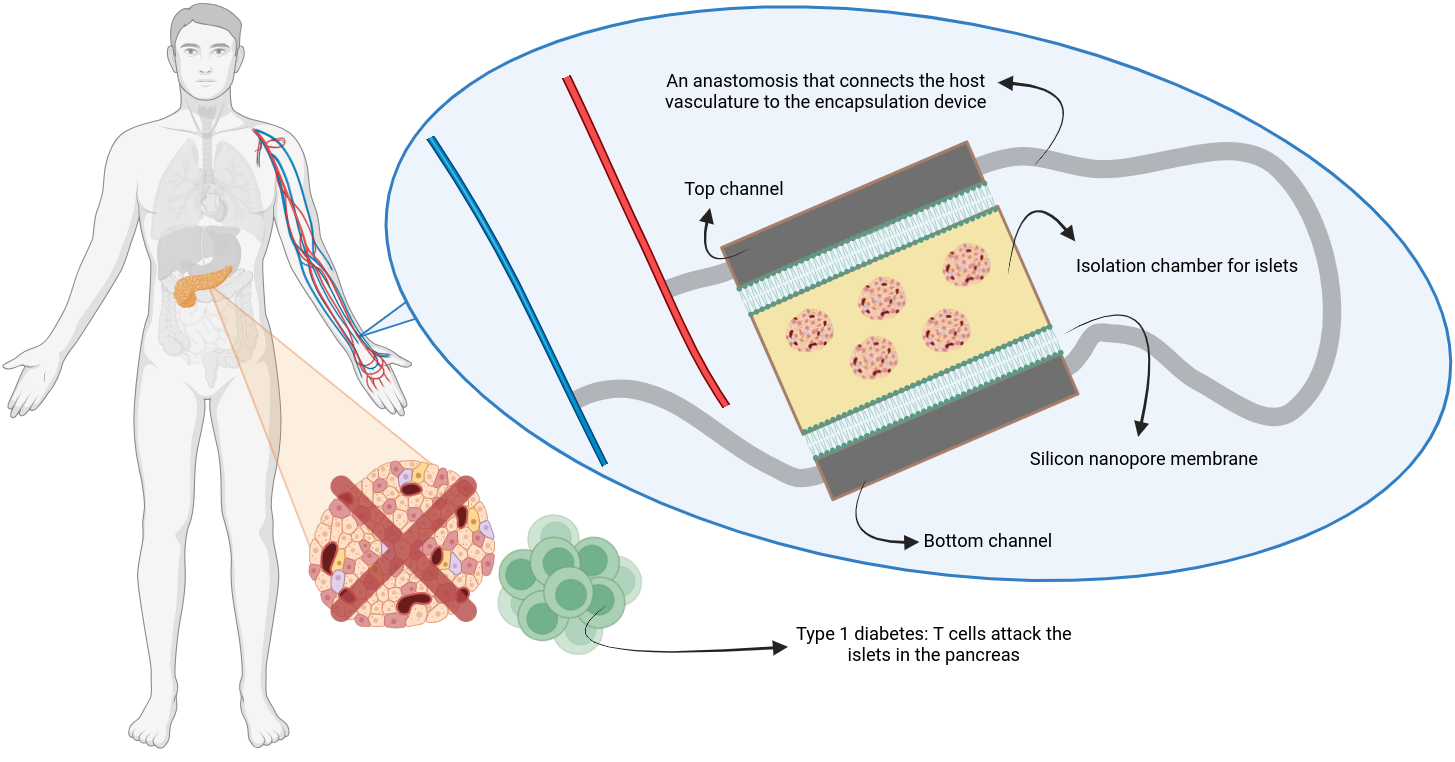}
    \caption{Example 2. Two-dimensional schematic illustration of the immune encapsulation device connected via anastomosis to the host vascular system. The isolation chamber for islets is shown in yellow, encapsulated with the \ac{snm} on both sides with arterial and venous blood in the top and bottom channels, respectively. 
    }\label{fig:illustration_device}
\end{figure}

The \ac{snm} offers a key design advantage in its ability to achieve precise control over extremely small pore sizes, enabling selective molecular transport with high accuracy. Following \cite{Song2016}, the membrane has a size of $1\,\unit{\mm}\times 3\,\unit{\mm}$ composed by $10^6$ pores of width $\num[round-mode=figures, round-precision=1]{7e-6}\, \unit{\mm}$, length $\num[round-mode=figures, round-precision=1]{3e-3}\, \unit{\mm}$, and depth $\num[round-mode=figures, round-precision=1]{3e-4}\, \unit{\mm}$. This composition allowed us to compute the porosity of the membrane as $\phi=\num[round-mode=figures, round-precision=1]{7e-3}$, which we assumed to coincide with the Biot--Willis coefficient $\alpha$. The experimental hydraulic permeability of the membrane is given as $\kappa= \num[round-mode=figures, round-precision=2]{7.5e-2}\, \unit{\mm^3\, \s^2}/\unit{\kg}$ (already scaled by the inverse of blood viscosity $\mu$ and characteristic length of the channel). The remaining coefficients are given by $\rho_p=\num[round-mode=figures, round-precision=3]{7.95e-10}\, \unit{\kg}/\unit{\mm^2}$, $D=\num[round-mode=figures, round-precision=3]{3.78e-4}\, \unit{\mm^2\, \kg}/\unit{\s^2}$, and $C_0=\num[round-mode=figures, round-precision=3]{5.77e-11}\, \unit{\mm\, \s^2}/\unit{\kg}$. On the other hand, typical values for the blood channel (with depth of $1\, \unit{\mm}$) are consider: $\mu=\num[round-mode=figures, round-precision=2]{3.5e-6}\, \unit{\kg}/\unit{\mm\, \s}$, $\rho_f=\num[round-mode=figures, round-precision=3]{1.05e-6}\, \unit{\kg}/\unit{\mm^3}$, and $\gamma=\num[round-mode=figures, round-precision=2]{1.1e-1}\, \unit{\kg}/\unit{\mm^2\, \s^{3/2}}$. The anastomosis (see Figure~\ref{fig:illustration_device}) is done in such a way that the change of pressure in the blood channel preserves the natural one from the body, given by $\Delta p = 13.79\, \unit{\kg}/\unit{\mm\, \s^2}$. 

The computational domain is presented in Figure~\ref{fig:application} (top left panel), with a discretisation given by $15\times15\times45$ cubic elements. The fluid has zero velocity on $\Gamma^{\bu}_{\cero}$ and the change of pressure is added in the model through traction conditions on $\Gamma^{\bsigma}_{\text{in}}$ and $\Gamma^{\bsigma}_{\text{out}}$. The plate is clamped and does not allow flux to escape from it ($w = \nabla_\Sigma w \cdot \bn_{\partial\Sigma} = 0$, and $\nabla \varphi \cdot \bn_{\partial\Sigma} = 0$). Note that this configuration changes the plate pressure boundary condition, which now considers a full Neumann type, necessitating the use of Lagrange multipliers in the formulation to impose uniqueness. Finally, the time step and tolerance for the fixed-point iteration are given by $\tau = \num[round-mode=figures, round-precision=1]{1.e-8}$ and $\text{tol}=\num[round-mode=figures, round-precision=1]{1.e-5}$, respectively. The experiment indicates that the parameters and the mesh should be rescaled to permit a larger time step.

Snapshots of the simulation are shown in the remaining panels of Figure~\ref{fig:application} after $9$ fixed-point iterations. We observe that the blood velocity has a maximum value of $\num[round-mode=figures, round-precision=2]{3.4e-1}\, \unit{\mm}/\unit{\s}$ matching with typical values of the human body. In addition, the blood flow follows the expected direction, driven by the change of pressure. Finally, we observe the small deflections of order $\num[round-mode=figures, round-precision=1]{9e-10}\, \unit{\mm}$ with their corresponding directions exactly in the zones where the fluid goes in and out of the membrane.

\begin{figure}[!t]
    \centering
    \raisebox{-0.3cm}{\includegraphics[width=.495\textwidth]{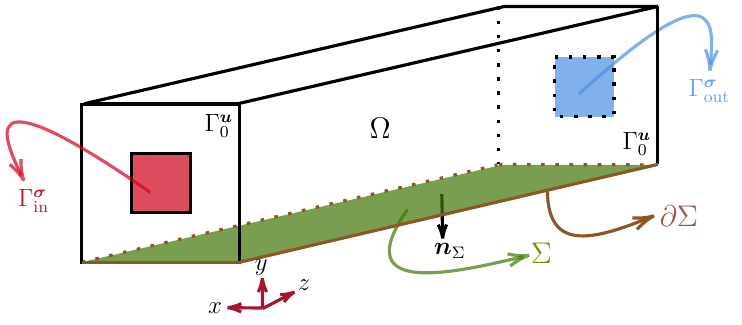}}\includegraphics[width=.495\textwidth,trim={.1cm .1cm .1cm .1cm},clip]{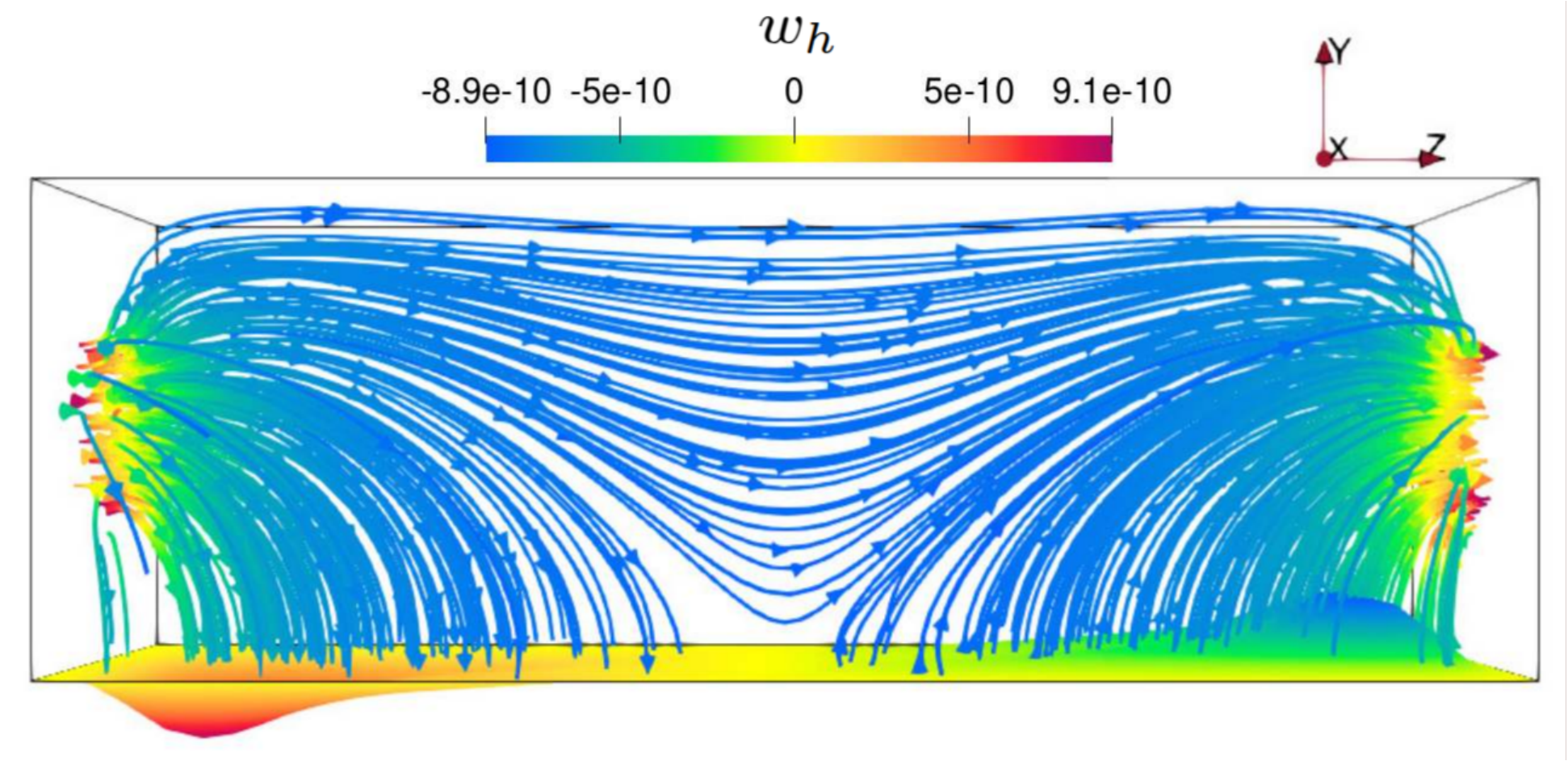}\\
\includegraphics[width=.495\textwidth]{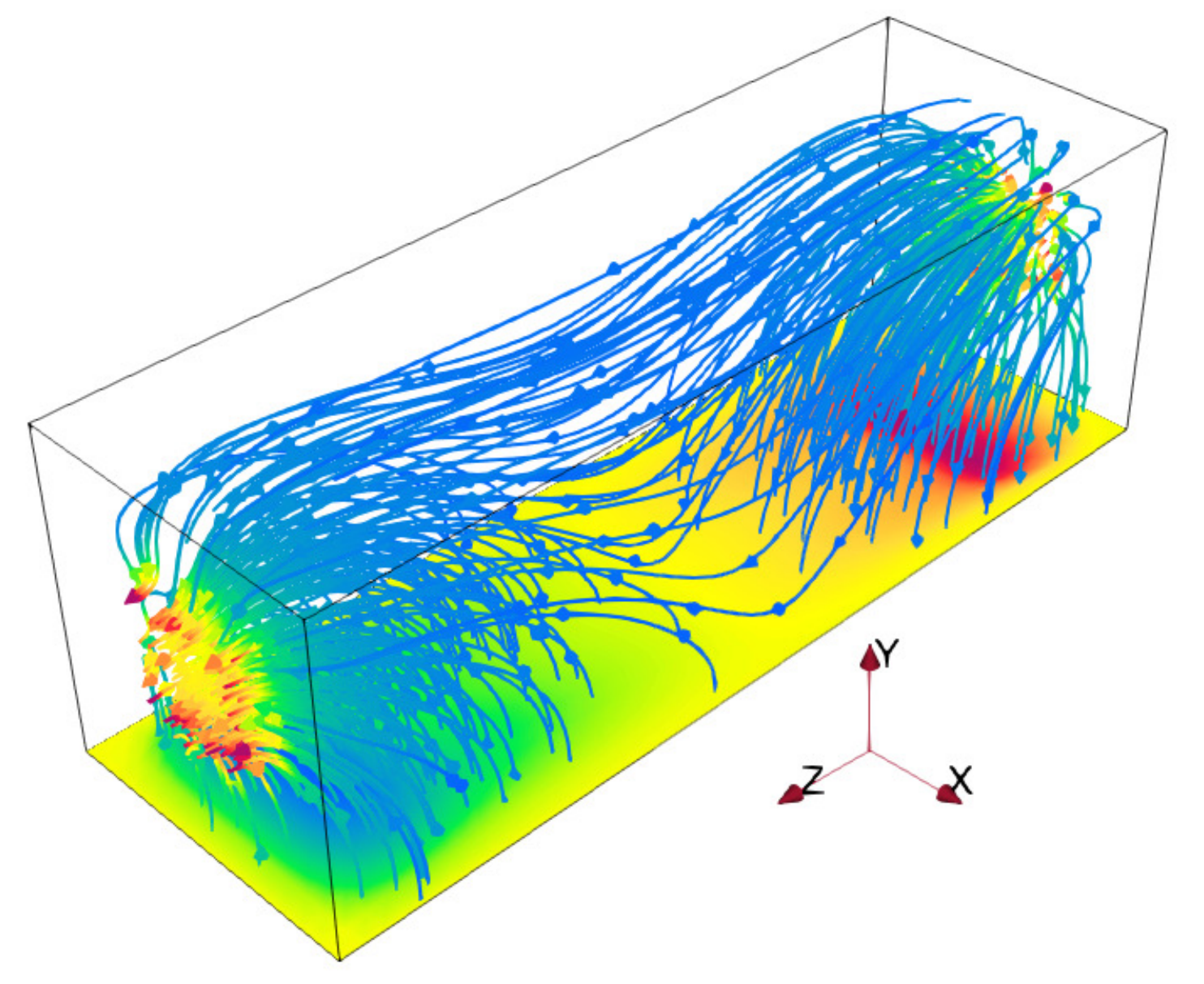}    \raisebox{0.5cm}{\includegraphics[width=.495\textwidth,trim={.1cm .1cm .1cm .1cm},clip]{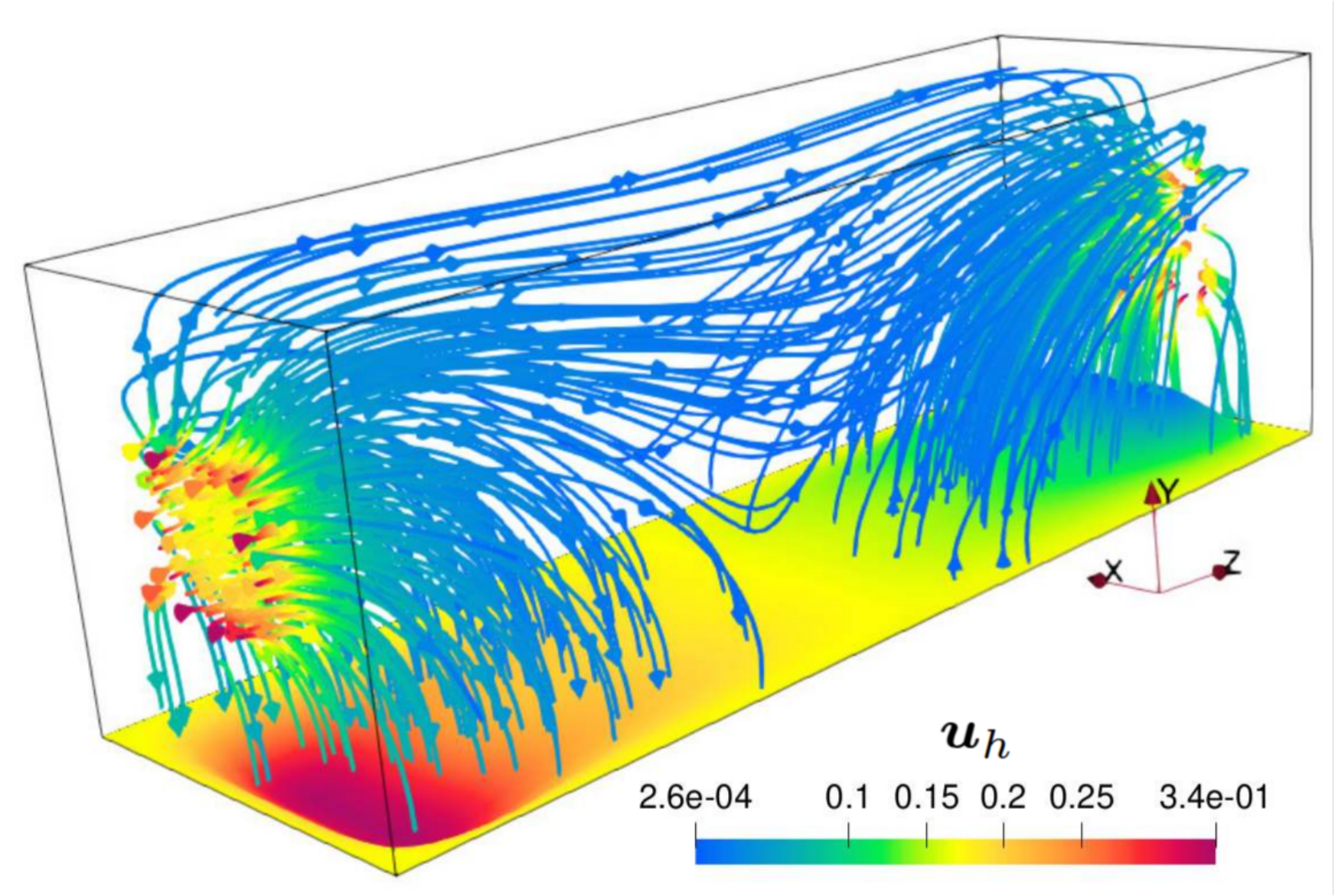}}
    \caption{Example 2. Sketch of the computational domain corresponding to the top channel $\Omega$ and silicon nanopore membrane $\Sigma$ (top left), and snapshots of the fluid velocity streamlines in the top channel, highlighted in black. The deflected plate is shown in the bottom part of the channel. The variables $\bu_h$ and $w_h$ correspond to the associated polynomial projections.}\label{fig:application}
\end{figure}

\subsection*{Acknowledgments} We kindly thank the stimulating discussions with Miguel Fern\'andez and Gabriel N. Gatica, regarding the initial stages of this work. We also thank the mathematical research institute MATRIX in Australia where part of this research was performed. 

\subsection*{Funding} FD was partially supported by the European Union (ERC Synergy, NEMESIS, project number 101115663), and he is also a member of the INdAM-GNCS group. AER and RRB were partially supported by the Australian Research Council through the \textsc{Future Fellowship} grant FT220100496.

\bibliographystyle{siam} 
\bibliography{dkrr_bib}
\end{document}